\documentclass[10pt,reqno]{amsart}
\usepackage[T1]{fontenc}
\usepackage[utf8]{inputenc}
\usepackage[english]{babel}
\usepackage{amsmath}
\usepackage{amsfonts}
\usepackage{amssymb}
\usepackage{amsthm}
\usepackage{bm}
\usepackage{bbm}
\usepackage{braket}
\usepackage{graphicx}
\usepackage{subcaption}

\usepackage{marginnote}
\usepackage[margin=1.2in]{geometry}
\usepackage{float}

\usepackage{color}
\usepackage[colorlinks=true, pdfstartview=FitV, linkcolor=darkblue, citecolor=darkred, urlcolor=cyan]{hyperref}
\definecolor{darkred}{RGB}{203,65,84}
\definecolor{darkblue}{RGB}{70,130,180}
\definecolor{brown}{RGB}{139,69,19}

\usepackage{dsfont} 
\usepackage{pxfonts}
\usepackage{microtype}

\usepackage{csquotes}

\newtheorem{theorem}{Theorem}[subsection]
\newtheorem{lemma}[theorem]{Lemma}
\newtheorem{proposition}[theorem]{Proposition}
\newtheorem{corollary}[theorem]{Corollary}

\newtheorem{definition}[theorem]{Definition}

\newcounter{foo}
\newtheorem{theo}[foo]{Theorem}

\newcommand{\rmks}{\vskip 0.1 truecm\noindent{\sc Remarks: }}
\newcommand{\rmka}{\vskip 0.1 truecm\noindent{\sc Remark: }}

\newcommand{\hh}{{\mathbf H^2}}
\newcommand{\uh}{{\mathsf U}\hh}
\newcommand{\USi}{{\mathsf U}\Sigma}
\newcommand{\bary}{\operatorname{Bary}}
\newcommand{\bh}{\partial_\infty \hh}
\newcommand{\diam}{\operatorname{diam}}

\newcommand{\G}{\ms G}

\newcommand{\psld}{\mathsf{PSL}_2(\mathbb R)}
\newcommand{\pslt}{\mathsf{PSL}_3(\mathbb R)}

\renewcommand{\angle}{\sphericalangle}

\renewcommand{\hom}{\operatorname{Hom}}
\newcommand{\End}{\operatorname{End}}
\newcommand{\tr}{\operatorname{Tr}}

\newcommand{\Id}{\operatorname{Id}}

\renewcommand{\d}{{\rm d}}
\newcommand{\D}{{\mathrm D}}

\newcommand{\pt}{{\partial_t}}
\newcommand{\ps}{{\partial_s}}

\newcommand{\sbt}{\,\begin{picture}(-1,1)(-1,-1)\circle*{2}\end{picture}\ }
\renewcommand{\dot}[1]{\overset{\sbt}{#1}}

\newcommand{\ms}{\mathsf}
\newcommand{\mk}{\mathfrak}
\newcommand{\pp}{\mathrm p}
\newcommand{\qq}{\mathrm q}
\newcommand{\PP}{\mathrm P}
\newcommand{\QQ}{\mathrm Q}
\renewcommand{\aa}{\mathrm a}
\renewcommand{\AA}{\mathbf a}

\newcommand{\uu}{\mathbf u}
\newcommand{\vv}{\mathbf v}

\newcommand{\Ell}{\mathsf{L}}

\newcommand{\T}{\ms T}

\newcommand{\defeq}{\coloneqq}
\newcommand{\eqdef}{\eqqcolon}
\renewcommand{\leq}{\leqslant}
\renewcommand{\geq}{\geqslant}

\renewcommand{\epsilon}{\varepsilon}
\renewcommand{\phi}{\varphi}

\newcommand{\GG}{{\mathcal C}}
\newcommand{\GT}{{\mathcal G}}
\newcommand{\casper}{\boldsymbol{1}}

\newcommand{\flo}[1]{\left(#1_t\right)_{t\in\mathbb R}}
\renewcommand{\fam}[1]
{
\left(#1_t\right)_{t\in ]-\epsilon,\epsilon[}
}
\newcommand{\famD}[1]
{
\left(#1_u\right)_{u\in\mathrm D}
}
\newcommand{\seq}[1]{ \{{#1}_m\}_{m\in\mathbb N}}
\newcommand{\mapping}[4] { \left\{
    \begin{array}{rcl}
      #1 &\rightarrow& #2\\
      #3 &\mapsto& #4
    \end{array}
  \right.  }

\newcommand{\bS}{\bm{S}}
\newcommand{\bI}{\mathsf{I}}
\newcommand{\J}{\bm{J}}

  \newcommand{\Gk}{\operatorname{Gr}_k}
    \newcommand{\Gnk}{\operatorname{Gr}_{n-k}}
  \newcommand{\bgamma}{{\boldsymbol{\gamma}}}

\newcommand{\bzeta}{{\eta}}
\newcommand{\bOm}{\bold{\Omega}}
\newcommand{\bLa}{\bold{\Lambda}}
\newcommand{\bXi}{\bold{\Xi}}

\usepackage{blindtext}

\makeatletter
\renewcommand\part{%
   \if@noskipsec \leavevmode \fi
   \par
   \addvspace{4ex}%
   \@afterindentfalse
   \secdef\@part\@spart}

\def\@part[#1]#2{%
    \ifnum \c@secnumdepth >\m@ne
      \refstepcounter{part}%
      \addcontentsline{toc}{part}{\thepart\hspace{1em}#1}%
    \else
      \addcontentsline{toc}{part}{#1}%
    \fi
    {\vskip 2cm\parindent \z@ \raggedright
     \interlinepenalty \@M
     \normalfont
     \ifnum \c@secnumdepth >\m@ne
       \Large\bfseries \partname\nobreakspace\thepart
       \par\nobreak
     \fi
     \huge \bfseries #2%
     \par}%
    \nobreak
    \vskip 3cm
    \@afterheading}
\def\@spart#1{%
    {\parindent \z@ \raggedright
     \interlinepenalty \@M
     \normalfont
     \huge \bfseries #1\par}%
     \nobreak
     \vskip 3ex
     \@afterheading}
\makeatother 

\title{Ghost polygons, Poisson bracket and convexity}
\author{Martin Bridgeman \and Fran\c{c}ois Labourie}
\thanks{M.~Bridgeman acknowledges funding by NSF grant DMS-2005498, DMS-2405291 and the Simons Fellowship 675497F.\\ \indent F.~Labourie acknowledges funding by the European Research Council under ERC-Advanced grant 101095722.  }

\begin{document}
\maketitle

\begin{abstract}
The moduli space of Anosov representations of a surface group in a semisimple group -- an open set in the character variety -- admits many more natural functions than the regular functions. We will study in particular length functions and correlation functions. Our main result is a formula that computes the Poisson bracket of those functions using some combinatorial devices called {\em ghost polygons} and {\em ghost bracket} encoded in a formal algebra called the {\em ghost algebra} related in some cases to the swapping algebra introduced by the second author. As a consequence of our main theorem, we show that the set of those functions -- length and correlation -- is stable under the Poisson bracket. We give two applications: firstly in the presence of positivity we prove the convexity of length functions, generalizing a result of Kerckhoff in Teichm\"uller space, secondly we exhibit subalgebras of commuting functions  associated to geodesic laminations. An important tool is the study of {\em uniformly hyperbolic bundles} which is a generalization of Anosov representations beyond periodicity.	
\end{abstract}
\section*{Introduction} The character variety of a discrete group  $\Gamma$ in a Lie group $\mathsf G$ admits a natural class of functions: the algebra of regular functions generated as a polynomial algebra by trace functions or {\em characters}. When $\Gamma$ is a surface group,  the character variety becomes equipped with a symplectic form generalizing the Poincar\'e intersection form -- called the Atiyah--Bott--Goldman symplectic form \cite{Atiyah:1983,  Goldman:1984, Labourie:2013ka} --  and a fundamental theorem of Goldman \cite{Goldman:1986} shows that the algebra of regular functions is stable under the Poisson bracket and more precisely that the bracket of two characters is expressed using a beautiful combinatorial structure on the ring generated by characters. 
The Poisson bracket associated to a surface group has been heavily studied in \cite{Goldman:1986}, \cite{Turaev:1991wk}; and in the context of Hitchin representations the link between the symplectic structure, coordinates  and  cluster algebras discovered by Fock--Goncharov in  \cite{Fock:2006a} (see also  Bonahon--Dreyer \cite{Bonahon:2014wo}), has generated a lot of attention: for instance see  Sun--Wienhard--Zhang \cite{Sun:2020vm}, Nie \cite{Nie:2013tu}, Sun--Zhang \cite{Sun:2017}, Choi--Jung--Kim \cite{Choi:2020aa} and Sun \cite{Sun:2021tj} for more results, and also relations with the swapping algebra \cite{Labourie:2012vka}.

 On the other hand the deformation space of Anosov representations admits many other natural functions besides regular functions. {\em Length functions}, associated to any geodesic current, studied by Bonahon \cite{Bonahon:1988} in the context of Teichm\"uller theory, play a prominent role for Anosov representations for instance in \cite{Bonahon:2014woa} and \cite{Bridgeman:2015ba}. Another class are the {\em correlation functions}, defined in \cite{Labourie:2012vka} and \cite{Bridgeman:2020vg}. 
 
 These functions are defined as follows. For the sake of simplicity, we focus in this introduction on the case of a projective Anosov representation $\rho$ of a hyperbolic group $\Gamma$ in $\mathsf{SL}(V)$. In that case, for every non trival element $g$ in $\Gamma$, as part of the Anosov property, $\rho(g)$ has an attractive fixed point $\xi(g)$ in $\mathbf P(V)$ and an attractive fixed point $\xi^*(g)$ in  $\mathbf P(V^*)$, therefore associating to $g$ the rank 1-projector $\pp_\rho(g)$ whose image is  $\xi(g)$ and kernel is  $\xi^*(g)$. The projector only depends on the endpoints of $g$ in $\partial_\infty \Gamma$. The assignment $g\mapsto \pp_\rho(g)$  can then -- thanks to the Anosov property again -- be extended to any {\em geodesic} $g$ of $\Gamma$, that is, a pair of distinct points in $\partial_\infty\Gamma$. The correlation function $\T_G$ associated to  a {\em configuration of $n$-geodesics} -- that is, an $n$-tuple  $G=(g_1,\ldots,g_n)$ of geodesics up to cyclic transformation -- is then
$$
\T_G:\rho\mapsto \T_G(\rho)\defeq\tr(\pp_\rho(g_n)\ldots\pp_{\rho}(g_1))\ .
$$
In Teichmüller theory, the correlation function of two geodesics is  the cross-ratio of the endpoints. Generally, the correlation functions of geodesics in Teichm\"uller theory is a rational function of cross-ratios. This is no longer the case in the higher rank.

For instance if $C$ is a geodesic triangle  given by the three oriented geodesics $(g_1,g_2,g_3)$, the map 
$$
\T^*_C:\rho\mapsto \T^*_C(\rho)\defeq\tr(\pp_\rho(g_1)\pp_{\rho}(g_2)\pp_{\rho}(g_3))\ ,
$$
is related to Goncharov triple ratio on the real projective plane. 

For a geodesic current $\mu$, its length function $\Ell_\mu$ is defined by an averaging process -- see equation \eqref{eq-def:length}. One can also average correlation functions: say a $\Gamma$-invariant measure $\mu$  on the set $\GG^n$ of generic $n$-tuples of geodesics is  an {\em integrable cyclic current} if it is invariant under cyclic transformations and satisfies some integrability conditions -- see section \ref{sec:current} for precise definitions. Then the {\em $\mu$-correlation function} or $\mu$-averaged correlation function  is
$$
\T_\mu:\rho\mapsto\int_{\GG^n/\Gamma}\T_G(\rho)\ \d\mu\ .
$$
The corresponding functions are analytic (see \cite{Bridgeman:2015ba}), but rarely algebraic.

In the case when $\Gamma$ is a surface group, the algebra of functions on the deformation space of Anosov representations admits a Poisson bracket coming from the Atiyah--Bott--Goldman symplectic form. 

To uniformize our notation, we write $\T^k_\mu$ for $\T_\mu$ when $\mu$ is supported on $\GG^k$ and $\T^1_\nu=\Ell_\nu$ for the length function of a geodesic current $\nu$. Then, one of the main results of this article, Theorem \ref{theo:poiss-bracket},  gives as a corollary 
\begin{theo}[{\sc Poisson stability}]\label{theo:A}
	The space of length functions and   correlation functions is stable under the Poisson bracket. More precisely there exists a Lie bracket on the polynomial algebra formally generated by  tuples of geodesics $(G,H)\mapsto [G,H]$ so that
	$$
	\{\T^k_\mu,\T^p_\nu\}=\int_{\GG^{k+p}/\Gamma}\T_{[G,H]}(\rho)\ \ \d\mu(G)\d\nu(H)\ .
	$$
\end{theo}
The full result, Theorem \ref{theo:poiss-bracket} and its Corollary \ref{coro:Poisson-stab}, allows us to recursively use this formula and indeed obtain stability. 

In Theorem \ref{theo:ham-cor} we compute explicitly what is the Hamiltonian vector field of the correlation functions. For instance in Teichmüller theory, this allows us to compute the higher derivatives of a length function  along twist orbits  by a combinatorial formula involving cross-ratios.

The bracket  $(G,H)\mapsto [G,H]$ -- that we call the {\em ghost bracket} --  is combinatorially constructed.  
In this introduction,  we explain the ghost bracket in the simple projective case and refer to section \ref{sec:ghost-polygon} for more details.  Recall first that an {\em ideal polygon} -- not necessarily embedded -- is a sequence $(h_1,\ldots, h_n)$  of geodesics in $\hh$  such that the endpoint of $h_i$ is the starting point of $h_{i+1}$.  Let then $G$ be the configuration of $n$ geodesics $(g_1,\ldots, g_n)$, with the endpoint of $g_i$ not equal the starting point of $g_{i+1}$. The associated {\em ghost polygon}  is  given by the uniquely defined  configuration $(\theta_1,\ldots \theta_{2n})$ of geodesics  (see figure \eqref{fig:Ghost2}) such that 
\begin{figure}[h]
\begin{subfigure}[h]{0.3\textwidth}   
\begin{center}
 \includegraphics[width=0.8\textwidth]{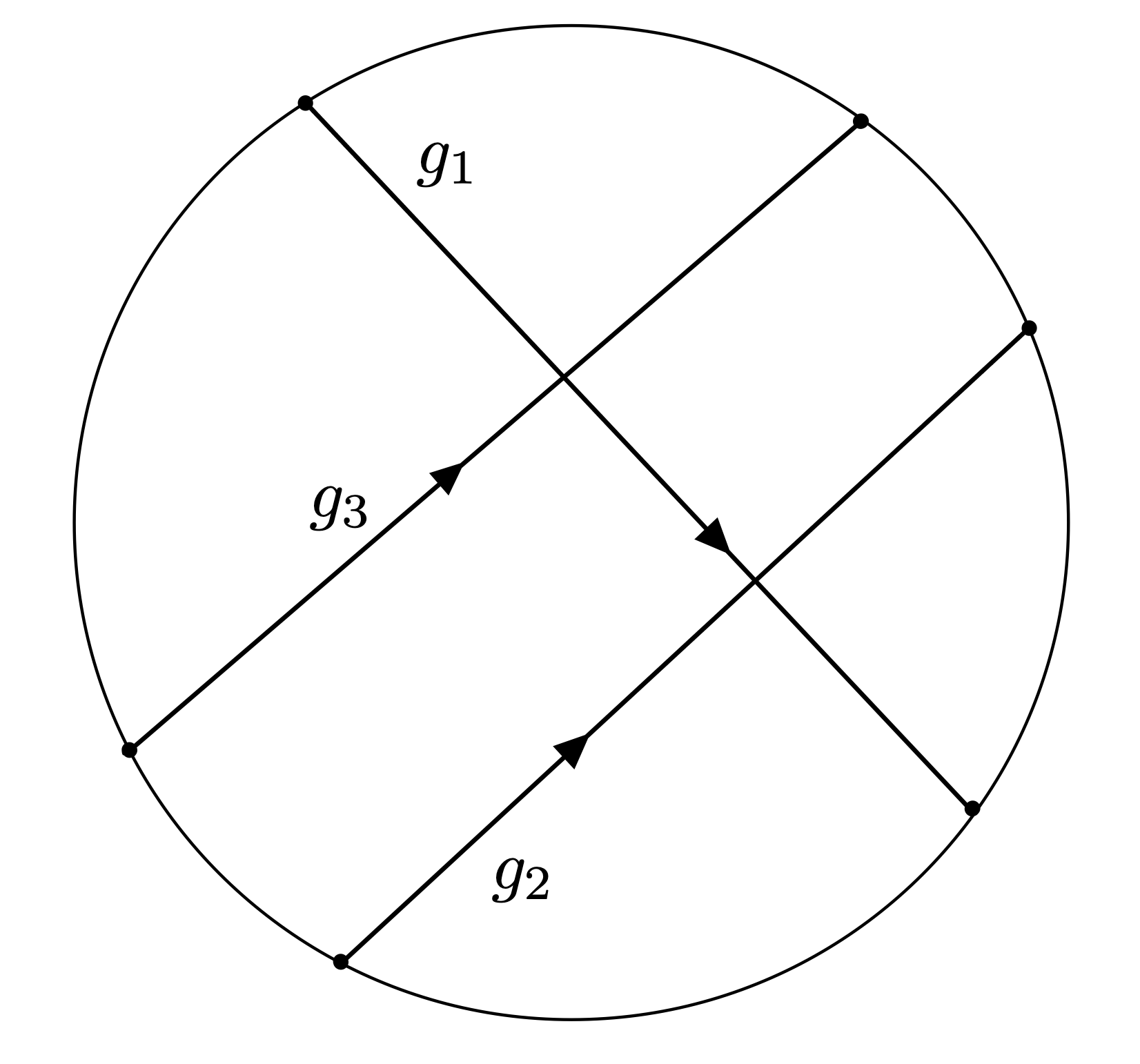}
       \caption{Configuration of geodesics}\label{fig:Config}
  \end{center}
  \end{subfigure}
\quad    
 \begin{subfigure}[h]{0.3\textwidth} \begin{center}
  \includegraphics[width=0.8\textwidth]{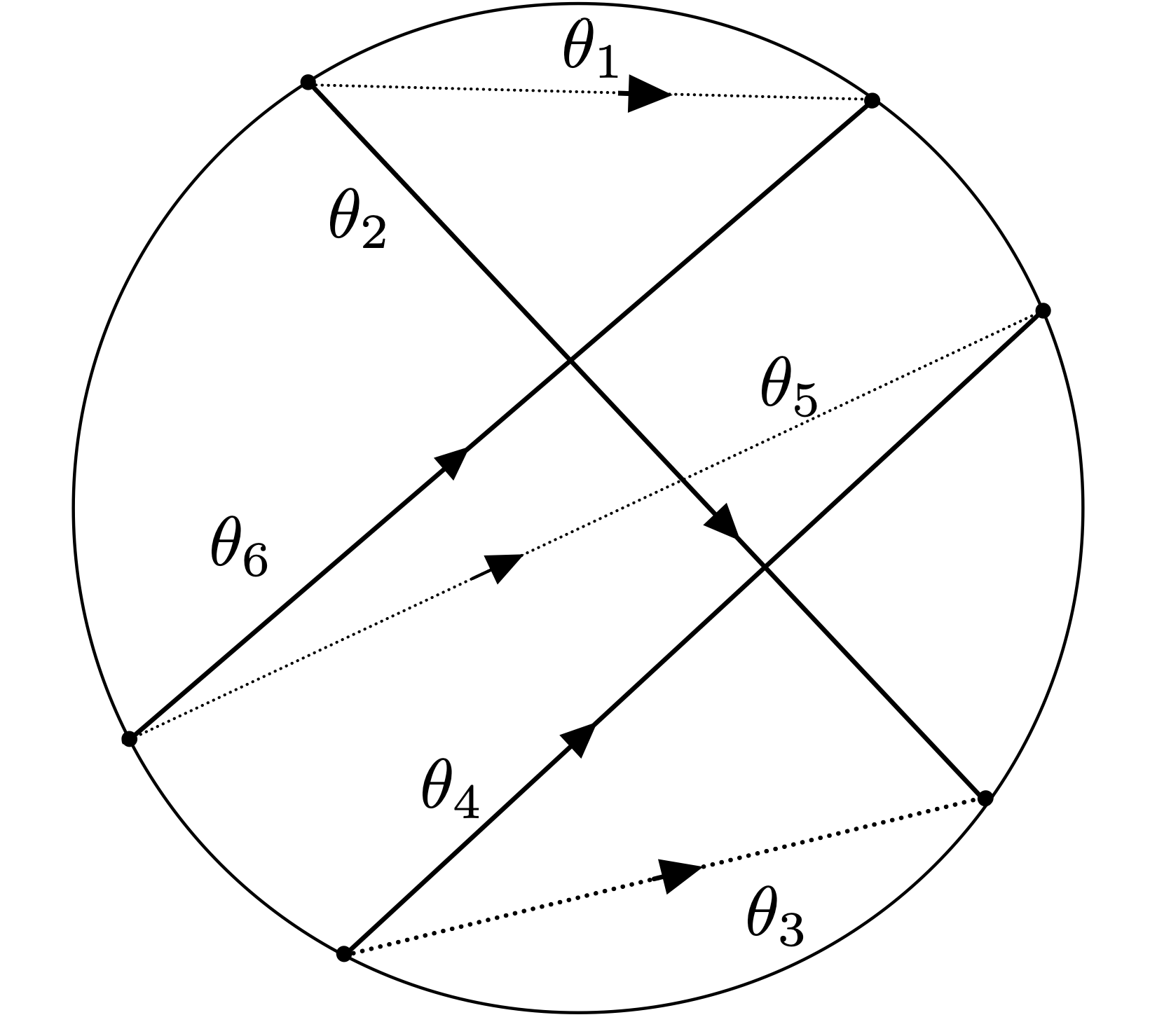}
  \end{center}
    \caption{Ghost Polygon}\label{fig:Gp}
 \end{subfigure}  
 \caption{Two ways to see a cyclically ordered tuple of geodesics}
 \label{fig:Ghost2}
\end{figure}

\begin{enumerate}
 	\item $(\bar\theta_1,\theta_2,\bar\theta_3 \ldots,\bar\theta_{2n-1},\theta_{2n})$ is an ideal polygon,
 	\item for all $i$, $\theta_{2i}=g_i$ and is called a {\em visible edge}, while $\theta_{2i+1}$ is called a  {\em ghost edge}.
 \end{enumerate}
We now denote by $\lceil g,h\rceil$ the configuration of two geodesics $g$ and $h$,   $\epsilon(g,h)$ their algebraic intersection,  and $\bar g$ is the geodesic $g$ with the opposite orientation. 
Then if $(\theta_i,\ldots,\theta_{2n})$ and $(\zeta_i,\ldots,\zeta_{2p})$ are the two ghost polygons associated to the configurations $G$ and $H$, we define the {\em projective ghost bracket} of $G$ and $H$ as \begin{eqnarray}
 	[G,H]\defeq G\cdot H\cdot\left(\sum_{i,j}(-1)^{i+j}\epsilon(\zeta_j,\theta_i)\  \lceil \zeta_j,\theta_i\rceil\right)\ ,\label{eq:def-ghost-bra}
 \end{eqnarray}
 which we consider as an element of the  polynomial algebra formally generated by configurations of geodesics. We  have similar formulas when $G$ and $H$ are geodesics with Theorem \ref{theo:A} then giving a generalization of Wolpert's cosine formula \cite{Wolpert:1983td}. In the case presented in the introduction -- the study of projective Anosov representations --  the ghost bracket is actually a Poisson bracket and is easily expressed in paragraph \ref{sec:bracket-ghost} using  the swapping bracket introduced by the second author in \cite{Labourie:2012vka}. Formula \eqref{eq:def-ghost-bra} is very explicit and the Poisson Stability Theorem \ref{theo:A} now becomes an efficient tool to compute recursively brackets of averaged correlations functions and length functions.
\vskip 0.2 truecm
In this spirit, we give two applications of this stability theorem. Following Martone--Zhang \cite{Martone:2019uf}, say a projective Anosov representation $\rho$ {\em admits a positive cross ratio} if
$0<\tr(\pp_\rho(g)\pp_\rho(h))<1$ for any two intersecting geodesics $g$ and $h$. Examples come from Teichm\"uller
 spaces and   Hitchin representations \cite{Labourie:2005,Martone:2019uf}. More generally positive representations are associated to positive cross ratios \cite{BGLPW}. Our first application is a generalisation of the convexity theorem of Kerckhoff \cite{Kerckhoff:1983th} and was the initial reason for our investigation:
 
 \begin{theo}[\sc Convexity Theorem]\label{theo:B} Let $\mu$ be the geodesic current associated to a measured geodesic lamination, $\Ell_\mu$ the associated length function. Let $\rho$ be a projective Anosov representation which admits a positive cross ratio, then for any geodesic current $\nu$,
 $$
 \{\Ell_\mu,\{\Ell_\mu,\Ell_\nu\}\}\geq 0\ .
 $$
 Furthermore the inequality is strict if and only if  $i(\mu,\nu) \neq 0$.
 \end{theo}
Recall that in a symplectic manifold $\{f,\{f,g\}\}\geq 0$ is equivalent to the fact that $g$ is convex along the Hamiltonian curves of $f$. This theorem involves a generalization of Wolpert's sine formula \cite{Wolpert:1983td}. 

Our second, and less surprising, application is to construct commuting subalgebras in the Poisson algebra of correlation functions. Let $\mathcal L$ be a geodesic lamination whose complement is a union of geodesic triangles $C_i$. To each such triangle, we call the associated correlation function $\T^*_{C_i}$ an {\em associated triangle function}. The {\em subalgebra associated to the lamination} is the subalgebra generated by triangle functions and length functions for oriented geodesic currents supported on $\mathcal L.$
\begin{theo}[\sc Commuting subalgebra]\label{theo:C} 
	For any geodesic lamination whose complement is a union of geodesic triangles, the associated subalgebra is commutative with respect to the Poisson bracket.
\end{theo}
In the context of coordinate functions and Hitchin representations, note that it is a well known fact that  geodesic laminations are associated to commuting subalgebras \cite{Fock:2006a,Bonahon:2014woa,Sun:2017}, but observe that  our results are more general: they are  about any type of projective Anosov representations.

\bigskip
   
This paper has been structured in several parts and sections, all of them preceded by an introduction describing the content of these sections or paragraphs. We concentrate in this introduction on the projective case. 
We now describe briefly the content of the paper and of some of our constructions of independent interest, namely the construction of a theory of  uniformly hyperbolic bundles and of a (non linear) intersection and integration for ghost polygons. 

\bigskip
\noindent{\bf Beyond representations: uniformly hyperbolic bundles.} This paper introduces a novel tool for describing what may be called ``universal Anosov representations'', in the spirit of  universal Teichmüller spaces. The use of this tool has been made necessary for understanding variations of the correlation functions. These universal representations are defined using {\em uniformly hyperbolic bundles}, as defined in \ref{def:unif-hyp-bund}. This new approach extends known results for Anosov representations. Notably, we show that these bundles share the fundamental properties of Anosov representations. We show that 
\begin{itemize}
\item 	uniformly hyperbolic bundles  are {\em stable} -- see Theorem \ref{theo:stab},
\item  uniformly hyperbolic bundles admit {\em Hölder limit curves} -- see  Propositions \ref{pro:pbd-unique} and \ref{pro:holder} as well as  Definition \ref{def:limit-map}.
\end{itemize}
We emphasize that periodicity under a cocompact discrete group is not required. Consequently, the (unaveraged) correlation functions become meaningful and their variation can be computed as shown in Proposition \ref{pro:dT-integ}. This result stems from solving the dynamical cohomological equation (Proposition \ref{pro:deriv-fund-proj}). Furthermore, important constructions such as ghost integration and ghost intersection are explored within the framework of uniformly hyperbolic bundles. We hope this new notion will be helpful in describing general phenomena in particular in the interaction between curves in flag manifolds and analytic problems.

\vskip 0.5 truecm

\noindent{\bf Ghost polygons and ghost integration:}
 In the first and preliminary section, we recall classical constructions; first in the hyperbolic plane, we explain the dual form of a geodesic $\omega_g$ and 
the intersection $\epsilon(g,h)$ of two geodesics $g$ and $h$ related by the following formula that we give to explain our later motivations:
\begin{equation}
	\epsilon(g,h)=\int_{g}\omega_h=-\int_h\omega_g=\int_\hh \omega_g\wedge\omega_h\ .\label{eq:hyp-int}
\end{equation}
We also recall the Atiyah--Bott--Goldman symplectic form.
\vskip 0.2 truecm
Our first part {\bf on the hyperbolic plane} deals with a generalisation of the above formula, using what we described above as "Anosov representations without a group" or more precisely as {\em uniformly hyperbolic bundles} which is the background of the whole project.  Given such a uniformly hyperbolic bundle $\rho$ and a configuration $G$ of geodesics -- that we call a {\em ghost polygon} -- we are then able to compute the derivatives of the correlation function $\T_G$, for a variation $\dot\nabla$ of the uniformly hyperbolic bundle at $\rho$:
$$
\d \T_G \left(\dot\nabla\right)=\oint_{\rho(G)}\dot\nabla\ ,
$$
where the righthand side is a procedure called {\em ghost integration} which involves computing solutions of the cohomological equation of dynamical systems. Motivated by this computation, and fixing a uniformly hyperbolic bundle,  we introduce the ghost intersection  $\bI_\rho(G,H)$ of two ghost polygons $G$ and $H$ and the ghost dual form  $\Omega_{\rho(G)}$ -- with values in some endomorphism bundle --  of a ghost polygon so that 
$$
\bI_\rho(G,H)=\oint_{\rho(G)}\Omega_H=-\oint_{\rho(H)}\Omega_G=\int_\hh \tr\left(\Omega_H\wedge\Omega_G\right)\ ,
$$
a formalism reminiscent of formulae \eqref{eq:hyp-int}. All this is better encoded by introducing a {\em ghost bracket} on a {\em ghost algebra} -- the polynomial algebra generated by ghost polygons -- and extending correlation functions to the ghost algebra to get
$$
\T_{[G,H]}(\rho)=\bI_{\rho}(G,H)\ .
$$ The constructions outlined above  are the analogues of classical constructions (integration along a path, intersection of geodesics) in differential topology described in section \ref{sec:prelim}, but in a non abelian setting.
\vskip 0.2 truecm
The second part {\bf on  closed surfaces} then moves to averaging  correlation functions and length functions, by using currents (for geodesics) and {\em cyclic currents} for ghost polygons. This is the part where we prove Theorem \ref{theo:A}. We prove this by carefully exchanging some integrals and using the formulae that we have obtained for the derivatives of correlation functions.
\vskip 0.2 truecm
In the third part, we prove the {\bf applications}: Theorems \ref{theo:B} and \ref{theo:C}. In the final  part, {\bf addendum}, we  establish a result of independent interest on the ghost bracket: the ghost bracket satisfies the Jacobi identity except for degenerate cases. We also recall some technical points.
 \vskip 0.2 truecm

  \vskip 0.2 truecm

\noindent{\bf A final remark: the general case of (non) projective Anosov representations. \ } For the sake of simplicity, this introduction focused on the case of  projective Anosov representations. More generally, one can construct correlation functions out of {\em geodesics decorated} with weights of the Lie group $\G$ with respect to a $\Theta$-Anosov representations. The {\em $\Theta$-decorated correlation functions} are described by configurations of $\Theta$-decorated geodesics.  The full machinery developed in this article computes more generally the brackets of these decorated correlation functions. Using that terminology, the Poisson Stability Theorem \ref{theo:A} still holds with the same statement, but the ghost bracket has to be replaced by a {\em decorated ghost bracket} which follows  a construction  given in paragraph \ref{par:theta-ghost-bra}, slightly more involved  than formula \eqref{eq:def-ghost-bra} and not anymore related to the swapping bracket.

\vskip 1 truecm
We would like to thank Dick Canary and Tianqi Wang for very useful comments, Fanny Kassel, Joaqu{\'i}n Lema, Curt McMullen,  Andrés Sambarino and  Tengren Zhang for their interest and remarks. We would also like to thank the reviewer for their comments and suggestions which greatly improved the paper.

\tableofcontents
\section{Preliminary}\label{sec:prelim}

In this preliminary section, we recall basic and well known facts about the intersection of geodesics in the hyperbolic plane, dual forms to geodesics and the Goldman symplectic form. We also introduce one of the notions important for this paper: geodesically bounded forms.

\subsection{The hyperbolic plane, geodesics and forms}\label{sec:hypplane}

We first recall classical results and constructions about closed geodesics in the hyperbolic plane.

\subsubsection{Geodesics and intersection}\label{sec:geod-inter}

Let us choose an orientation in $\hh$. We denote in this paper by $\GG$ the space of oriented geodesics of $\hh$ that we identify with the space or pairwise distinct points in $\bh$.  We denote by $\bar g$ the geodesic $g$ with the opposite orientation.

\begin{definition} Let $g_0$ and $g_1$ be two oriented geodesics.
	The {\em intersection of $g_0$ and $g_1$} is the number $\epsilon(g_0,g_1)$ which satisfies the following rules $$
\epsilon(g_0,g_1)=-\epsilon(g_1,g_0)=-\epsilon(\bar g_0,g_1)\ ,
$$
and verifying the following 
\begin{itemize}
	\item $\epsilon(g_0,g_1)=0$ if $g_0$ and $g_1$  do not intersect or $g_0=g_1$.
	\item $\epsilon(g_0,g_1)=1$ if $g_0$ and $g_1$  intersect and $(g_0(\infty),g_1(\infty),g_0(-\infty),g_1(-\infty)$ is oriented.
\item  $\epsilon(g_0,g_1)=1/2$ if $g_0(-\infty)=g_1(-\infty)$ and  $(g_0(\infty),g_1(\infty),g_1(-\infty))$ is oriented.
\end{itemize}
\end{definition} 

Observe that $\epsilon(g_0,g_1)\in\{-1,-1/2,0,1/2,1\}$ and that we have the {\em cocycle property}, if $g_0,g_1,g_2$ are the sides of an ideal triangle with the induced orientation, then for any geodesic $g$ we have
\begin{equation}
	\sum_{i=0}^2\epsilon(g,g_i)=0\ . \label{eq:triang-intersect}
\end{equation}
We need an extra convention for coherence  
\begin{definition}
	A {\em phantom geodesic} is a pair $g$ of identical points $(x,x)$ of $\partial_\infty\hh$. If $g$ is a phantom geodesic, $h$ any geodesic (phantom or not), we define $\epsilon(g,h)\defeq 0$.
\end{definition}

\subsubsection{Geodesic forms}
Let us denote by $\Omega^1(\hh)$ the space of 1-forms on the hyperbolic plane. 
  
\begin{definition}\label{def:bounded-forms}
	 A form $\omega$ in  $\Omega^1(\hh)$ is {\em bounded} if $|\omega_x(u)|$ is bounded uniformly for all $(x,u)$ in  $\ms U\hh$ the unit tangent bundle of $\hh$. We note by  $\bLa^\infty$ be the vector space of bounded forms.
\end{definition}
 
\begin{proposition}\label{pro:def-omegag}
We have a $\psld$ equivariant mapping 
$$\mapping{\GG}{\Omega^1(\hh)\ ,}{g}{\omega_g\ , }$$ which satisfies the following properties
\begin{enumerate}
\item   $\omega_g$ is a closed 1-form in $\hh$ supported in the tubular neighbourhood  of $g$ at distance $1$, outside the tubular neighbourhood  of $g$ at distance $1/2$.\label{it:i}
	\item $\omega_g=-\omega_{\bar g}$ \label{it:ii}
	\item Let  $g_0$ be any geodesic, then
	\begin{eqnarray}\int_{g_0}\omega_g& =&\epsilon(g_0,g)\ .\label{eq:duality0}
	\end{eqnarray}
\end{enumerate}	
\end{proposition}
\begin{proof} The construction runs as follows. Let us fix a function $f$ from $\mathbb R^+$ to $[0,1]$ with support in $[0,1]$ which is constant and equal to $1/2$ on $[0,1/2]$  neighbourhood of $0$. We extend (non-continuously) $f$ to $\mathbb R$ as an odd function. Let finally $R_g$ be the ``signed distance" to $g$, so that $R_{\bar g}=-R_g$. We finally set $\omega_g=-\d (f\circ R_g)$. Then \eqref{it:i} and \eqref{it:ii} are obvious. We leave the reader check the last point in all possible cases.
\end{proof}

\rmka   We extend the above map to phantom geodesics by setting $\omega_g=0$ for a phantom geodesic and observe that the corresponding assignment still obey proposition \ref{pro:def-omegag}.

The form $\omega_g$ is called the {\em geodesic form} associated to $g$. Such an assignment is not unique, but we fix one, once  and for all. Then we have 

\begin{proposition}
	For any pair of geodesics $g_0$ and $g_1$, $\omega_{g_1}\wedge \omega_{g_0}=f\d \operatorname{area}_\hh$ with $f$ bounded and in $L^1$.
\end{proposition}

\begin{proof}
The only non-trivial case is if $g_0, g_1$ share an endpoint. In the upper half plane model let $g_0$ be the geodesic $x=0$, while $g_1$ is the geodesic $x=a$. Observe that the support of $\omega_{g_1}\wedge\omega_{g_0}$ is in the sector $V$ defined by the inequations $y>B>0$ and $\vert x/y\vert < C$ for some positive constants $A$ an $C$. Finally as the signed distance for $g_0$ satisfies $\sinh(R_{g_0}) = x/y$ then
$$\omega_{g_0}=f_0\ {\rm d} \left(\frac{x}{y}\right)\ , \ \omega_{g_1}=f_0\  {\rm d} \left(\frac{x-a}{y} \right)\ ,$$ where $f_0$ and $f_1$ are functions bounded by a constant $D$. An easy computation shows that 

$$
{\rm d}\left(\frac{x}{y}\right)\wedge {\rm d} \left(\frac{x-a}{y} \right)
=a \  \frac{{\rm d} x\wedge {\rm d y}}{y^3}\ .
$$
Oberve that $\vert f_1 f_0 a\vert$ is bounded by $D^2a$, and 
$$
\int_V \frac{{\rm d} x\wedge {\rm d y}}{y^3}\leq 2C\int_B^\infty \frac{1}{y^2}{\rm d} y=\frac{1}{B}< \infty\ .
$$
This completes the proof. \end{proof}

 \rmka The above result is still true whenever $g$ or $h$ are phantom geodesics.

From that it follows that
\begin{proposition}\label{prop:inter}
For any pair of geodesics, phantom or not,    $g$ and $g_0$, we have 	\begin{eqnarray}\int_{g_0}\omega_g = \epsilon(g_0,g)=\int_{\hh}\omega_{g_0}\wedge \omega_{g}	\ .\label{eq:duality2}
 	\end{eqnarray}
 	Moreover for any  (possibly ideal) triangle $T$ in $\hh$
	\begin{equation}
		\int_{\partial T}\omega_g=0\ .\label{eq:triangle-vanish0}
	\end{equation}
\end{proposition}

\subsection{The generic set and barycentric construction}
For any oriented geodesic $g$ in $\GG$ we denote   by $\bar g$ the geodesic with opposite orientation, and we write $g\simeq h$, if either $g=h$ or $g=\bar h$. Let us also 
denote the extremities of $g$ by $(\partial^-g,\partial^+g)$ in $\bh\times\bh$.

For $n\geq 2$, let  us the define the 
{\em singular set} as
$$
\GG^n_{1}\defeq\{(g_1,\ldots,g_n)\mid \forall \ i,j, \ \ g_i\simeq g_j \}\ ,
$$
and the 
{\em generic set} to be
$$
\GG_\star^n\defeq \GG^n\setminus \GG^n_{1}\ .
$$

We define a {\em $\Gamma$-compact set} in $\GG_\star^n$  to be the preimage of a compact set in the quotient $\GG_\star^n/\Gamma$.

The {\em barycenter} of a family $G=(g_1,\ldots,g_n)$ of geodesics is the point $\bary(G)$ which attains the minimum of the sum of the distances to the geodesics $g_i$. 
Choosing a uniformization, the barycentric construction yields a  $\psld$-equivariant map from 
$$
\bary:\mapping
{\GG_\star^n}
{\hh\ ,}
{(g_1,\ldots,g_n)}
{\bary(y)\ .}
$$
It follows from the existence of the barycenter map that  the diagonal action of $\Gamma$ on $\GG_\star^n$ is proper. The {\em barycentric section} is then the section $\sigma$ of the following fibration restricted to $\GG_\star^n$ 
$$
F:(\mathsf U\hh)^n\to \GG^n\ ,
$$
given by 
$$
\sigma=(\sigma_1,\ldots,\sigma_n)\ ,
$$
where $\sigma_i(g_1,\ldots,g_n)$ is the orthogonal projection of $\bary(g_1,\ldots,g_n)$ on $g_i$. 
Obviously

\begin{proposition}
	The barycentric section is equivariant  under the diagonal  action of $\psld$ on $\GG_\star^n$ as well as the natural action of the symmetric group $\mathfrak S_n$.
\end{proposition}

\subsection{Geodesically bounded forms}

We abstract the properties of geodesic forms in the following definition,    where we use definition \ref{def:bounded-forms}:

\begin{definition}[{\sc Geodesically bounded forms}] \label{def:geod-bded} Let $\alpha$ be a closed 1-form on $\hh$. We say that $\alpha$ is {\em geodesically bounded} if \begin{enumerate}
\item $\alpha$ belongs to $\bLa^\infty$ and $\nabla\alpha$ is bounded.
\item for any geodesic $g$, $\alpha(\dot g)$ is in $L^1(g, \d t)$, $\omega_g\wedge\alpha$ belongs to $L^1(\hh)$ and
	\begin{equation}
		\int_{g}\alpha =\int_{\hh}\omega_g\wedge \alpha\ .\label{eq:duality1}
	\end{equation}
\item Moreover for any  (possibly ideal) triangle $T$ in $\hh$
	\begin{equation}
		\int_{\partial T}\alpha =0\ .\label{eq:triangle-vanish}
	\end{equation}
\end{enumerate}
\end{definition}
We denote by $\bXi$ the vector space of geodesically bounded forms.
We observe that any geodesically bounded form is closed and that any geodesic form belongs to $\bXi$.

\subsection{Polygonal arcs form}\label{sec:polyg-arc}

We will have to consider {\em geodesic polygonal arcs} which are a finite union of oriented geodesic arcs
$$
\bgamma=\gamma_0\cup\cdots\cup\gamma_p\ ,
$$
such that $\gamma_i$ joins $\gamma_i^-$ to  $\gamma_i^+$ and we have $\gamma_i^-=\gamma_{i-1}^+$, while $\gamma_0^-$ and $\gamma_p^+$ are distinct points at infinity. 
We say that $\gamma_1,\ldots,\gamma_{p-1}$ are the {\em interior arcs}.

We have similar to above

\begin{proposition}[\sc Dual forms to polygonal arcs] Given a geodesic polygonal arc $\bgamma=\gamma_0\cup\cdots\cup\gamma_p$ there exists a closed 1-form $\omega_\bgamma$ so that
\begin{enumerate}
	\item the 1-form $\omega_\bgamma$ is supported on a 1-neighborhood of $\bgamma$,
	\item Let $B$ be a ball containing a 1-neighbourhood of the interior arcs, such that outside of $B$ the 1-neighbourhood $V_0$ of $\gamma_0$ and the 1-neighbourhood $V_1$ of $\gamma_p$ are disjoint then
	$$
	\left.\omega_{\bgamma}\right\vert_{V_0}=\left.\omega_{g_0}\right\vert_{V_0}\ \ \ ,\ \ \  \left.\omega_{\bgamma}\right\vert_{V_1}=\left.\omega_{g_p}\right\vert_{V_1}\ .
	$$
	where $g_0$ and $g_p$ are the complete geodesics containing the arcs $\gamma_0$ and $\gamma_p$. 
	\item For any element $\Phi$ of $\psld$,
	$\omega_{\Phi(\bgamma)}=\Phi^*(\omega_\bgamma)$.
	\item For any geodesic $g$,
	$
	\int_g\omega_{\bgamma}=\epsilon(g,[\gamma_0^-,\gamma_p^+])
	$.
	\item\label{it:integ-polyg}
	Let $\bgamma$   be a polygonal arc with extremities at infinity $x$ and $y$, then for any 1-form $\alpha$ in $\bXi$ we have
	$$
	\int_{\hh}\omega_{\bgamma}\wedge\alpha=\int_{[x,y]}\alpha \ .
	$$
\end{enumerate}
\end{proposition}
\begin{proof} The construction runs as the one for geodesics. Let us fix a function $f$ from $\mathbb R^+$ to $[0,1]$ with support in $[0,1]$ which is constant and equal to $1/2$ on $[0,1/2]$. We extend (non-continuously) $f$ to $\mathbb R$ as an odd function. Let finally $R_g$ be the ``signed distance" to $g$, so that $R_{\bar g}=-R_g$. We finally set $\omega_g=-\d (f\circ R_g)$. Then $(1)$, $(2)$, $(3)$ and $(4)$ are obvious.

Then writing $\hh\setminus\bgamma=U\sqcup V$ where $U$ and $V$ are open connected sets. We have that
$$
\int_U \omega_\bgamma\wedge\alpha=\int_U \d (f\circ R_g)\wedge \alpha=\frac{1}{2}\int_g \alpha\ ,
$$
by carefully applying  Stokes theorem. The same holds for the integral over $V$, giving us our desired result.
\end{proof}

The form $\omega_\bgamma$ is the {\em polygonal arc form}.

\subsection{The Goldman symplectic form}\label{sec:gold-sympl}
Let $S$ be a closed surface with  $\Sigma$ its universal cover that we identify with $\hh$ by choosing a complete hyperbolic structure on $S$. Given a representation $\rho:\pi_1(S) \rightarrow G$ we  let $E = \Sigma\times_\rho\mathfrak g$ be the bundle over $S$ by taking the quotient of the trivial bundle  $\Sigma\times\mathfrak g \rightarrow \Sigma$ by the  action of $\pi_1(S)$ given by  $\gamma(x,v) = (\gamma(x), Ad\rho(\gamma) (v))$. Let $\nabla$ be the associated flat connection on the bundle $E$  and denote by $\Omega^k(S)\otimes\End(E)$ the vector space of $k$-forms on $S$ with values in $\End(E)$. Recall that $\nabla$ gives rise to a differential 
$$\d^\nabla:  \Omega^k(S)\otimes\End(E)\to\Omega^{k+1}(S)\otimes\End(E)\ . $$ We say a  1-form $\alpha$ with values in  $\End(E)$ is {\em closed} if $\d^\nabla\alpha=0$ and {\em exact} if $\alpha=\d^\nabla\beta$.
 Let then consider
\begin{eqnarray*}
	C^1_\rho(S)&\defeq&\{\hbox{closed 1-forms with values in $\End(E)$}\}\ ,\crcr
	E^1_\rho(S)&\defeq&\{\hbox{exact 1-forms with values in $\End(E)$}\}\ ,\crcr
	H^1_\rho(S)&\defeq&C^1_\rho(S)/E^1_\rho(S)\ .
\end{eqnarray*}
\begin{definition}
	When $S$ is closed, the {\em Goldman symplectic form}  $\bOm$ on $H^1_\rho(S)$ is given by 
\begin{equation}
\bOm(\alpha,\beta)\defeq\int_S\tr(\alpha\wedge\beta)\ ,	
\end{equation}
where for $u$ and $v$ in $\T S$: $\tr(\alpha\wedge\beta)(u,v)\defeq\tr(\alpha(u)\beta(v))-\tr(\alpha(v)\beta(u))$.
\end{definition}

Identifying $H^1_\rho(S)$ with the tangent space to the representation variety $\hom(\pi_1(S),\G)/\G$ at $[\rho]$, this is the classical Goldman symplectic form on representation varieties introduced in \cite{Goldman:1984}, described with a De Rham point of view in \cite{Labourie:2013ka}. Observe that if we consider complex bundles, the Goldman symplectic form is complex valued, while it is real valued for real bundles.

\subsection{Geometry of $\uh$}\label{sec:geom-uh}
Recall first some terminology and recall that the unit bundle $\uh$ of $\hh$ carries three fundamental 1-dimensional foliations
\begin{enumerate}
	\item The foliation by the orbits of the geodesic flow $\flo{\phi}$ generated by the vector field $\partial_t$,
	\item The {\em strong stable foliation},
	\item The {\em strong unstable foliation}.
\end{enumerate}
The three foliations are themselves invariant by the flow. Then the two {\em central stable foliation} and {\em central unstable foliation} are the 2-dimensional foliations generated by the strong stable foliation (respectively the strong unstable foliation) and the flow.

These foliations are related to the {\em boundary at infinity $\bh$} of $\hh$, by the projections $\varpi^+$ from $\uh$ to $\bh$, such that the preimages of every point is exactly a central stable leaf. In other words $\bh$ is the space of central stable leaves of $\uh$ or equivalently the space of central unstable leaves. We recall that $\bh$ is homeomorphic to $\mathsf{S^1}$.
Thus, we have
\begin{enumerate}
	\item The projection $\varpi^+$ from $\uh$ to the space of central stable leaves,
	\item The projection $\varpi^-$ from $\uh$ to the space of central unstable leaves,
\end{enumerate}

\part{On the hyperbolic plane} In this first part, we deal with higher rank analogues of geodesics, intersections  and hyperbolic metrics on the disk, which we recalled in  paragraph \ref{sec:hypplane}. We define these notions independent of the presence of a cocompact surface group and work on the hyperbolic plane and more precisely on the unit tangent bundle of the hyperbolic plane.

 One of our first goals is to define (non averaged) correlation functions $\T_G(\nabla)$ associated to a configuration of geodesics $G$, and uniformly hyperbolic bundle $\nabla$ (the non periodic generalisation of Anosov representations). We more precisely aim at computing the variation of $\T_G(\nabla)$ when one varies $\nabla$. This computation  is achieved through the introduction of {\em ghost polygons}, their {\em ghost dual form } and {\em intersections}, as well as  the {\em ghost integration}, generalizing step by step the classical framework of geodesics, intersection and dual forms explained in paragraph \ref{sec:geod-inter}
.
 \begin{enumerate}
 	\item  In section \ref{sec:unif-hyp-bund}, we answer to the question: how do we generalize the hyperbolic metric on the disk in higher rank?  In the periodic case, that is in the presence of a closed surface or alternatively a cocompact group acting on the disk, a good answer is given by Anosov representations. To deal with this non periodic case, we introduce {\em uniformly hyperbolic bundles} and their {\em fundamental projector} $p$. Such a fundamental projector is associated to a pair of distinct points in $\partial_\infty\hh$. Our main result is then given by proposition  \ref{pro:deriv-fund-proj} that computes the variation of the fundamental projectors for a variation of a uniformly hyperbolic bundle using the cohomological equation.
\item In section \ref{sec0:ghost-polygon}, we introduce some of the main players of this article: the {\em ghost polygons} which are associated to a cyclic configuration $G$ of geodesics. To such a ghost polygon $G$ and a uniformly hyperbolic bundle $\nabla$ we associate a correlation function $\T_G(\nabla)$ and prove  certain  analytic properties.
\item In section \ref{sec:ghost-integ}, we explain how to integrate on a ghost polygon certain 1-forms with values in the endormorphisms bundle, and among them the 1-form $\dot\nabla$ associated to a variation of a uniformly hyperbolic bundle. This {\em ghost integration} is the non abelian analogue of integration along geodesics.  This allows us to obtain a formula and proposition \ref{pro:dT-integ} which fulfills our first goal: {\em computing the variation of correlation functions from the variation of a hyperbolic bundle}. In this same section we introduce the dual objects to ghost integration given by certain 1-forms with values in the endomorphisms bundles, playing the role of Poincaré duality in this context. 
\item Finally in section \ref{sec:ghost-inter}, we use the dual 1-forms to define the {\em ghost intersection} of two ghost polygons, a procedure dual to wedging. This intersection (with respect to a uniformly hyperbolic bundle) is actually purely combinatorial: it defines a bracket on the formal polynomial algebra generated by ghost polygons. We finally explain how this combinatorial structure is related in the projective case to swapping algebras.
 \end{enumerate}
At each of these steps, the special case of projective uniformly hyperbolic bundles gives simpler formulae.
 
\vskip 0.2 truecm
 In the next part, we use this first part to obtain results in the periodic case after averaging. But we note that the results of this first part do not depend on the assumption of periodicity.

\section{Uniformly hyperbolic  bundles and projectors}\label{sec:unif-hyp-bund}

We introduce the notion of {\em uniformly hyperbolic bundles} over the unit tangent bundle $\uh$  of $\hh$ -- see  definition \ref{def:unif-hyp-bund}.  This notion is a universal version of Anosov representations defined in \cite{Labourie:2006}.     Roughly speaking an Anosov representation is given by a flat bundle satisfying some expanding/contracting features. These expanding/contracting features are measured with respect to the choice of a Euclidean metric on the bundle. However, on a compact surface this choice is irrelevant: all Euclidean metrics are uniformly equivalent. On the contrary, if we work on the hyperbolic plane, the choice of the Euclidean metric is now meaningful and has to be specified in the definition.

More specifically, we explain in the projective case, that such an object is a Euclidean  vector bundle $E$ over $\mathsf U\hh$ satisfying some contracting properties,  the metric being only considered up to uniform equivalence. Associated to these bundles are  \begin{enumerate}
 \item A special section $\pp$ of the endomorphism bundle $\End(E)$ given by projectors fibrewise and that we call the {\em fundamental projector}
 \item a flat connection $\nabla$ on the bundle $E$, such that 
\begin{equation}
	\nabla_\pt\pp=0\ ,\label{eq:nablap}
\end{equation}
where $\pt$ is the generator of the geodesic flow on $\mathsf U\hh$.
\item Special curves in flag manifolds that by extension of Anosov representations we also call {\em limit maps}. 
 \item As ancillary definitions, we have that the notions of a
\begin{enumerate}
 \item  {\em family of uniformly hyperbolic bundles} -- and the associated stability lemma -- 
 \item  {\em variation of a uniformly hyperbolic bundle}: we describe this as a 1-form $\dot\nabla$  with values in the endormorphism bundle of a uniformly hyperbolic bundle. 
 \item {\em  equivalent bundles} using either  gauge fixing or metric fixing.
 \end{enumerate}\end{enumerate}
Our main result -- proposition \ref{pro:deriv-fund-proj} -- is the  description of the variation  $\dot\pp$ of such a projector under a variation of the data defining the uniformly hyperbolic bundles. We present the outline briefly: differentiating $\pp = \pp^2$ we see that
\begin{equation*}
\dot\pp =\dot \pp \pp +\pp\dot\pp
\end{equation*}
and it follows that $\dot\pp$ is a section of the subbundle $F_0$ of $\End(E)$ where
\begin{equation*}
F_0 = \{ B \in \End(E)\ | \ B=B\pp+\pp B\}.
\end{equation*}
As a consequence of uniform hyperbolicity this bundle splits with $F_0 = F_0^+\oplus F_0^-$ where $F_0^-$ contracts under forward flow an $F_0^+$ contracts under negative flow. Using these contraction properties, we see that $\dot\pp$ is a solution of  the {\em cohomological equation} obtained by differentiating \eqref{eq:nablap}, 
$$
\nabla\dot\pp + [\dot\nabla,\pp] =0\ .
$$
We then solve this equation to obtain an integral formula  for $\dot\pp$  under a change in  connection $\dot\nabla$
$$		\dot\pp(x)=\int_{-\infty}^0 \left( \pp\cdot [\pp,\dot\nabla_{\ps}] \right)(x^s)\ \d s- \int_0^\infty \left([\pp,\dot\nabla_{\ps}]\cdot \pp \right)(x^s)\ \d s\ \ ,
$$
This rough outline has to be made rigorous, and has several avatars depending on whether we want to emphasize gauge fixing or metric fixing, a classical fixture in Higgs bundles.

\vskip 0.2 truecm
After defining uniformly hyperbolic bundles, we consider {\em $\Theta$-uniformly hyperbolic bundles}, which is just an extra decoration.

Finally, we recover Anosov representations as periodic cases of uniformly hyperbolic bundles.  Also the notion of a uniformly hyperbolic bundle  is the structure underlying the study of quasi-symmetric maps in \cite{Labourie:2020tv}.

This notion has a further  generalization to all hyperbolic groups $\Gamma$, replacing $\uh$ by a real line bundle $X$ over 
$$\partial_\infty\Gamma\times\partial_\infty\Gamma
\setminus\{(x,x)\mid x\in\partial_\infty\Gamma\}\ ,$$
equipped with a $\Gamma$-action so that 
$X/\Gamma$ is the geodesic flow of $\Gamma$ as in \cite{Bridgeman:2015ba}.
We will not discuss it in this paper, since this will needlessly burden our notation.

We introduce several bundle related notation for  bundles that we will use freely in the sequel.
\begin{enumerate}
	\item First when we have a projection $\pi$ from a set $F$ to $\uh$, we denote  $F_x\defeq \pi^{-1}(x)$ for $x$ in $\uh$. 
	\item When $F$ and $G$ both projects to $\uh$ by $\pi_F$ and $\pi_G$, we consider the closed subset
$$
F\boxtimes G\defeq\{(f,g)\in F\times G\mid \pi_F(f)=\pi_G(g)\}\ .
$$
\end{enumerate}

\subsection{Uniformly hyperbolic bundles: definitions and first properties}
Let $\uh$ be the unit tangent bundle of $\hh$. We denote by $\ps$ the vector field on $\uh$ generating the geodesic flow $\flo{\phi}$. This first paragraph deals with basic definitions and properties.

\subsubsection{Preliminary: Metrics on Grassmannian}    In this paragraph,  we first  collect facts on Grassmannians and vector spaces without considering the bundle context.

We first recall that given a Euclidean metric $g_0$ on a vector space $E_0$, we define the {\em angle} between vector subspaces $F$ and $G$ of $E_0$ as
$$
\angle_0(F,G)=\inf\left\{\angle_0(u,v)\mid u\in F, v\in G\right\}\ ,
$$
where for non zero vectors $u$ and $v$
$$ 
\angle_0(u,v)=\arccos\left(\frac{g_0(u,v)}{\sqrt{g_0(u,u)g_0(v,v)}}\right)\ .
$$
Observe that  $\angle_0(F,G)=0$ if and only if $\dim(F\cap G)\geq 1$. Finally, we consider the  distance $d_0$  of diameter 1 in $\Gk(E_0)$, associated to $g_0$.

In the next lemma and the sequel, we identify $\Gk(E_0^*)$ with $\Gnk(E_0)$ where $n=\dim(E_0)$.

We need to introduce three different sets depending on some positive $\epsilon$.
\begin{enumerate}
	\item We first consider 	\begin{align*}
O_\epsilon\defeq \{(\uu,\vv)\in \Gk(E_0)\times \Gk(E_0^*)\ \mid \    \angle_0(\uu,\vv)\geq \epsilon\} \ .
	\end{align*}
	\item For  $P_0$ in $\Gk(E_0^*)$,  we consider 
	\	\begin{align*}
O^+_\epsilon(P_0)\defeq \{\uu\in \Gk(E_0) \ \mid\  \angle_0(\uu,P_0)\geq \epsilon\}\ .
	\end{align*}
	\item Finally, for $L_0$ in  $\Gk(E_0)$, we consider
	 	\	\begin{align*}
O^-_\epsilon(L_0)\defeq \{L\mid \angle_0(L,P)\geq \epsilon, \hbox{ for all } P \hbox{ such that }\angle_0(L_0,P )\geq {2}\epsilon\}=\bigcap_{P\ \mid \ (L_0,P )\in O_{2\epsilon}}O^+_\epsilon(P)\ , 
	\end{align*} 
\end{enumerate}
Observe that we have the following inclusions for $(L_0,P_0)$ in $O_{2\epsilon}$\ , 
$$
\{L_0\}\subset O^-_\epsilon(L_0)\subset O^+_\epsilon (P_0)\ \ \ ,\ \ \ O^+_\epsilon (P_0)\times \{P_0\} \subset O_{\epsilon}\ .
$$

\begin{lemma}[\sc Equivalent metrics]\label{lem:equi}
	Let $E_0$ be a vector space. Then for any positive $\epsilon$ there exist positive constants $K$, $\alpha$ and $\beta$ such that for any Euclidean metric $g_0$ in $E_0$, the following holds:  
	
Let $(L_0,P_0)$ in $O_{2\epsilon}$ and $\Vert\ \cdot \Vert_0$ the Euclidean norm on $L_0^*\otimes P_0$ coming form the Euclidean norms on $L_0$ and $P_0$,.
Let  $B$ and $C$ are elements of $L_0^*\otimes P_0$ and denote   $L_B$ and $L_C$ their graphs of $B$  and $C$. Then 
\begin{enumerate}
\item Assume that
  $L_C$ and $L_B$ are  in $O^+_{\epsilon}(P_0)$, then we have the following inequality.
	 \begin{align}
	 	 \frac{1}{K}\ d_0(L_C,L_B)\leq \Vert B-C\Vert_0 \leq  K\  d_0(L_C,L_B)\ .\label{ineq:metric-equiva}
	 \end{align}
	 	\item  if  $L_B$ belongs to $O_\epsilon^+(P_0)$, then 
	  $\Vert B\Vert_0< \beta$,
	\item Conversely, if $\Vert B\Vert_0< \alpha$, then  $L_B$ belongs to $O^-_\epsilon (L_0)$.

\end{enumerate}
\end{lemma} 
\begin{proof}  We first prove the first  statement.  The easy proof  relies on a compactness argument. Indeed let $(L_0,P_0)$ in $O_\epsilon$ and consider 
$$
\mathcal K_{L_0,P_0}\defeq \{B\in L_0^*\otimes P_0\mid (L_B,P_0)\in  O_\epsilon\}\ .
$$
The set  $\mathcal K_{L_0,P_0}$ is  compact in the vector space  $L_0^*\otimes P_0$ on which, we have two Riemannian metrics,
\begin{enumerate}
	\item  the one induced from the Riemannian metric $d_0$ on $\Gk(V)$ by the map $B\mapsto L_B$,
	\item the Euclidean norm on $L_0^*\otimes P_0$  coming from the Euclidean metric $g_0$.
\end{enumerate}
It follows that these two Riemannian metrics are equivalent by a constant $K_{L_0,P_0}$ depending on $(L_0, P_0)$ and $\epsilon$, that is 
 \begin{align}
	 	 \frac{1}{K_{L_0,P_0}}\ d_0(L_C,L_B)\leq \Vert B-C\Vert_0 \leq  K_{L_0,P_0}\  d_0(L_C,L_B)\ .
	 \end{align}
	 Moreover we can make $K_{L_0,P_0}$ continuous on $(L_0,P_0)$. Since $\mathcal O_\epsilon$ is compact, it follows that 
$$
K\defeq \sup_{(L,P)\in O_{2\epsilon}} K_{L,P}<\infty.
$$
This constant $K$ gives the answer to the lemma.

The second statement is an immediate consequence of the first applied to  $L_B$ and  $L_C=L_0$.  We  then get
$$
\Vert B\Vert_0\leq Kd_0(L_0,L_B)\leq K\ , 
$$ 
where we used in the second equality that $\diam(\Gk(E_0))\leq 1$. This tells us that the second statement holds for $\beta=K$.

The third statement also  follows: by continuity, there exists $a$, independent of $L_0$, such that if 
$$
d_0(L_C,L_0)\leq a\implies\hbox{ for all } P \hbox { such that }\angle_0(L_0,P)\geq 2\epsilon\ \hbox{, then } \angle_0(L_C,P)\geq \epsilon\  .
$$
In other words,  assume that 
$(L_0,P_0)\in O_{2\epsilon}$, then  
if $$d_0(L_C,L_0)\leq a\ , $$ then  we just showed that  $(L_C,P_0)$ belongs to $O_\epsilon$ hence by
inequality \eqref{ineq:metric-equiva},
$$
\Vert C\Vert_0\leq Ka\ , 
$$ 
thus $\alpha=Ka$ works and we have proved the last statement.
\end{proof}

In the sequel, we will consider a bundle situation that is a family of  Euclidean metrics $g_x$ on vector spaces $E_x$, where $x$ is a parameter and we will then define by some abuse of language  $\angle_x$ and $d_x$.

\subsubsection{Uniformly hyperbolic bundles: definitions}

We now consider the  trivial bundle $E=V\times \uh$, for $V$ a vector space, with fiber $E_x$ at a point $x$ of $\uh$. The associated Grassmann bundle over $\uh$ whose fiber at  $x$ in $\uh$ is $\Gk(E_x)$ and whose total space is denoted $\Gk(E)$.

   We recall that a frame $h$ on $E$, that is a section of the frame bundle, defines naturally a metric $g_h$ on this bundle (hence on all associated bundles): the unique metric for which the frame is orthogonal. We will use this association freely in the sequel.

For any flat connection $\nabla$ on $E$, we consider 
	the lift $\flo{\Phi^\nabla}$ of the geodesic flow $\flo{\phi}$ of $\uh$ given by the parallel transport along the orbits of $\partial_t$. When $\rm D$ is the trivial connection on $E$, we just write $\flo{\Phi}\defeq\flo{\Phi^D}$ and observe that $\Phi_t(x,v)=(\phi_t(x),v)$ where $x$ is in $\uh$ and $v$ in $V$.
\begin{definition}[\sc Uniformly hyperbolic bundle]\label{def:unif-hyp-bund}   
	A {\em rank $k$ uniformly hyperbolic  bundle} is a pair $(\nabla,h)$ where $h$ is a section of the frame bundle on $E$, $\nabla$  a trivializable connection  on the bundle $E$. We denote $\Vert \cdot \Vert$ the associated metric to the frame $h$.

	We  first assume   the (standard) {\em bounded cocycle hypothesis}: $\Vert\Phi_1^\nabla\Vert$    and $\Vert\Phi_{-1}^\nabla\Vert$ are   uniformly bounded. 	
Then we assume that we have a $\flo{\Phi^\nabla}$ invariant decomposition (not a priori continuous) at every point $x$
$$
E_x=L_x\oplus P_x\ ,
$$
where $L_x$  and $P_x$ are subspaces  with $\dim(L_x)=k$ such  that \begin{enumerate}
\item \label{it:vol} There is a volume form on $E$, which is bounded with respect to $h$ and $\nabla$-parallel along orbits on the flow.
 \item  there exist positive constant $B$ and $b$  so that for all positive real $s$, for all $x$ in $\USi$ for all non-zero vector $u$ and $v$ in $L_x$ and $P_x$ respectively 
\begin{equation}
	\frac{\Vert \Phi^\nabla_s(u)\Vert_{\phi_s(x)}}{\Vert u\Vert_x}\leq B e^{-bs} \frac{
	\Vert \Phi^\nabla_s(v)\Vert_{\phi_s(x)}}
	{\Vert v\Vert_x}\ .\label{eq:def-proj}
\end{equation}
We will refer to this condition abusively by saying that the " bundle $L\otimes P^*$ is contracting".
\item \label{hyp:3hp} Let $\mathcal L$ be the closure of the image of $L$ in $\Gk(E)$ and $\mathcal P$ be the closure of the image of $P$ in $\Gk^*(E)$, then for any  $x$ in $\uh$,  for any $\uu$ in $\mathcal L_{x}$ and $\vv$ in  $\mathcal P_{x}$, we have 
$$
\angle_x(\uu,\vv)> 2\epsilon_0 
$$
where $\epsilon_0$ is a positive constant.
\item \label{hyp:tw} Let us consider $E$ as a trivial bundle $E=V\times\uh$ over $\uh$, where $\nabla$ is the trivial connection. We finally assume that there exists  a continuous decomposition of $E$, $E=L^\infty\oplus P^\infty$ -- seen as a mapping into $\Gk(V)\times \Gk^*(V)$ -- such that for all $x$ in $\uh$, 
$$
\angle_x(L^\infty_x,P^\infty_x)> \epsilon_1 ,
$$
where $\epsilon_1$ is a positive constant, 
such that furthermore there exists constant $H$ and $\alpha$ such that for any $x$ and $y$ in $\uh$
$$
d_x(L^\infty_x,L^\infty_y)\leq H\  d(x,y)^\alpha\ ,
$$
and the same holds for $P^\infty$.\footnote{This inequality is satisfied for instance when $L^\infty$ and $P^\infty$ are uniformly $C^1$. It will only be used to prove a certain subset is nonempty; see footnote \ref{note:ne}.}
\end{enumerate}

 The {\em fundamental projector} associated to a uniformly hyperbolic bundle is the section $\pp$ of $\End(E)$ given by the projection on $L$ parallel to $P$.
\end{definition}

A natural example comes from a projective Anosov representation of a cocompact surface group.
We call such an  example, where the frame and the connections are invariant  under the action of a cocompact surface group  a {\em periodic bundle}. We discuss periodic bundles in \ref{periodic_case}.

We emphasize that we do not require a priori any continuity on the "bundles'' $L$ and $P$. Instead of a continuity hypothesis on $L$ and $P$ we use the weaker condition \eqref{hyp:3hp}. The subsets 
$\mathcal L$ and $\mathcal P$ are a priori not sections but subsets of $\Gk(E)$; In the end, ultimately we will prove that $L$ and $P$ are continuous and thus that $\mathcal L$ and $\mathcal P$ are the images of the sections $L$ and $P$.  However this continuity is not a requirement in the definition as it is not for Anosov flows.

In the sequel we will choose $\epsilon_0=\epsilon_1$.

We thank Tianqi Wang for pointing out to us that the first three hypothesis were not enough.
 When the dimension of $L_x$ is 1, we talk of a {\em projective uniformly hyperbolic bundle}, when it is $k$, we talk of a {\em rank $k$ uniformly hyperbolic bundle}.

The hypothesis \eqref{it:vol} is for simplification purposes. Using that hypothesis, one sees that $\det(L)$ and $\det(P)$ are respectively contracting and expanding line bundles.

The bounded cocycle assumption,  akin to a similar condition in Oseledets theorem, is    equivalent to the fact   that there exists positive constants  $B$ and $C$ so that 
\begin{equation}
	\Vert \Phi_s^\nabla\Vert\leq Be^{C|s|}\ .\label{eq:exp-bounded}
\end{equation}

Finally let us define a notion of equivalence for uniformly hyperbolic bundles:
\begin{definition}[\sc Equivalent bundles]
	Two uniformly hyperbolic bundles $(\nabla_0,h_0)$ and $(\nabla_1,h_1)$  are {\em equivalent} if there is a section  $B$ of $\mathsf{GL}(E)$ so that 
 \begin{enumerate}
	\item $\nabla_1=B^*\nabla_0$\ ,
	\item The metrics $g_{h_0}$ and $B^*g_{h_1}$ are uniformly equivalent.
\end{enumerate} 

\end{definition}
 
\subsubsection{Limit maps}
We start with a proposition which is the crucial result of this paragraph.

\begin{proposition}\label{pro:pbd-unique}
Let $(\nabla,h)$ be a uniformly hyperbolic bundle $E=L\oplus P$. Let us choose a trivialisation so that $\nabla$ is the trivial connection. Then 
\begin{enumerate}
	\item The fundamental projector $\pp$ is parallel along the geodesic flow and a continuous bounded section of $\End(E)$: 
	 $$\sup_{x\in\uh}\Vert \pp_x\Vert_x <+\infty\ .
	 $$	
	\item the subdundle $L$ is constant along the central  stable foliation of the geodesic flow of $\uh$, and there exist positive constants $H$ and $\beta$, such that for all $x$ and $y$ in $\uh$
	$$
	d_x(L_x,L_y)\leq H\  d(x,y)^\beta\ .
	$$
	\item Symmetrically  $P$ is constant along the central unstable foliation of $\uh$, and there exist positive constants $H$ and $\beta$, such that for all $x$ and $y$ in $\uh$
	$$
	d_x(P_x,P_y)\leq H\  d(x,y)^\beta\ .
	$$
\end{enumerate}
\end{proposition}

As an immediate corollary of this proposition, we see that $L$ and $P$ are, as promised, indeed continuous.

Before giving the proof of this proposition, observe that 
it  allows us to define the {\em limit maps} of the uniformly hyperbolic bundle $(\nabla,h)$. More precisely:  
  
\begin{definition}[\sc Limit maps]\label{def:limit-map}
	Let us choose a trivialization  $E=V\times\uh$ so that $\nabla$ is trivial. In that context, 
the {\em limit map} of the uniformly hyperbolic bundle is the  continuous map
$$\xi: \partial_\infty\hh\to \Gk(V)\ ,$$   where $\xi(\varpi^+(y))=L_y$ for any $y$ in $\uh$. Symmetrically, the  {\em  dual limit map} of the uniformly hyperbolic bundle is the  continuous map
$$\xi^*: \partial_\infty\hh\to \Gk(V^*)\ ,$$   where $\xi^*(\varpi^-(y))=P_y$, for any $y$ in $\uh$.
\end{definition}
\vskip 0.2truecm
The projections $\varpi^+$ and $\varpi^-$ are defined in section \ref{sec:geom-uh}.
We also prove 
\begin{proposition}\label{pro:holder}
	The limit maps are H\"older.
\end{proposition}
 \begin{proof} We use proposition \ref{pro:pbd-unique}.  Let $\mathsf{S}^1$  in $\uh$ be a fiber of the projection to $\hh$. Recall that $\varpi^+$ is a homeomorphism to $\mathsf{S}^1$ to $\bh$. We can then see $\xi$, by an abuse of language, as a map from $\mathsf{S}^1$ to $\Gk(V)$.  By proposition \ref{pro:pbd-unique}, for any $x$  and $y$ in  $\mathsf{S}^1$  then using this abuse of language
$$
d_x(\xi(x),\xi(y))\leq H\  d(x,y)^\beta\ .
$$
Since $\mathsf{S}^1$ is compact, given any metric $d_0$ on $\Gk(V)$ coming from a Euclidean metric on $V$, there is a constant $K$ such that for all $x$ in $\mathsf{S}^1$
$$
d_0\leq K\ d_{x}\ .
$$
It follows that for any $x$ and $y$ in $\mathsf{S}^1$,
$$
d_0(\xi(x),\xi(y))\leq KH\  d(x,y)^\beta\ .
$$
Hence $\xi$ is H\"older. The same holds for $\xi^*$. 
\end{proof}

We start the proof of proposition \ref{pro:pbd-unique} by two lemmas. \begin{lemma}\label{lem:U-V}
	Let $(\nabla,h)$ be a uniformly hyperbolic bundle. Then there exists closed sets $\mathcal V$ and $\mathcal U$ of $\Gk(E)$ and $\Gk(E^*)$,  respectively, where $k$ is the dimension of $L_x$,  as well as a positive real $T$, so that
	\begin{enumerate}
	\item $L$ and  $L^\infty$  are sections with values in $\mathcal V$, while $P$ and $P^\infty$ are  sections with values in $\mathcal U$
	\item  For every $\vv$ in $\mathcal V_x$, $\vv$ is transverse to $P_x$, for every $\uu$ in $\mathcal U_x$, $\uu$ is transverse to $L_x$, 
 \item $\Phi_{-T}$ sends $\mathcal V$ to $\mathcal V$ and is $1/2$-Lipschitz: for any $x$ in $\uh$, $\uu$ and $\vv$ in $\mathcal V_x$, then 
		$$
		d_{\phi_{-T}(x)}(\Phi_{-T}(\uu),\Phi_{-T}(\vv))\leq \frac{1}{2} d_x(\uu,\vv)\ .
		$$

		\item 	Symmetrically $\Phi_T$ sends $\mathcal U$ to $\mathcal U$ and is $1/2$-Lipschitz. 
		 \end{enumerate}
\end{lemma}
   \begin{proof} 
Observe first that $\mathcal L$ and $\mathcal P$ are also invariant by $\flo{\Phi}$.   Let  
$$
\mathcal O_{\epsilon}\defeq \{(\uu,\vv)\in \Gk(E)\boxtimes\Gk(E^*)\mid \angle(\uu,\vv)\geq\epsilon\}\ .
$$
The set $\mathcal O_{\epsilon}$ is the bundle version of the set $O_{\epsilon}$ introduced in lemma \ref{lem:equi}. Let's now choose $\epsilon_0$, where $\epsilon_0$ is as in hypothesis \eqref{hyp:3hp}.
By construction
$$
\Phi_T(\mathcal L\boxtimes \mathcal P)=\mathcal L\boxtimes \mathcal P\subset \mathcal O_{\epsilon_0}\ .
$$
We consider the function on $\Gk(E)$ defined by 
	$$
	\uu\mapsto \angle(\uu,\mathcal P)\defeq \inf\left\{\angle_x(\uu,\vv)\mid x=\pi(\uu)\ , \ \vv\in \mathcal P_x\right\} \ ,
	$$
and symmetrically the function 
	$$
	\vv\mapsto\angle(\vv,\mathcal L)\defeq \inf\left\{\angle_x(\uu,\vv)\mid x=\pi(\vv)\ ,  \uu\in \mathcal L_x\right\}\ ,
	$$
	where again $x$ in $\uh$ is the basepoint of $\vv$.
	These two functions are continuous and we define  
	$$
	\mathcal V \defeq\left\{{\vv} \in \Gk(E)\mid \angle(\uu,\mathcal P)\geq \epsilon_0\right\} \ , \ 	{\mathcal U}\defeq\left\{{\uu}\in   \Gk(E^*) \mid \angle(\uu,\mathcal L) \geq \epsilon_0\right\}\ . 
	$$
	By continuity of the  defining functions, $\mathcal V$ and $\mathcal U$ are closed sets in the bundle $\Gk(E)$.

	By hypothesis \eqref{hyp:3hp}, $L$ and $P$ are sections with values in $\mathcal V$ and $\mathcal U$ respectively and by hypothesis \eqref{hyp:tw} $L^\infty$ and $P^\infty$ as well. Thus $\mathcal U$ and $\mathcal V$ satisfy the first condition, as well as the second.

To prove the  third  property, we  will use lemma \ref{lem:equi} for the constant $\epsilon_0$ and consider the constants $K$,  $\alpha$ and $\beta$ that this lemma provides. We  will concentrate on the subset $\mathcal V$, the subset  $\mathcal U$ being treated symmetrically.

Observe that 
$$
\mathcal V_x=\bigcap_{P\in \mathcal P_x}O^+_{\epsilon_0}(P)\ ,
$$
and thus
$$
O^-_{\epsilon_0}(L_x)\subset \mathcal V_x \subset O^+_{\epsilon_0}(P_x)\ .
$$
Recall that $(L,P)$ are sections of $\mathcal O_{2\epsilon_0}$, thus it follows from the last two statements of lemma \ref{lem:equi} that if $B$ is an element of $L_y^*\otimes P_y$, for any $y$ in $\uh$, 
then 
\begin{align}
\Vert B\Vert_y \leq \alpha&\hbox{, then } L_B\in 	\mathcal V_y\ ,\label{eq:Uinclud0}\\
L_B\in 	\mathcal V_y&\hbox{, then } \Vert B\Vert_y \leq \beta\ ,\label{eq:Uinclud1}
\end{align}
Next, by the contraction property of 
 $L^*\otimes P$ in negative times  there exists $T_0$ such that for all $T$ greater than $T_0$, 
 \begin{align}
	\|\Phi_{-T}^* (C)\|_{\phi_{-T}(x)}&  \leq  \inf\left(\frac{\alpha}{\beta}, \frac{1}{2K^2}\right)\|C\|_x\ .\label{ineq:contract00}
\end{align}
It follows first that from the two implications \eqref{eq:Uinclud0} and  \eqref{eq:Uinclud1} applied to $y=x$ and $y=\phi_T(x)$ that 
\begin{align}
\Phi_{-T}(\mathcal V_x)\subset \mathcal V_{\phi_{-T}(x)}\ .\label{eq:Uinclud00}	\end{align}

We now apply the first statement of Lemma \ref{lem:equi} to get for $\uu$ and $\vv$ elements of $\mathcal V_y$, graphs respectively of $B_y$ and $C_y$,
\begin{align*}
\frac{1}{K}\ d_y(\uu,\vv)&\leq \Vert B_y-C_y\Vert_y \leq  K\  d_y(\uu,\vv)\  .
\end{align*}
In particular, since $\uu$ and $\vv$, as well as  $\Phi_{-T}(\uu)$ and $\Phi_{-T}(\vv)$ are respectively elements of $\mathcal V_x$ and $\mathcal V_{\phi_{-T}(x)}$  by the inclusion \eqref{eq:Uinclud00}, this  inequality yields  
\begin{align*}
\Vert B_x-C_x\Vert_x &\leq  K\  d_x(\uu,\vv)\ , \\ \frac{1}{K} d_{\phi_{-T}(x)}(\Phi_{-T}(\uu),\Phi_{-T}(\vv)) &\leq \|\Phi_{-T}^*(B_x-C_x)\|_{\phi_{-T}(x)}\ .
\end{align*}
Combining these two inequalities with  the contraction property \eqref{ineq:contract00}, it follows that 
$$d_{\phi_{-T}(x)}(\Phi_{-T}(\uu),\Phi_{-T}(\vv)) \leq K\|\Phi_{-T}^*(B_x-C_x)\|_{\phi_{-T}(x)} \leq \frac{1}{2K}\|B_x-C_x\|_x \leq \frac{1}{2}d_x(\uu,\vv).$$
This yields the third property.
A symmetric argument for $\mathcal U$, obtained by reversing the role of $L_0$ and $P_0$ as well as the time,  completes the proof of  the   fourth property.\end{proof}
  
Our second lemma is 
\begin{lemma}\label{lem:pbd-unique}
Let $(\nabla,h)$ be a uniformly hyperbolic bundle $E=L\oplus P$. Let us choose a trivialisation so that $\nabla$ is the trivial connection. Then 
\begin{enumerate}
	\item There exist positive constants $H$ and $\gamma$, such that for all $x$ and $y$ in $\uh$
	$$
	d_x(L_x,L_y)\leq H\  d(x,y)^\gamma\ .
	$$
	\item There exist positive constants $H$ and $\gamma$, such that for all $x$ and $y$ in $\uh$
	$$
	d_x(P_x,P_y)\leq H\  d(x,y)^\gamma\ .
	$$
\end{enumerate}
\end{lemma}
  
\begin{proof} We follow, and adapt, the classical arguments of Hirsch--Pugh--Shub \cite{Hirsch:bxQiQXnw}. Let $\Gamma(\Gk(E))$ be the space of (non necessarily continuous) sections of $\Gk(E)$. Using the trivialisation of $\nabla$ we will see alternatively $\Gamma(\Gk(E))$ as the space of maps from $\uh$ to $\Gk(V)$, denoted $\operatorname{F}(\uh, \Gk(V))$.  
  For any $x$ in $\uh$, let $d_x$ be the associated distance on $\Gk(E_x)=\Gk(V)$ coming from the frame. For any section $\sigma_0$ and $\sigma_1$ of $\Gk(E)$, let
$$
d_\infty(\sigma_0,\sigma_1)\defeq\sup\{d_x(\sigma_0(x),\sigma_1(x))\mid x\in \uh\}\ .
$$
Let $\mathcal U$, $\mathcal V$ and  $T$ as in lemma \ref{lem:U-V}.
Let $\Psi$ be the map from  $\Gamma(\Gk(E))$ to itself defined by 
$$
[\Psi(\sigma)](x)=\Phi_{T}(\sigma\circ\phi_{-T}(x))\ . 
$$

Let $\Gamma(\mathcal U)$, with $\Gamma(\mathcal U)\subset \Gamma(\Gk(E))$ be the space of  sections of $\mathcal U$. Equipped with $d_\infty$, $\Gamma(\mathcal U)$ is a complete metric space since $\mathcal U$ is closed. 
Lemma \ref{lem:U-V} implies that $\Psi$ sends $\Gamma(\mathcal U)$ to itself and that 
$$
d_\infty (\Psi(\sigma_0),\Psi(\sigma_1))\leq \frac{1}{2}d_\infty (\sigma_0,\sigma_1)\ .
$$

\vskip 0.2 truecm
\noindent{\em Step 1: A space of  H\"older sections.}
Let $\gamma$ with $0<\gamma\leq\alpha$, where $\alpha$ is the H\"older constant appearing in hypothesis \eqref{hyp:tw}.  For any $\sigma$ a section of $\Gk(E)$,  let 
\begin{align*}
V_{\gamma}(\sigma)\defeq\sup_{x,y\in \uh}\left\{
\frac{d_x(\sigma(x),\sigma(y))}{d(x,y)^\gamma}\mid  0<  d(x,y)\leq 1
\right\}\ ,
\end{align*}
and finally 
\begin{align*}
 \Gamma^{\gamma,H}
(\mathcal U)&\defeq\left\{\sigma\in \Gamma(\mathcal U)\mid V_\gamma(\sigma)\leq H\right\}\ ,\end{align*}
where $H\geq 1$.
\begin{enumerate}
	\item If $d(x,y)\geq 1$, then for any $\sigma$ in $\Gamma(U)$
	 $$
	d_x(\sigma(y),\sigma(x))\leq 1\leq  H\  d(x,y)^\gamma\ ,
	$$
	where the first inequality comes from the fact that $\diam(\Gk(E_x))=1$ and the second one is because $H\geq 1$.
	 \item Assume  $\sigma$ belongs to $\Gamma^{\gamma,H}(\mathcal U)$, then for all $x$, $y$ with $x\not=y$ and $d(x,y)\geq 1$, we  have 
\begin{align*}
d_x(\sigma(y),\sigma(x))\leq H\  d(x,y)^\gamma\ .	
\end{align*}
\end{enumerate} 
Thus, combining both cases, when  $\sigma$ belongs to $\Gamma^{\gamma,H}(\mathcal U)$, then for all $x$, $y$ with $x\not=y$, we  have 
\begin{align}
d_x(\sigma(y),\sigma(x))\leq H\  d(x,y)^\gamma\ .	\label{ineq:sxy}
\end{align}

Observe that  hypothesis \eqref{hyp:tw} implies that $L^\infty$ belongs to $\Gamma^{\gamma,H}(\mathcal U)$  which is therefore non empty \footnote{That is the only point where hypothesis \eqref{hyp:tw} is used \label{note:ne}}.

\vskip 0.2 truecm
\noindent{\em Step 2: Closedness of the space of sections.} We first prove that $\Gamma^{\gamma,H}(\mathcal U)$ is  closed in  $\Gamma(\mathcal U)$ with respect to $d_\infty$. Let $\seq{\sigma}$ be a sequence of sections in $\Gamma^{\gamma,H}(\mathcal U)$ converging to $\sigma$. Then for all $x$ and $y$ in $\uh$ such that  $0<d(x,y)\leq 1$, 
\begin{align*}
\frac{d_x(\sigma(x),\sigma(y))}{d(x,y)^\gamma}=\lim_{m\to\infty}\frac{d_x(\sigma_m(x),\sigma_m(y))}{d(x,y)^\gamma}\leq \lim_{m\to\infty}V_{\gamma}(\sigma_m)\leq H . 
\end{align*}
Thus $\Gamma^{\gamma,H}(\mathcal U)$ is closed in $\Gamma(\mathcal U)$ for $d_\infty$.
\vskip 0.2 truecm
\noindent{\em Step 3: Action of $\Psi$.} We want to show that there exists some $\gamma_0$ less than $\alpha$ such that  the set  $\Gamma^{\gamma_0,H}(\mathcal U)$ is invariant by $\Psi$. It is enough to prove that there  exists 
 $\gamma_0$ less than $\alpha$ 
 such that for all $\sigma$, \begin{equation}
V_{\gamma_0}(\Psi(\sigma))\leq  V_{\gamma_0}(\sigma)\ .		
\label{ineq:vpsi2}	
\end{equation}
Let  $k$ be the Lipschitz constant of $\phi_{-T}$.
Let $x$ and $y$ be such that $0<d(x,y)\leq 1$. Then $d(X,Y)\leq k$, where $X\defeq \phi_{-T}(x)$ and $Y\defeq \phi_{-T}(y)$. Moreover
\begin{align*}
d_{x}([\Psi(\sigma)](x),[\Psi(\sigma)](y))\leq \frac{1}{2}d_X(\sigma(X),\sigma(Y))&\leq \frac{1}{2} V_\gamma(\sigma)\  d(X,Y)^\gamma
\leq \frac{k^\gamma}{2}  V_\gamma(\sigma)\  d(x,y)^\gamma\ .
\end{align*}
Thus
$$
V_\gamma(\Psi(\sigma))\leq \frac{k^\gamma}{2} V_\gamma(\sigma)\ .
$$
We now choose $\gamma_0$, less than $\alpha$,  such that 
$ k^{\gamma_0}\frac{1}{2}\leq 1
$.
This concludes the proof. 
\vskip 0.2 truecm
\noindent{\em Step 4: Conclusion.}  We apply Banach fixed point theorem  to the closed  $\Psi$-invariant set $\Gamma^{\gamma_0,H}(\mathcal U)$ with the metric $d_\infty$.  It follows that   there is a  $\Psi$-invariant section $\sigma_0$ in $\Gamma^{\gamma_0,H}(\mathcal U)$. As $\Psi$ is a contraction for $d_\infty$ on $\Gamma(\mathcal U)$, $\sigma_0$ is the unique fixed point for $\Psi$ in $\Gamma(U)$. Thus it follows that $\sigma_0 = L$ since $L$ is by definition a fixed point of $\Psi$ in $\Gamma(\mathcal U)$. In particular $L$ belongs to $\Gamma^{\gamma_0,H}(\mathcal U)$ and is Hölder.   The same arguments works for $P$.  This concludes the proof. \end{proof}
We can now proceed to the proof of the proposition
\begin{proof}[Proof of proposition \ref{pro:pbd-unique}]
We choose again a trivialisation of $E$ such that $\nabla$ is the trivial connection. In this trivialisation
$$
\Phi^\nabla_t(x,u)=(\phi_t(x),u)\ .
$$
 Since $L$ is invariant -- and in this gauge constant  -- under the flow, in order to prove that $L$ is constant on the central stable leaf it is enough to show that $L$ is constant on a strong stable leaf. 
 
 As $L^*\otimes P$ is contracting in the negative direction, there exists a $T>0$ such that for $C$ in $L^*\otimes P$ and $x \in \uh$,
 \begin{align}
 	\|C\|_{x} &\leq \frac{1}{2}\|C\|_{\phi_{T}(x)}\ .\label{ineq:214-00}
 \end{align}
Let  $x$ and $y$ be on the same strong stable leaf.  Since $L$ and $P$ are continous by lemma \ref{lem:pbd-unique}, there exist a neighborhood $N_x$ of $x$ such that  for every $y$ in $N_x$ 
$$
L_y\pitchfork P_x\ .
$$
It follows that $L_y$ is the graph of some endomorphisms $B$ from $L_x$ to $P_x$.

Since $x$ and $y$ are on the same strong stable leaf,  there exists a positive constant $k_0$, with $k_0<1$, such that 
$$
d(x_n,y_n)\leq k_0^n\ .
$$
where $x_n=\phi_{nT}(x)$ and $y_n=\phi_{nT}(y)$. Then by lemma \ref{lem:pbd-unique} there exist positive $\gamma$ and $H$ such that 
\begin{align}
d_{x_n}(L_{x_n},L_{y_n})\leq H k_0^{\gamma n}\ .	\label{eq:xnyn-hold}
\end{align}
By the definition of uniformly hyperbolic bundles there exists an $\epsilon_0 >0$ such  that $\angle_{x_n}(L_{x_n},P_{x_n}) \geq 2\epsilon_0$. From inequality \eqref{eq:xnyn-hold}, it follows that  there exists some $p$ such that for $n>p$, 
$$
\angle_{x_n}(L_{y_n},P_{x_n})\geq \epsilon_0\ .
$$
In other words
\begin{align}
L_{y_n}&\in O_{\epsilon_0}^+(P_{x_n})\ .	\label{ineq:Oepsilon214}
\end{align}
By the same argument as above  $L_{y_n}$ is the graph of an element $B_n$ from $L_{x_n}$ to $P_{x_n}$.
Recall now that $L_{x_n}=L_x$, $P_{x_n}=P_{x}$ and $L_{y_n}=L_y$. Thus $B_n=B$. 
Then, from assertion \eqref{ineq:Oepsilon214} and Lemma \ref{lem:equi} (for $\epsilon= \epsilon_0$), it follows that there exists a positive constant $\beta$ such that for $n$, with $n>p$, 
$$
\Vert B\Vert_{x_n}\leq \beta\ .
$$

Combining this bound with  the contraction property of uniformly hyperbolic as in  Inequality \eqref{ineq:214-00}, we have  for $n>p$
$$
\Vert B\Vert_{x}\leq \frac{1}{2^{n}}\Vert B\Vert_{x_n} \leq \frac{1}{2^{n}}\beta \ .
$$
Since this is true for every  $n$ greater than $p$, we have that $B=0$, thus $L_y=L_x$ for all $y$ in $N_x$. Since strong  stable leaves are connected, we have that $L_y=L_x$ on the whole  central stable leaves of $x$.  
The same argument works for $P$. This concludes the proof.
\end{proof}

\subsection{Families of uniformly hyperbolic bundles and their variations}

Let $E$ be a vector bundle. A frame is a section of the frame bundle of $E$, observe that a frame $h$ defines canonically a metric $g_h$ on $E$ such that  the frame is orthonormal for  this metric.

We recall that a family of connections $t\mapsto\nabla^t$ for $t\in ]-\epsilon,\epsilon[$ is a {$C^k$-family} when writing $\nabla^t=\nabla^0+B^t$, then for every compact $K$ in $\uh$ then the map $t\mapsto B^t(\ps)$ is $C^k$ as a map with values in the Banach space $C^0(K, \End(V))$. We will now assume that $k\geq 1$.

\begin{definition}[\sc Bounded variation in the fixed frame  of view]\label{def:bd-var}
	A {\em $C^k$-bounded variation on a neighborhood of $0$} of a uniformly hyperbolic bundle $(\nabla,h)$ is a $C^k$-family of connections  $(\nabla^t)_{t\in ]-\epsilon,\epsilon[}$ such  that 
	\begin{enumerate}
		\item $\nabla^0=\nabla$,
		\item the maps  $
		t\mapsto  \dot\nabla_\ps $ are uniformly bounded on $\uh$ with respect to $g_{h}$ for $t$ in a neighborhood of zero.
	\end{enumerate}
\end{definition}

We then extend naturally this definition to $C^k$-bounded variation on some interval $I$, by saying it is a $C^k$ variation at $t=s$ for any $s$ in $I$.
 
We want to make this definition ''gauge invariant''. For that we use the following definition

\begin{definition}[\sc Bounded variation]\label{def:bd-var}
	A {\em $C^k$-bounded variation on a neighborhood of  $0$} of a uniformly hyperbolic bundle $(\nabla,h)$ is a family  $(\nabla^t,h_t)_{t\in ]-\epsilon,\epsilon[}$ where  $(\nabla^t)_{t\in ]-\epsilon,\epsilon[}$ is a smooth family of connection  and $(h_t)_{t\in ]-\epsilon,\epsilon[}$ is a smooth family of frames, such that there exists a smooth family of gauge transformations  $(G_t)_{t\in ]-\epsilon,\epsilon[}$ such that 
	\begin{enumerate}
		\item $G_t^*h_t=h$,
		\item The family $(G_t^*\nabla^t)_{t\in ]-\epsilon,\epsilon[}$ is a $C^k$-bounded variation on a neighborhood of  $0$ in the fixed frame point of view.
	\end{enumerate}
\end{definition}

Obviously if $(\nabla^t,h_t)_{t\in ]-\epsilon,\epsilon[}$ is a $C^k$-bounded variation, and $(H_t)_{t\in ]-\epsilon,\epsilon[}$ a smooth family of gauge transform -- where we recall that a gauge transform is just a section of the automorphism bundle -- then $(H_t^*\nabla^t,H_t^*h_t)_{t\in ]-\epsilon,\epsilon[}$ is itself of bounded variation. 

We will see that  any smooth family of periodic bundles is of bounded variation in Proposition \ref{pro:periodic}.

Thus in order to study $C^k$-bounded variation of uniformly hyperbolic bundles, we will adopt alternatively two different gauge-fixing points of view:
\begin{enumerate}
	\item The {\em fixed gauge point of view}: we allow the frame to vary but we fix the connection,
	\item The {\em fixed frame point of view}: we allow the connection to vary but we fix the frame.
\end{enumerate}

The following theorem is the analog  for uniformly hyperbolic bundles of the stability of Anosov representations.

\begin{theorem}[{\sc Stability}]\label{theo:stab}
	Assume that $\fam{\nabla^t,h}$ is a $C^1$-bounded variation at $t=0$ of a uniformly hyperbolic bundle. Then for $t$ in some neighbourhood of zero, the bundle  $(\nabla^t,h_t)$ is uniformly hyperbolic.
	
	Assume now that we have a $C^k$-bounded variation. Let $\pp_t$ be the  associated projector,  then for every $y$ in $\uh$, $y\mapsto \pp_t$   is $C^{k-1}$-derivable as a function $t$ at $0$. 
	\end{theorem}

\begin{proof} This result is gauge invariant, so let's us  choose a smooth family of gauge transformations $\fam{G}$ so that $G_t^*h_t=h$.  Let us write
$$
\nabla^t=\nabla + B^t\ .
$$
Let $(\Phi_s^t)_{s\in\mathbb R}$ be the parallel transport along the orbits of the geodesic flow for $\nabla^t$, which involves solving the linear differential  equation  
$$
\frac{\rm d}{\rm ds} U= -B^t(\ps) U\ .
$$
For a $C^1$-bounded variation we have in particular, that 
for a given $T$, for any $\alpha$, there exists $\beta$ so that $\vert u\vert\leq \beta$, implies that uniformly in $\uh$,
$$
\Vert \Phi_T-\Phi_T^u\Vert \leq \alpha\ . 
$$
Thus from lemma \ref{lem:U-V}, since $\Phi_T$ is 1/2-Lipshitz, then  for $u$ small enough, $\Phi_{T}^u$ preserves $\mathcal U$ and is $3/4$-Lipschitz, while the same holds for $\Phi_{-T}^u$ and $\mathcal V$. Thus $\Phi^u_T$ has an attracting fixed point $P^u$ in the Banach manifold $\Gamma(\mathcal U)$ (modelled on sections of a vector bundle with the $C^0$-norm). Symmetrically, $\Phi^u_{-T}$ has an attracting fixed point $L^u$ in $\Gamma(\mathcal V)$.

We have that as $u$ converges to $0$, $L^u$ converges uniformly to $L$   and $P^u$ converges uniformly to $P$. It also  follows that for $u$ possibly even smaller 
$$
\angle(L^u,P^u)\geq\frac{1}{2}\epsilon_0\ .
$$
Using again lemma \ref{lem:equi} we have that the bundle $L^u\otimes (P^u)^*$ is $K$-contracting for $\Phi^u_T$ while  $P^u\otimes (L^u)^*$ is $K$-contracting for $\Phi^u_{-T}$, where $K<1$.

Finally, since condition $\eqref{hyp:tw}$ is a condition only involving $h$. This condition is also satisfied for $(\nabla^u,h)$. 

As a conclusion  of this first discussion, we have that if we have a $C^1$-bounded variation of a uniformly hyperbolic bundle $(\nabla,h)$ then $(\nabla^u,h)$ is uniformly hyperbolic for $u$ small enough.

\smallskip
We now move to the last part of the statement. To obtain this result, we first use Cauchy--Lipschitz Theorem on the dependance of the solutions of ODE to obtain that on the Banach manifold $\mathbb R\times\Gamma(\mathcal U)$, the function $(u,\sigma)\mapsto (\Phi_T^u)^*(\sigma)$ is a $C^{k}$-function of $u$ and $\sigma$. Observe that $\Vert\Phi^u_T\Vert$ is less than 3/4, then we can apply the Implicit Function Theorem to the function $(u,\sigma)\mapsto (\Phi_T^u)^*(\sigma)-\sigma$ to obtain the that fixed point of $\Phi^u_T$ varies $C^k$ in $u$. It follows that $L^u$ varies $C^k$ on $u$, the same holds for $P^u$, and hence the fundamental projector $\pp^u$.  
\end{proof}
\subsection{$\Theta$-Uniformly hyperbolic bundles}\label{sec:Theta}
We now generalize the situation described in the previous paragraphs, using the same notational convention. Let $V$ be a finite dimension vector space, let $\Theta=(K_1,\ldots,K_n)$ be a strictly increasing $n$-tuple so that 
$$
1\leq K_{1}<\ldots< K_{n} < \dim(V)\ .
$$
Then a {\em $\Theta$-uniformly hyperbolic bundle} over $\uh$ is given  by a pair $(\nabla,h)$  for which there exists a $\flo{\Phi^\nabla}$- invariant decomposition 
$$
E_0=E_1\oplus\ldots \oplus E_{n+1}\ , 
$$
so that $(\nabla,h)$  is uniformly hyperbolic of rank $K_\aa$ for all $\aa$ in $\{1,\ldots,n\}$ with invariant decomposition given by 
$$
E_0=F_\aa\oplus  F^\circ_\aa\ , \hbox{ with } F_\aa=E_1\oplus\ldots\oplus  E_\aa\ , F^\circ_{\aa}=E_{\aa+1}\oplus \ldots \oplus E_{n+1}\ .
$$
The flag $(F_1,\ldots,F_n)$ will be called a {\em $\Theta$-flag}.

In other words, we generalized the situation described before for Grassmannians to flag varieties. We leave the reader check that the Stability Theorem holds in that generalization.

\subsection{The periodic case} \label{periodic_case} 
We start with a definition.

\begin{definition}
	Given $(\nabla,h)$ a uniformly hyperbolic bundle,   $(\nabla,h)$ is {\em periodic under $\Gamma$} if $\Gamma$ is a group of isometries of $\hh$ such that for any $\gamma$ in $\Gamma$, we have $\gamma^*\nabla=\nabla$ and $\gamma^*h=h$.
	
	The uniformly hyperbolic bundle is {\em periodic} if there  exists a cocompact $\Gamma$ in the  group of isometries of $\hh$ such that the uniformly hyperbolic bundle is periodic under $\Gamma$. 
\end{definition}

We relate periodic bundles to Anosov representations: let $\Sigma$ be the universal cover of a closed surface $S$.
We denote by $\pi$ the projection from $\Sigma$ to $S$ and $p$ the projection from $\USi$ to $\Sigma$.

Let $\Gamma$ be the fundamental group of $S$ and $\rho$ be a projective Anosov representation of $\Gamma$ on some vector space $V$. Let $E_1$ be the associated flat bundle on $S$ with connection $\nabla$.

We will use in the sequel the associated trivialization of the bundle  $E=p^*\pi^*E_1$ on which $\nabla$ is trivial. Let us choose a  $\Gamma$-invariant Euclidean metric $g$ on the bundle $E$. Let us finally choose a orthonormal frame $h$ for $g$ so that $g=g_h$.

It follows from the definition of projective Anosov representations that the corresponding bundle $(\nabla,h)$ is uniformly hyperbolic. 

More generally, let $\mathsf P_\Theta$ be the parabolic group stabilizing a $\Theta$-flag. Then a $\mathsf P_\Theta$-Anosov representation defines a $\Theta$-uniformly hyperbolic bundle. Conversely any periodic $\Theta$-uniformly hyperbolic bundle is the lift of a $\Theta$-Anosov representation.

We have the following result.

\begin{proposition}\label{pro:periodic}We have

\begin{itemize}
\item Given an Anosov representation $\rho$, a different choice of a $\Gamma$-invariant metric yields an equivalent uniformly hyperbolic bundle.
\item Similarly, two conjugate representations give equivalent uniformly hyperbolic bundles.
\item A $C^\infty$ deformation of an Anosov representation gives a $C^\infty$-bounded variation.
\end{itemize}
\end{proposition}
\begin{proof} 
The first two statements follow immediately from the definitions and constructions. For the  second one we use the classical fact that, given a smooth family $\fam{\rho}$ of representations of the fundamental group  one can find (in the neighborhood of zero) a smooth family of flat connections 	$\fam{\nabla}$ such that the holonomy of $\nabla_t$ is 
$\rho_t$. 
\end{proof}

\section{The fundamental projector and its variation}

Our goal is to compute the variation of  the associated family of fundamental projectors of a bounded variation of  a uniformly hyperbolic bundle. More precisely, let assume we have a uniformly hyperbolic bundle $(\nabla^0,h_0)$ with decomposition 
$$
E=L_0\oplus P_0\ .
$$
We prove in this paragraph the following proposition, in which we use Theorem \ref{theo:stab}.
\begin{proposition}[\sc Variation of the fundamental projector]\label{pro:deriv-fund-proj} Assume that we have a  $C^2$-bounded variation  $\fam{\nabla^t,h}$ of the uniformly hyperbolic bundle $(\nabla^0,h_0)$ in the fixed connection point of view, that is $\nabla^t$ is the trivial connection $\rm D$.

	The derivative of the fundamental projector at a point $x$ in a geodesic $g$, is given by 
	\begin{equation}
\dot\pp_0=[\dot A,\pp_0] + \int_{g^+} [\d\dot A,\pp_0] \cdot\pp_0+ \int_{g^-} \pp_0\cdot [\d\dot A,\pp_0]\  .\label{eq:deriv-proj2}\
\end{equation}
where $g^+$ is the geodesic arc from $x$ to $g(+\infty)$ and $g^-$ is the arc from $x$ to $g(-\infty)$ (in other words with the opposite orientation to $g$), and $\dot A$ is the endormorphism so that 
$$
\dot A\cdotp h =\left.\frac{\partial}{\partial t}\right\vert_{t=0} h_t\ . 
$$
\end{proposition}

\subsubsection{Preliminary: subbundles of $\End(E_0)$}
Let $\nabla$ be a flat connection on $E_0$, 
Then $\pp$ is parallel for the induced flat connection on $\End(E_0)$ along the flow. Let also $F_0$ be the subbundle of $\End(E_0)$ given by
$$
F_0\defeq\{B\in\End(E_0)\mid B{\pp}+{\pp} B=B\}\ .
$$
Observe that for any section  $C$ of $\End(E_0)$, $[C,\pp]$ is a section of $F_0$ and that for any element $A$ in $F_0$ we have $\tr(A)=0$. 
\begin{lemma}
	\label{lem:Ano-subbun}
	The bundle $F_0$ decomposes as two parallel subbundles
	\begin{equation}
		F_0=F_0^+\oplus F_0^-\ , \label{def:decomp1}
	\end{equation}
 where we have the identification
	\begin{eqnarray}
			F_0^+=P^*\otimes L \ & ,&  	F_0^-=L^*\otimes P\ .\label{eq:identifF0}
	\end{eqnarray}

	The projection of $F_0$ to $F_0^+$ parallel to $F_0^-$ is given by
	$B\mapsto \pp B$, 
	while the projection on 
	 $F_0^-$ parallel to $F_0^+$ is given by 
	 	$B\mapsto  B \pp$.

Finally there exists positive constants $A$ and $a$, so that for all positive  time $s$, endomorphisms $u^+$ in  $F_0^+$ and $u^-$ in  $F_0^-$, we have 
	\begin{eqnarray}
			\Vert \Phi_{-s}(u^-)\Vert \leq A e^{-as} \Vert u^-\Vert\ \ ,\ \  \Vert \Phi_{s}(u^+)\Vert \leq A e^{-as} \Vert u^+\Vert\ .\label{ineq:contractF0}
	\end{eqnarray}
\end{lemma}
Consequently, for any section $f$ of $F_0$, we write $f=f^++f^-$ where $f^\pm$ are sections of $F_0^\pm$ according to  the decomposition \eqref{def:decomp1}.

\begin{proof} Let us write 
\begin{equation*}
	\End(E_0)=E_0^*\otimes E_0=(L^*\otimes L)\oplus (P^*\otimes P)\oplus (L^*\otimes P)\oplus (P^*\otimes L) , 
\end{equation*} 
In that decomposition, 
$
F_0=(P^*\otimes L)\oplus(L^*\otimes P)$.
Let 
$$F_0^+=P^*\otimes L\ \ , \ \ F_0^-=L^*\otimes P\ .$$ 
Thus, we can identify $F_0^+$ as the set of elements  whose image lie in $L$ and $F_0^-$ are those whose kernel is in $P$. Thus
\begin{eqnarray*}
F_0^+&=&\{B\in F_0\mid \pp B=B\}=
\{B\in F_0\mid B\pp=0\}\ ,\\
F_0^-&=&\{B\in F_0\mid \pp B=0\}=\{B\in F_0\mid  B\pp=B\}\ .	\end{eqnarray*}
Then the equation for any element $B$ of $F_0$,
\begin{equation}
	B=\pp B+ B\pp \ ,\label{eq:projpm1}
\end{equation}
corresponds to the decomposition $F_0=F_0^+\oplus F_0^-$. Thus the projection on $F_0^+$ is given by $B\mapsto \pp B$, while the projection on  $F_0^-$ is given by $B\mapsto  B\pp$. 

The definition of $F_0^+$ and $F_0^-$ and the corresponding contraction properties  of the definition of a uniformly hyperbolic bundles give the contraction properties on  $F_0^+$ and $F_0^-$.
\end{proof}

\subsubsection{The cohomological equation}

\begin{proposition}\label{pro:cohom}
	Let $\sigma$ be a bounded  section of $F_0$, then there exists a unique section $\eta$ of $F_0$ so that $\nabla_{\ps} \eta=\sigma$. This section $\eta$ is given by
	\begin{equation}
			\eta(x)=\int_{-\infty}^0 \pp\cdot\sigma(\phi_s(x))\ \d s-\int_0^\infty \sigma(\phi_{s}(x))\cdot\pp \ \d s\ .\label{eq:formeta}
	\end{equation}

\end{proposition}
Classically, in dynamical systems, the equation $\nabla_{\ps} \eta=\sigma$ is called the {\em cohomological equation}.
\begin{proof}
	Since $\pp\sigma$ belongs to $F_0^+$ while $\sigma\pp$ belongs to  $F_0^-$, by lemma \ref{lem:Ano-subbun}, the right hand side of equation \eqref{eq:formeta} makes sense using the exponential contraction properties given in the inequalities \eqref{ineq:contractF0}. Indeed, for a positive $s$ by lemma \ref{lem:exp-decay} again,
\begin{eqnarray*}
	\Vert \Phi_{-s}\left(\sigma(\phi_s) \cdot \pp\right) \Vert&\leq& Ae^{-as} \Vert \sigma\Vert_\infty\ ,\\
		\Vert \Phi_{s}\left(\pp  \cdot \sigma(\phi_{-s})\right) \Vert&\leq& Ae^{-as} \Vert \sigma\Vert_\infty\ .
\end{eqnarray*}
 It follows that using the above equation as a definition for $\eta$ we have
	$$
	\eta(\phi_s(x))=\int_{-\infty}^s\pp\cdot\sigma(\phi_u(x)) \d u-\int_s^\infty \sigma(\phi_{u}(x))\cdot  \pp\ \ \d u\ .
	$$
	Thus
	$$
	\nabla_\ps \eta=\pp\sigma +\sigma\pp=\sigma\ ,
	$$
	since $\sigma$ is a section of $F_0$. Uniqueness follows from the fact that $F_0$ has no parallel section: indeed neither $F_0^+$ nor $F_0^-$ have a parallel section. 
\end{proof}

\subsubsection{Variation of the fundamental projector: metric gauge fixing}

We continue to adopt  the variation of connection point of view and consider after gauge fixing only hyperbolic bundles where the metric is fixed.

Let $(\nabla^t,h)_{t \in ]-\epsilon,\epsilon[}$ give rise to a bounded variation of the uniformly hyperbolic bundle $(\nabla^0,h)$, where $\nabla^0$ is the trivial connection ${\mathrm D}$. 

Our first result is   
\begin{lemma}\label{lem:var-proj0}

The variation of the fundamental projector $\pp_t$  associated to $(\nabla^t,h)$ is given by \begin{equation}
		\dot\pp(x)=\int_{-\infty}^0 \left( \pp\cdot [\pp,\dot\nabla_{\ps}] \right)(x^s)\ \d s- \int_0^\infty \left([\pp,\dot\nabla_{\ps}]\cdot \pp \right)(x^s)\ \d s\ \label{eq:ppdot}\ ,
	\end{equation}
where $x^s=\phi_s(x)$ and $\dot\nabla_\ps(u)=\left.\frac{\partial}{\partial s}\right\vert_{t=0}\nabla^t_\ps (u)$.
\end{lemma}

\begin{proof} Let us distinguish for the sake of this proof the following connections. Let $\nabla$ be the flat connection on $E_0$ and $\nabla^{\tiny\operatorname{End}}$ the associated flat connection on $\End(E_0)$. Then from the equation
$
\pp^2=\pp
$, 
we obtain after differentiating, 
$$
\dot\pp \pp +\pp\dot\pp =\dot\pp\ .
$$
Thus $\dot\pp$ is a section of $F_0$. We note that $\nabla^{\operatorname{End}} A = [\nabla,A]$. Thus taking the variation of the equation $[\nabla_\ps,\pp] = \nabla^{\operatorname{End}}_\ps \pp=0$ yields
$$
\nabla^{\operatorname{End}}_\ps \dot\pp= [\nabla_\ps,\dot\pp] = -[\dot \nabla_\ps,\pp] = [\pp, \dot \nabla_\ps].
$$
In other words, the variation of the fundamental projector $\dot\pp$ is a solution of the cohomological equation $\nabla^{\operatorname{End}}_\ps \eta=\sigma$, where  $\sigma=[\pp,\dot\nabla_{\ps}]$ and $\eta=\dot\pp$. Applying proposition \ref{pro:cohom}, yields the equation \eqref{eq:ppdot}.
\end{proof}

\subsubsection{The fixed connection point of view and the proof of proposition \ref{pro:deriv-fund-proj}}

 We can now compute the variation of the projector in the fixed frame point of view and prove proposition \ref{pro:deriv-fund-proj}. 
We first need to switch from the fixed frame point of view to the fixed connection point of view.

Let $(\nabla^t,h)$ be a variation in the fixed frame  point of view. Let $A^t$ be so that $\nabla^t=A_t^{-1}{\rm  D}A_t $ and $A_0={\mathrm Id}$. In particular, we have
\begin{eqnarray}
\dot\nabla_\ps={\rm D}_\ps \dot A=\d \dot A(\ps )\ .
\label{eq:dADa}
\end{eqnarray}

Then the corresponding variation in the fixed connection point of view is $({\rm D},h_s)$ where
$
h_t=A_t(h)$.
It follows that 
\begin{eqnarray}
\dot h= \dot A(h)\ \ \ ,\ \ \  \dot\nabla_\ps ={\d  \dot A}(\ps)={\mathrm D}_\ps\dot A \  .\label{eq:DA}	
\end{eqnarray}
Let now $\pp_0^t$ be the projector -- in the fixed connection point of view-- associated to $({\mathrm D}, h_t)$, while $\pp^t$ is the projector associated to $(\nabla^s, h)$. Obviously
$$
\pp_0^t=A_t\pp^t A_t^{-1}\,
$$
giving
$$
\dot\pp_0=[\dot A,\pp]+\dot \pp\ .
$$
Using lemma \ref{lem:var-proj0} and equations \eqref{eq:DA}, we have
\begin{equation}
\dot\pp=\int_{-\infty}^0 \pp\cdot [\pp,\dot\nabla_{\ps}] \circ \phi_{s}\ \d s- \int_0^\infty [\pp,\dot\nabla_{\ps}]\cdot\pp \circ\phi(s)\ \d s\ \label{eq:dT3}\ .
\end{equation}
 Letting $\pp_0 \defeq \pp_0^0$ then $\pp^0 = \pp_0$ which yields:
\begin{equation}
\dot\pp_0=[\dot A,\pp_0] +\int_{-\infty}^0\pp_0\cdot  [\pp_0,\dot\nabla_{\ps}]\circ \phi_{s}\ \d s- \int_0^\infty [\pp_0,\dot\nabla_{\ps}]\cdot\pp_0 \circ\phi(s)\ \d s\ \label{eq:dT4}\ .
\end{equation}
From equation \eqref{eq:dADa}, we get that 
$$
\int_0^\infty [\pp_0,\dot\nabla_{\ps}]\cdotp \pp_0 \circ\phi(s)\ \d s=\int_{g^+}[\pp_0,\d\dot A]\cdot\pp_0=-  \int_{g^+}[\d\dot A,\pp_0]\cdot\pp_0\ , 
$$
while 
$$
\int_{-\infty}^0 \pp_0 \cdotp [\pp_0,\dot\nabla_{\ps}]\circ\phi(s)\ \d s=-\int_{g^-}\pp_0 \cdotp[\pp_0,\d\dot A]  =\int_{g^-}\pp_0\cdotp[\d\dot A,\pp_0]\ . 
$$
This concludes the proof of proposition \ref{pro:deriv-fund-proj}.

\subsection{Projectors and notation} In this section, we will work in the context of a $\Theta$-uniformly hyperbolic bundle $\rho=(\nabla,h)$ associated to a decomposition of a trivializable bundle 
$$
E=E_1\oplus\cdots\oplus E_{n+1}\ .
$$
Let us denote $k_\aa\defeq\dim(E_\aa)$ and $K_\aa\defeq k_1+\ldots k_\aa$ so that $\Theta=(K_1,\ldots, K_n)$.

We then write for a geodesic $g$,
$$
\pp^\aa(g)\ ,
$$
the projection on $F_\aa=E_1\oplus\ldots\oplus E_\aa$ parallel to $F_\aa^\circ \defeq E_{\aa+1}\oplus\ldots\oplus E_{n+1}$.

  When $g$ is a phantom geodesic we set the convention that $\pp^\aa(g)\defeq\Id$. 

Observe that all $\pp^\aa(g)$ are well defined projectors in the finite dimensional vector space $V$ which is the space of $\nabla$-parallel sections of $E$. Or in other words the vector space so that in the trivialization given by $\nabla$, $E=V\times\uh$. 

Finally, we will consider a $\Theta$-geodesic $g$, given by a geodesic $g_0$ labelled by an element  $\aa$ of $\Theta$ and write
\begin{equation}
	\pp(g)\defeq\pp^\aa(g_0)\ , \ \Theta_g=\tr(\pp^\aa)=K_\aa\ .\label{def:projet-theta-geod}
\end{equation}

\section{Ghost polygons}\label{sec0:ghost-polygon}
This section consists mainly of definitions: {\em ghost polygons, $\Theta$-decorated ghost polygons,  correlation functions}, and related useful notions notably in the presence of a uniformly hyperbolic bundle ({\em opposite endormorphisms, core diameter}). In the periodic (Anosov) case we finally show the analyticity of  correlation functions.

We will consider the space $\GG$ of oriented geodesics of $\hh$, and an  oriented geodesic $g$ as a pair $(g^-,g^+)$ consisting of two distinct points in $\bh$.

\subsection{Ghost polygons}\label{sec:ghost-polygon}
\vskip 0.2 truecm 
A {\em ghost polygon } is a cyclic collection of geodesics $\vartheta=(\theta_1,\ldots,\theta_{2p})$.  The {\em ghost edges} are the geodesics (possibly phantom) $\theta_{2i+1}$ , and the 
{\em visible  edges} are the even labelled edges $\theta_{2i}$, such that
$$
\theta_{2i+1}^+=\theta_{2i}^+\ \ , \ \ \theta_{2i-1}^-=\theta_{2i}^-\ .
$$

\rmks
\begin{enumerate}
	\item The geodesics are allowed to be phantom geodesics,
\item It will be convenient some time to relabel the ghost edges as $\zeta_i\eqdef\theta_{2i+1}$.
\item It follows from our definition that $(\bar\theta_1,\theta_2,\bar\theta_3,\ldots,\theta_{2p})$ is closed ideal polygon.
\end{enumerate}

We have an alternative point of view. 
A {\em configuration of geodesics of rank $p$} is just a finite cyclically ordered set  of $p$-geodesics. We denote the cyclically ordered set of geodesics $(g_1,\ldots,g_p)$ by $\lceil g_1,\ldots, g_p\rceil$. The cardinality of the configuration  is called the {\em rank} of the the configuration.  

We see that the data of a ghost polygon and a configuration of geodesics is equivalent (see figure \eqref{fig:Ghost2}):
\begin{itemize}
	\item we can remove the ghost edges to obtain a configuration of geodesics from a ghost polygon,
	\item conversely, given any configuration   $G=(g_1,\ldots,g_p)$, the associated ghost polygon $\vartheta=(\theta_1,\ldots,\theta_{2p})$ is given by $\theta_{2i}\defeq g_i$, $\theta_{2i+1}\defeq(g_{i+1}^-,g_i^+)$ 
	\end{itemize}
We finally say that two configurations are {\em non-intersecting} if their associated ghost polygons do not intersect.  

Let us add some convenient definitions.   Let $\vartheta=(\theta_1,\ldots,\theta_{2p})$ be a ghost polygon associated to the configuration  configuration $\lceil g_1,\ldots, g_p\rceil$, We then define the opposite configurations as follows.
\begin{itemize}
\item For  visible edge $g_1$  of $G$, the {\em opposite configuration} is tuple $g_1^*\defeq (g_1,g_2,\ldots,g_{p},g_1)$. 
\item  For ghost edge $\theta_1$ of $G$, the {\em ghost opposite configuration} is the tuple $\theta_1^*\defeq(g_2,\ldots,g_p,g_1)$.
\end{itemize}
Observe that both opposite configurations are not configurations per se but actually tuples -- or {\em ordered configurations}.

We finally define the {\em core diameter $r(G)$} of a ghost polygon  $G$ to be the minimum of those $R$ such that, if $B(R)$ is the ball of radius $R$ centered at the barycenter $\bary(G)$, then $B(R)$ intersects all visible  edges. We obviously have

\begin{proposition}\label{pro:core-cont}
	The map $G\mapsto r(G)$ is a continuous and proper map from $\GG_\star^n/\psld$ to $\mathbb R$.
\end{proposition}

\subsubsection{$\Theta$-Ghost polygons} We now $\Theta$-decorate the situation. As in paragraph  \ref{sec:Theta}, let $\Theta=(K_1,\ldots,K_n)$ with 
$K_\aa<K_{\aa+1}$. Let $G$ be a ghost polygon, a $\Theta$-decoration is a map $\AA$ from the set of visible  edges to ${1,\ldots,n}$. 
 
 We again have the equivalent description in terms of configurations. A {\em $\Theta$-configuration of geodesics of rank $p$} is configuration $(g_1,\ldots,g_p)$   with a map $\AA$ -- the $\Theta$-decoration -- from the collection of geodesics to $\{1,\ldots,n\}$. We think of a {\em $\Theta$-decorated geodesic}, or in short a {\em $\Theta$-geodesic},  as a geodesic labelled with an element of $\Theta$.
 
\subsection{Ghost polygons and uniformly hyperbolic bundles} 
 
When $\rho$ is a $\Theta$-uniformly hyperbolic bundle and $\pp^\aa(g)$ a fundamental projector associated to a geodesic $g$, we will commonly use the following shorthand. 

Let $G$ be ghost polygon  $(\theta_1,\theta_2,\ldots,\theta_{2p})$ be given by configuration $\lceil g_1,\ldots, g_p\rceil$.   
 \begin{enumerate}
	\item For visible edge $g_i$ we write $\pp_i\defeq\pp^\AA_i\defeq\pp^{\AA(i)}(g_i)$.
 	\item For visible edge $g_i$,  the {\em opposite ghost endomorphism} is \begin{equation}\pp_G^\AA(g^*_j) \defeq\pp_{j}\cdot\pp_{j-1}\cdots\pp_{j+1}\cdot\pp_j\ .\end{equation}
	\item For ghost edge $\zeta_i$, the {\em opposite ghost endormorphism} is 
	 \begin{equation}
\pp_G^\AA(\zeta_i^*)\defeq\pp_{i}\cdot\pp_{i-1}\ldots\pp_{i+1}\ .
	\end{equation}	The reader should notice that in the product above, the indices are decreasing.
\end{enumerate}

The opposite ghost endomorphisms have a simple structure in the context of projective uniformly hyperbolic bundles (that is when $\Theta=\{1\}$).
\begin{lemma}[{\sc Opposite endormorphisms}]\label{lem:opp-endo}
When $\Theta=\{1\}$, $\pp_G(\theta^*_i)=\T_G(\rho)\ \pp(\theta_i)$. 	
\end{lemma}
\begin{proof}
 Let $G=(\theta_1,\ldots,\theta_{2p})$ be a ghost polygon with configuration $\lceil g_1,\ldots, g_p\rceil$.   If $g_{i}^+ = g_{i+1}^-$ then $p_{i+1}p_i = 0$ and the equality holds trivially with both sides zero. We thus can assume there is a ghost edge $\zeta_i = \theta_{2i+1}$ for each $i \in \{1,\ldots,p\}$. 
 
 When $\Theta=\{1\}$ all projectors have rank 1. Thus for visible edge $g_i$

\begin{eqnarray*}
\pp_G(g_i^*)=\pp_i\pp_{i-1}\ldots\pp_{i+1}\pp_i&=& \tr(\pp_i\ldots\pp_{i+1})\ \pp_i=	\T_G(\rho)\ \pp(g_i)\ .
\end{eqnarray*}
For a ghost edge $\zeta_i$ as $\tr(\pp_{i+1}\pp_i) \neq 0$
\begin{eqnarray*}
\pp_G(\zeta^*_{i})=\pp_{i}\pp_{i-1}\ldots\pp_{i+1}&=&\pp_{i}\pp_{i+1}\frac{\tr(\pp_n\ldots\pp_1)}{\tr(\pp_i\pp_{i+1})}= \T_G(\rho)\  \qq\ ,
\end{eqnarray*}
where $\qq\eqdef\frac{1}{\tr(\pp_i\pp_{i+1})} \pp_{i}\pp_{i+1}$. Then we see that $\qq$ has trace 1, its image is the image of $\pp_i$, and its kernel is the kernel of $\pp_{i+1}$. Thus $\qq$ is the rank 1 projector on the image of $\pp_i$, parallel to  the kernel of $\pp_{i+1}$. Hence  $\qq=\pp(\zeta_{i})$.  The result follows.
\end{proof}

\vskip 0.2 truecm

\subsection{Correlation function}\label{sec:cor-func}

Given a $\Theta$-configuration of geodesics $G=\lceil g_1,\ldots,g_p\rceil$ given by a $p$-uple of geodesics $(g^0_1,\ldots,g^0_p)$, with a $\Theta$-decoration $\AA$ the {\em correlation function} associated to $G$ is 
\begin{equation}
	\T_G: \rho\mapsto\T_{\lceil g_1,\ldots,g_p\rceil}(\rho)\defeq\tr\left(\pp^{\aa(p)}(g^0_p)\cdots \pp^{\aa(1)}(g^0_1)\right)=\tr\left(\pp(g_p)\cdots \pp(g_1)\right)\ ,
\end{equation}
where $\pp$ is the projector associated to the uniformly hyperbolic bundle $\rho$. The reader should notice (again) that the geodesics and projectors are ordered reversely.

\subsection{Analyticity in the periodic case}

In this subsection we will treat first the case of complex bundles, that is representation 
in $\mathsf{SL}(n,\mathbb C)$ of  the (complex) parabolic group $\mathsf P^\mathbb C_\Theta$ associated to $\Theta$. We now have, as a consequence of \cite[Theorem 6.1]{Bridgeman:2015ba}, the following

\begin{proposition}\label{pro:ana-TG} Let $G$ be a ghost polygon.
Let $\famD{\rho}$ be an analytic family of $\mathsf P^\mathbb C_\Theta$-Anosov representations parametrized by the unit disk $\D$. Then, the  
 function $u\mapsto \T_G(\rho_u)$ is analytic. Moreover the map $G\mapsto \T_G$ is a continuous function with values in the analytic functions. 
\end{proposition}
\begin{proof}
Indeed the correlation functions only depends on the limit curve of the representation and thus the analyticity of the limit curve proved in  \cite[Theorem 6.1]{Bridgeman:2015ba} gives the result. \end{proof}

We deduce the general analyticity result from this proposition by complexifying the representation.

\section{Ghost integration}\label{sec:ghost-integ}  In this section, given a $\Theta$-uniformly hyperbolic bundle $\rho$, a $\Theta$-ghost polygon $G$ and  a $1$-form $\alpha$ on $\hh$ with values in the endormorphism bundle of a uniformly hyperbolic bundle , we produce a complex (or real) number denoted
$$\oint_{\rho(G)}\alpha\ .$$ This  procedure is called   {\em ghost integration}.
The construction is motivated by the following formula that allows us to compute the variation of a correlation function with respect to a variation of uniformly hyperbolic bundles:
$$
\d\T_G\left(\dot\nabla\right)=\oint_{\rho(G)}\dot\nabla\ .
$$
\begin{enumerate}
	\item In order to define ghost integration, we first have to consider which type of forms we wish to integrate. This is done in paragraph \ref{sec:bd-geodbdforms}.
	\item In paragraph \ref{sec:line integration} we define {\em line integration}, a procedure reminiscent of how one gets the solution of the cohomogical solution: integrating a ``contracting part" towards the future and a   ``dilating part" towards the past. We, in particular show some crucial convergence properties of the line integration in lemma \ref{lem:exp-decay}.  
	\item We then define  the  ghost integration  in paragraph \ref{sec:bd-geodbdforms}, using the line integration as a building block.
	\item In paragraph \ref{sec:build-ghibd}, we obtain other formulae depending on the type of forms considered, or whether we are in the projective case or not.
	\item In paragraph \ref{sec:dual-form}, we introduce the dual cohomology object $\Omega_{\rho(G)}$, the {\em ghost dual form  to a ghost polygon}, which is a  1-form with values in the endomorphism bundle so that
 $$
 \int_{\hh}\tr\left(\alpha\wedge \Omega_{\rho(G)}\right)=\oint_{\rho(G)}\alpha\ .
 $$
 \item We finally achieve one of our goals by relating ghost integration to the derivative of correlation functions in paragraph \ref{sec:der-cor}.
\end{enumerate}

\subsection{Bounded and geodesically bounded forms}\label{sec:bd-geodbdforms}
In this paragraph, we define a certain type of  1-forms with values in $\End(E)$, where $E$ is a uniformly hyperbolic bundle $(\nabla,h)$. All norms and metrics will be using the  Euclidean metric $g_h$ on $E$ associated to a framing $h$.

\begin{definition}[\sc Bounded forms]\label{def:bounded-1-form} A {\em bounded 1-form} $\omega$ on $\hh$ with values in  $\End(E)$ is a form so that  $\Vert\omega_x(u)\Vert_x$ is bounded uniformly for all $(x,u)$ in  $\ms U\hh$. Let us denote $ \bLa^\infty(E)$ the vector spaces of those forms and $$
\Vert\omega\Vert_\infty=\sup_{(x,u)\in\ms U\hh} \Vert\omega_x(u)\Vert_x\ .
$$
\end{definition}\label{def:bded-form}
As an example of such forms, we have
\begin{enumerate}
	\item Given a $\Theta$-geodesic $g$, given by a (possibly phantom)   geodesic $g_0$, and an element $\aa$ of $\Theta$, the {\em projector form} is
\begin{equation}
	\beta_{\rho(g)}\defeq \omega_g\ \pp(g)=\omega_g\ \pp^\aa(g_0)\ .\label{def:beta}
\end{equation}	
where we used the notation \eqref{def:projet-theta-geod}.
\item Any $\Gamma$-equivariant continuous form in the case of a periodic bundle.
	\item Given $(A_t)_{t\in]-1,1[}$ a bounded variation of a uniformly hyperbolic bundle (see definition \ref{def:bd-var}, the form
	$$
	\dot{A}\defeq\left.\frac{\partial A_t}{\partial t}\right\vert_{t=0}\  ,
	$$
	is by definition a bounded 1-form.
\end{enumerate}

We do not require forms in $ \bLa^\infty(E)$ to be closed.

\begin{definition}[{\sc Geodesically bounded forms}]
A  form $\alpha$ is {\em geodesically bounded} if for any parallel section $A$ of  $\End(E)$, $\tr(\alpha A)$ is geodesically bounded as in definition \ref{def:geod-bded}. We denote by $ \bXi(E)$ the set of 1-forms which are geodesically bounded.
\end{definition}

Again for any geodesic, the projector form $\beta_{\rho(g)}$ is geodesically bounded. However $\Gamma$-equivariant forms are never geodesically bounded unless they vanish everywhere.

\subsection{Line integration}\label{sec:line integration}

Let $\omega$ be a 1-form in $ \bLa^\infty(E)$. Let $x$ be a point on the oriented  geodesic $g$ and $Q$ a parallel section of $\End(E)$ along $g$. The {\em line integration} of $\omega$ -- with respect to the uniformly hyperbolic bundle $\rho$ -- is given by 
\begin{eqnarray}
	\bS_{x,g,\QQ}(\omega)	& \defeq & 
		 \int_{g^+}\tr\left(\QQ\  [\omega,\pp] \ \pp\right) + \int_{g^-}\tr\left(\QQ\ \pp\ [\omega,\pp]\right)\label{eq:defJ}\ .
\end{eqnarray}
Observe that since for a projector $\pp$, we have $$\tr\left(A\  \pp\  [B,\pp]\right) =\tr\left([\pp,A]\  \pp\  B\right)\ , $$ 
we have the equivalent formulation
\begin{eqnarray}
	\bS_{x,g,\QQ}(\omega)	& = & 
		 \int_{g^+}\tr\left(\omega\ \pp\  [\pp,\QQ] \right) + \int_{g^-}\tr\left(\omega\ [\pp,\QQ]\ \pp\right)\label{eq:defJalt}\ .
\end{eqnarray}
Now let $\alpha$ be a section of $\End(E)$ so that $\d\alpha$ belongs to $ \bLa^\infty(E)$. We also define the {\em primitive line integration} of $\alpha$  by  
\begin{eqnarray*}
	\J_{x,g,\QQ}(\alpha)&\defeq & \tr\left(\alpha(x)\ [\pp,\QQ] \right)+ \bS_{x,g,\QQ}(\d \alpha)\crcr
	&=& \tr\left([\alpha(x),\pp]\  \QQ \right)+ \bS_{x,g,\QQ}(\d \alpha) \ .
\end{eqnarray*}

\subsubsection{Bounded linear forms and continuity}
\begin{proposition}[\sc Continuity]\label{pro:J-cont}
The line integration operator  $$
\omega\mapsto \bS_{x,g,Q}(\omega)\   ,$$ is  a continuous linear form on $ \bLa^\infty(E)$.
\end{proposition}
This proposition is an immediate consequence of the following lemma
\begin{lemma}[\sc Exponential decay] \label{lem:exp-decay} There exist positive constants $B$ and $b$, only depending on $\QQ$ and $x$, so that for any  $\omega$ in  $ \bLa^\infty(E)$
	 if   $y$  is  a point in $g^+$, $z$ a point in $g^-$  and denoting $\pt$ the tangent vector to the geodesic $g$, then
	 \begin{eqnarray}
	 		\left\vert\tr\left(\QQ\  [\omega_y(\pt),\pp]\ \pp\right)\right\vert &\leq& Be^{-bd(x,y)} \Vert\omega\Vert_\infty\ ,\label{ineq:expdecay1}\\
	 		\left\Vert [\QQ,\pp]\ \pp\right\Vert_z &\leq& Be^{-bd(x,z)} \ ,\label{ineq:expdecay2}\\
	 		\left\Vert \pp\  [\QQ,\pp]\right\Vert_y &\leq& Be^{-bd(x,y)} \ .\label{ineq:expdecay3}
	 		\end{eqnarray}

\end{lemma}

\begin{proof} Let us  choose a trivialization of $E$ so that $\nabla$ is trivial. By hypothesis $\omega$ is in $\bLa^\infty(E)$ and thus
\begin{eqnarray}
	\Vert \omega_y(\pt)\Vert_y\leq \Vert\omega\Vert_\infty\ .\label{eq:ombd}
\end{eqnarray}
Then 
$$\sigma: y\mapsto\sigma(y)\defeq [\omega_y(\pt),\pp]\ \pp\ ,$$
is a section of $F_0^-$. Since  $\pp$ is bounded -- see proposition \ref{pro:pbd-unique} -- there exists $k_1$ such that for all $y$
$$
\Vert \sigma(y)\Vert_y\leq k_1 \Vert\omega\Vert_\infty\ .
$$
By lemma \ref{lem:Ano-subbun}, $F_0^-$ is a contracting bundle in the negative direction, which means there exists positive constants $A$ and $a$ so   that if $y=\phi_t(x)$ with $t>0$, then
$$
\Vert \Phi_{-t}^\nabla( \sigma(y))\Vert_x\leq A e^{-at}\Vert \sigma(y)\Vert_y\  ,
$$
where $\nabla$ is the connection.
However in our context, since we have trivialized the bundle,  $\Phi_{-t}^\nabla$ is the identity fiberwise, and thus combining the previous remarks we get that if $y$ is in $g^+$, then
\begin{eqnarray}
	\left\Vert [\omega_y(\pt),\pp]\  \pp \right\Vert_x&\leq &A e^{-a(d(y,x)}\Vert \omega\Vert_\infty\ .\label{eq:exp-decay2}
\end{eqnarray}
As our metric is Euclidean, by Cauchy--Schwarz, for all endomorphisms $U$ and $V$, we have 
\begin{equation}
	\vert \tr(U\  V)\vert\leq \Vert U\Vert_x\Vert V\Vert_x\ .\label{eq:exp-decay1}
\end{equation}
Thus combining equations \eqref{eq:exp-decay1} and \eqref{eq:exp-decay2} we obtain
\begin{eqnarray*}
	\left\vert\tr\left( [\omega_y(\pt),\pp]\   \pp\  \QQ\right)\right\vert
	\leq \Vert \QQ\Vert_x \ \Vert [\omega_y(\pt),\pp]\   \pp \Vert_x
	\leq  Ae^{-a(d(y,x)} \Vert \QQ\Vert_x\ \ \Vert \omega\Vert_\infty\ ,
\end{eqnarray*}  
and the inequality \eqref{ineq:expdecay1} follows.  Similarly, $[\QQ,\pp]\pp$ is a parallel section of $F_0^-$, thus the inequality \eqref{ineq:expdecay1} is an immediate consequence of inequality \ref{ineq:contractF0}.\end{proof}

\subsubsection{Properties of the primitive line integration}
We explain now two properties of the primitive line integration that will be useful in the definition of the ghost intergration;
\begin{proposition}
The primitive line integration $\J_{x,g,Q}(\alpha)$ does not depend on the choice of $x$ on $g$.

\end{proposition}
\begin{proof}
Let us write for the sake of this proof
$\J_x\defeq\J_{x,g,Q}(\alpha)$.
	Let  $\mu$ be the geodesic arc from $y$ to $x$. Let us consider a parametrization of $g$ so that $x=g(s_0)$ and $y=g(t_0)$. Then letting $\omega = \d\alpha$
		\begin{eqnarray*}
		\J_{y}-\J_{x}
		&=&\tr\left((\alpha(y)-\alpha(x))\ [\pp,\QQ]\right)\crcr
			&+&\int_{t_0}^{\infty} \tr\left(\omega(\dot g)\ \pp\ [\pp,Q]\right)\d t +\int_{t_0}^{-\infty} \tr\left(\omega(\dot g)\   [\pp,Q]\ \pp \right)\d t \crcr
		&-&\int_{s_0}^{\infty} \tr\left(\omega(\dot g)\ \pp\  [\pp,Q]\right)\d t-\int_{s_0}^{-\infty} \tr\left(\omega(\dot g)\ [\pp,Q]\ \pp \right)\d t \crcr
	&=&\int_{s_0}^{t_0}\tr\left(\omega(\dot g)\ 
		\left(  [\pp,Q]-\pp\ [\pp, Q]-[\pp,Q]\  \pp  \right)\right)\ \d s =0\ ,
	\end{eqnarray*}
where the last equality comes form the fact that, since $\pp$ is a projector
$$
[\pp,Q]\  \pp +\pp\ [\pp,  Q]=[\pp,Q]\ .\qedhere
$$
\end{proof}
Finally we have, 
\begin{proposition}\label{pro:closed}
	Assume that  $\beta$ is bounded.
	Then $
	\J_{x,Q}(\beta)= 0$.
\end{proposition}
\begin{proof}
	Let $\varpi=\d\beta$. It follows that 
	$$
	\tr\left(\varpi(\pt)\ \pp \  [\pp,\QQ]\right)=\frac{\partial}{\partial t}\tr\left(\beta\ \pp \ [\pp,\QQ]\right)\ .
	$$
	Thus by the exponential decay lemma \ref{lem:exp-decay}, we have
	$$
	\int_{g^+} \tr(\varpi\ \pp\ [\pp,\QQ])=-\tr(\beta(x)\ \pp\   [\pp,\QQ])\ .
	$$
	Similarly 
	$$
	\int_{g^-} \tr(\varpi\  \pp\  [\pp,\QQ])=-\tr(\beta(x)\    [\pp,\QQ]\ \pp )\ .
	$$
	It follows that 
	$$
	\bS_{x,g_0,Q}(\varpi)= -\tr(\beta(x)\  \pp\  [\pp,\QQ])-\tr(\beta(x)\  [\pp,\QQ]\  \pp)=-\tr(\beta(x)\ [\pp,\QQ])\ .$$
	This concludes the proof.
\end{proof}
\subsection{Ghost integration: the construction}\label{sec:build-ghi}
Let now  $G$ be a configuration of geodesics with a $\Theta$-decoration $\AA$. Let $\rho$ be a $\Theta$-uniformly hyperbolic bundle, where $G=\lceil g_1,\ldots, g_p\rceil$. Let $\pp_i=\pp^{\AA(i)}(g_i)$ and 
\begin{eqnarray*}
\PP_i=\pp_{i-1}\ldots\pp_{i+1}\ .	
\end{eqnarray*}
Let  $\alpha$ be a closed 1-form with values in $\End(E)$. Assume that $\alpha$ belongs to   $ \bLa^\infty(E)$. Let   $\beta$ be a primitive of $\alpha$ -- that is  a section of $\End(E)$ so that $\d\beta=\alpha$ -- let 
\begin{eqnarray*}
	\J_{\rho(G)}(\beta)\defeq \sum_{i=1}^n \J_{g_i,\PP_i}(\beta)\ ,
\end{eqnarray*}
\begin{proposition}
	The quantity $\J_{\rho(G)}(\beta)$ only depends on the choice of $\alpha$ and not of its primitive.
\end{proposition}
\begin{proof} Let $\beta_0$ and $\beta_1$ two primitives of $\alpha$. Observe that $B\defeq \beta_1-\beta_0$ is constant, then 
\begin{eqnarray*}
\J_G(\beta_1)-\J_G(\beta_0)=\sum_{i=1}^p \tr\left(B[\pp_i,\PP_i]\right)=\sum_{i=1}^p \tr\left(B\pp_i\PP_i\right)-\sum_{i=1}^p \tr\left(B\PP_i\pp_i\right)=0\ ,
\end{eqnarray*}
since $\PP_i\pp_i=\pp_{i-1}\PP_{i-1}$.
\end{proof}

\begin{definition}[\sc Ghost integration]
	We define the {\em ghost integration} of a 1-form $\alpha$ in $ \bLa^\infty(E)$ with respect to a  $\Theta$-ghost polygon $G$ and a uniformly hyperbolic bundle $\rho$  to be the quantity 

\begin{eqnarray*}
	\oint_{\rho(G)}\alpha\defeq\J_{\rho(G)}(\beta)\ ,
\end{eqnarray*}
where $\beta$ is a primitive of $\alpha$. 

\end{definition}
Gathering our previous results, we summarize the important properties of ghost integration: 

\begin{proposition}\label{pro:oint-cont}
The ghost integration enjoys the following properties:
	\begin{enumerate}
\item The map $\alpha\mapsto \oint_{\rho(G)}\alpha$ is a continuous linear form on $ \bLa^\infty(E)$.
		\item Assume $\alpha=\d\beta$, where $\beta$ is a bounded  section of $\End(E)$. Then
			\begin{eqnarray*}
			\oint_{\rho(G)}\alpha=0 \ .
		\end{eqnarray*}
			\end{enumerate}
\end{proposition}

We remark that the second item implies that ghost integration is naturally an element of the dual of the first bounded cohomology with coefficients associated to  the bundle.

\begin{proof}
	These are consequences of the corresponding properties for $J_{x,g,Q}$ proved respectively in propositions \ref{pro:J-cont} and \ref{pro:closed}.
	\end{proof}

\subsection{Ghost integration of geodesic forms}\label{sec:build-ghibd}

Recall that we denoted by $ \bXi(E)$ the space  of geodesically bounded forms, and observe that for any geodesic $g$, the projector form $\beta_{\rho(g)}$ belongs to $ \bXi(E)$.
\begin{proposition}[\sc Alternative formula]\label{pro:gintegr-alt} Let $\rho$ be a $\Theta$-uniformly hyperbolic bundle.   Let $G$ be configuration of geodesics of rank $p$ associated to  a ghost polygon $\vartheta\defeq(\theta_1,\ldots\theta_{2p})$ and a $\Theta$-decoration. Assume that $\alpha$ is in $ \bXi(E)$. Then 
	\begin{eqnarray*}
	\oint_{\rho(G)}\alpha&=& - \left(\sum_{i=1}^{2p}(-1)^{i}\int_{\theta_i}\tr\left(\alpha\ \pp^\AA_G(\theta_i^*)\right)\right)\ ,
\end{eqnarray*}
where $\pp_G^\AA(\theta_i^*)$ denotes the opposite ghost endomorphism to $\theta_i$. \end{proposition}

In the context of projective uniformly hyperbolic bundle, that is $\Theta=\{1\}$, then the previous formula is much simpler as an immediate consequence of lemma \ref{lem:opp-endo}.
\begin{proposition}[\sc Projective  formula] \label{pro:gintegr-alt2}   Let $G$ be configuration of geodesics of rank $p$ associated to  a ghost polygon $\vartheta\defeq(\theta_1,\ldots\theta_{2p})$ and a $\Theta$-decoration.  Let $\rho$ be a projective uniformly hyperbolic bundle. Assume that $\alpha$ is in $ \bXi(E)$. Then
	\begin{eqnarray*}
	\oint_{\rho(G)}\alpha&=& - \T_G(\rho)\left(\sum_{i=1}^{2p}(-1)^{i}\int_{\theta_i}\tr\left(\alpha\ \pp(\theta_i)\right)\right)\ .
\end{eqnarray*}

\end{proposition}
Observe that both formulae above do not make sense for a general bounded form. 
Observe also that 
\begin{proposition}\label{pro:notrace}
	Let $G$ be a ghost polygon, and $\alpha$  a 1-form with values in the center of $\End(E)$ then
	$$
	\oint_{\rho(G)}\alpha=0\ .
	$$
\end{proposition}

\subsubsection{An alternative construction: a first step}

\begin{proposition}\label{pro:2-altcons}
Let $x$ be a point in $\hh$, $\gamma_i^\pm$ the geodesic from $x$ to $g_i^\pm$. Assume that $\alpha$ is in $ \bXi(E)$ then 

\begin{eqnarray*}
	\oint_{\rho(G)}\alpha &=&  \sum_{i=1}^p\left(\int_{\gamma_i^+}\tr\left(\alpha \ \pp_i\ [\pp_i,\PP_i]\right)+ \int_{\gamma_i^-}\tr\left(\alpha\ [\pp_i,\PP_i]\ \pp_i\ \right)\right)\ .
\end{eqnarray*}
\end{proposition}
\begin{proof}
	Fix a point $x_i$ in each of the $g_i$. Let $\beta$ be a primitive of $\alpha$ so that $\beta(x)=0$. Let $\eta_i$ be the geodesic from $x$ to $x_i$. It follows that, since $\alpha$ is geodesically bounded,  we have by the cocycle formula \eqref{eq:triangle-vanish}
	$$
	\int_{\gamma_i^+}\tr(\PP_i\ [\alpha,\pp_i]\ \pp_i)= \int_{\eta_i}\tr(\PP_i\ [\alpha,\pp_i]\ \pp_i)+\int_{g_i^+}\tr(\PP_i\ [\alpha,\pp_i]\ \pp_i)\ .
	$$
	Similarly 
$$
	\int_{\gamma_i^-}\tr(\PP_i\ \pp_i\ [\alpha,\pp_i])= \int_{\eta_i}\tr(\PP_i\ \pp_i\ [\alpha,\pp_i]+\int_{g_i^-}\tr(\PP_i\ \pp_i\ [\alpha,\pp_i])\  .
	$$
\begin{figure}[h] 
  \begin{center} \includegraphics[width=0.3\textwidth]{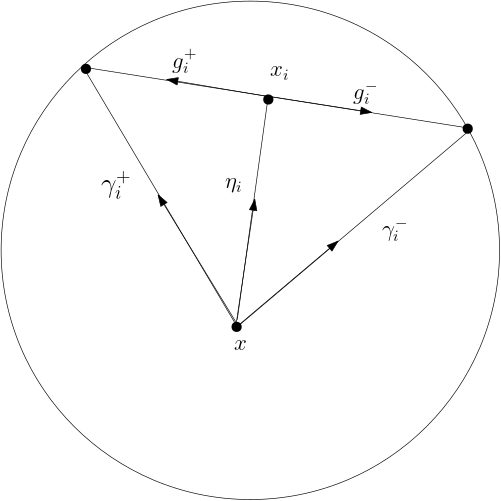} 
  \end{center}	
   \caption{Curves $\gamma_i^{\pm}$ and $\eta_i$}
   \label{fig:gammas}
\end{figure}	
	Observe now that, using the relation  $
[\pp,Q]\  \pp +\pp\ [\pp,  Q]=[\pp,Q]$, we have \begin{eqnarray*}
		\int_{\eta_i}\tr(\PP_i\ [\alpha,\pp_i]\ \pp_i)+\int_{\eta_i}\tr(\PP_i\ \pp_i\ [\alpha,\pp_i])
		=\int_{\eta_i}\tr(\PP_i\ [\alpha,\pp_i])=\tr\left(\PP_i\ [\beta(x_{i}),\pp_i]\right)\ .
\end{eqnarray*}
Thus, we can now conclude the proof:
\begin{eqnarray*}\J_{\rho(G)}(\beta)&=&
			 \sum_{i=1}^p\left(\int_{\gamma_i^+}\tr\left(\PP_i\ [\alpha,\pp_i]\ \pp_i\right)+ \int_{\gamma_i^-}\tr\left(\PP_i\ \pp_i\ [\alpha,\pp_i]\right)\right)\crcr
	&=&  \sum_{i=1}^p\left(\int_{\gamma_i^+}\tr(\pp_i\ [\pp_i,\PP_i]\ \alpha)+ \int_{\gamma_i^-}\tr([\pp_i,\PP_i]\ \pp_i\  \alpha)\right)\ .
		    \end{eqnarray*}
		     \end{proof}

\begin{proof}[Proof of proposition \ref{pro:gintegr-alt}]
Let us assume we have a ghost polygon   $\vartheta = (\theta_1,\ldots,\theta_{2p})$ given by a configuration of geodesics  $G=\lceil g_1,\ldots, g_p\rceil$. Let $\pp_i=\pp(g_i)$ and $\alpha$ an element of $ \bXi(E)$. We have 
\begin{eqnarray*}\pp_i\ [\pp_i,\PP_i]&=& \pp_i\PP_i- \pp_i\PP_i\pp_i\ ,	\\ \ 
[\pp_i,\PP_i]\ \pp_i&=& \pp_i\PP_i\pp_i- \PP_i\pp_i=\pp_i\PP_i\pp_i- \pp_{i-1}\PP_{i-1} \ .\end{eqnarray*}
Since $\alpha$ is geodesically bounded we have
\begin{eqnarray*}
	\oint_{\rho(G)}\alpha
	=\sum_{i=1}^p\left(\int_{\gamma_{i}^+}\tr(\alpha\  \pp_i\ \PP_i\ )- \int_{\gamma_{i+1}^-}\tr(\alpha\ \pp_i\ \PP_i)\right)	-\sum_{i=1}^p\left(\int_{\gamma_i^+}\tr(\alpha\ \pp_i\ \PP_i\ \pp_i)-\int_{\gamma_{i}^-}\tr(\alpha\ \pp_i\ \PP_i\ \pp_i) \right)\ .
\end{eqnarray*}
 For $i\in \{1,\ldots,p\}$, let $\zeta_i$ be the ghost edge joining $g_{i+1}^-$ to $g_{i}^+$, that is $\zeta_i=\theta_{2i+1}$. For a closed form  $\beta$ which is geodesically bounded the cocycle formula \eqref{eq:triangle-vanish} yields
\begin{eqnarray*}
	\int_{\gamma_i^+}\beta-\int_{\gamma_i^-}\beta=\int_{g_i}\beta\ \ \ ,\ \ \ 
	\int_{\gamma_{i+1}^-}\beta-\int_{\gamma_i^+}\beta=-\int_{\zeta_i}\beta\ .
\end{eqnarray*}
Thus 
\begin{eqnarray*}
	\J_{\rho(G)}(\alpha)
	&=&\sum_{i=1}^p\left(\int_{\zeta_i}\tr(\alpha\pp_i\PP_{i}) -\int_{g_i}\tr(\alpha\pp_i\PP_i\pp_i)\right)\ .
\end{eqnarray*}
\begin{figure}[htbp]   \centering
   \includegraphics[width=0.3\textwidth]{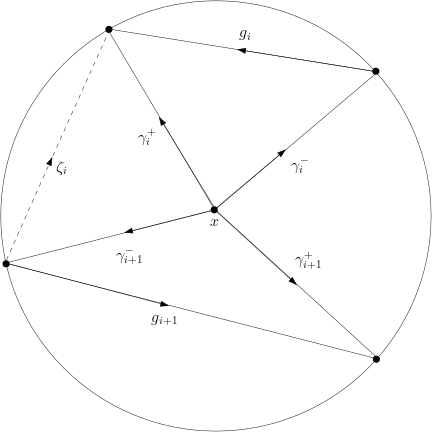} 
   \caption{Appearance of ghosts}
   \label{fig:gammas2}
\end{figure}
To conclude we need first to observe that as  $g_i$ is a visible geodesic then  $\pp_i\PP_i\pp_i$ is the opposite ghost endomorphism $\pp_G(g_i^*)$. On the other hand as $\zeta_j$ is a ghost edge then $\pp_j\PP_j$ is the opposite ghost endomorphism $\pp _G(\zeta_j^*)$. Thus
\begin{eqnarray*}
	\J_{\rho(G)}(\alpha)
	&=&- \left(\sum_{i=1}^{2p}(-1)^{i}\int_{\theta_i}\tr\left(\alpha\ \pp_G(\theta_i^*)\right)\right)\  .
\end{eqnarray*}
\end{proof}

\subsubsection{Another altenative form with polygonal arcs} Let $G = (\theta_1,\ldots,\theta_{2p})$ be a $\Theta$  ghost polygon given by configuration $\lceil g_1,\ldots,g_p\rceil$ with $g_i = \theta_{2i}$. Let $x$ be the barycenter of $G$.  Let $x_i$ be the projection of $x$ on $g_i$. For a ghost edge $\zeta_i = \theta_{2i+1}$, let us consider the polygonal arc $\bzeta_i$ given by 
$$
\bzeta_i=a_i\cup b_i\cup c_i\cup d_i\ ,
$$  
where 
\begin{itemize}
\item the geodesic arc $a_i$ is  the arc (along $g_{i+1}$) from $g_{i+1}^-$ to $x_{i+1}$, 
\item the geodesic arc $b_i$ joins $x_{i+1}$ to $x$,
\item  the geodesic arc $c_i$ joins $x$ to $x_i$,
\item the geodesic arc $d_i$ joins $x_i$ to $g_{i}^+$.
\end{itemize}

\begin{figure}[htbp]
   \centering
   \includegraphics[width=0.3\textwidth]{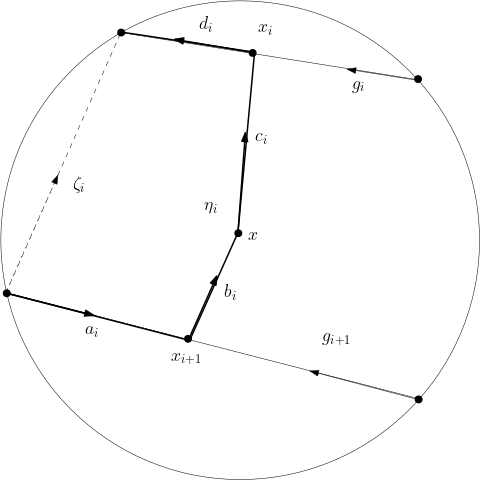} 
   \caption{Polygonal arc $\bzeta_i$ for  ghost edge $\zeta_i$}
   \label{fig:poly_ghost}
\end{figure}
We then have, using the same notation as in proposition \ref{pro:gintegr-alt} 

\begin{proposition}[\sc Alternative formula II]\label{pro:gintegr-alt-II} We have for $\alpha$ in $ \bXi(E)$	\begin{eqnarray}
	\oint_{\rho(G)}\alpha&=& -\sum_{i}
	 \int_{g_i}\tr\left(\alpha\ \pp_G(\theta_i^*)\right)\ 
	 + 
 \int_{\bzeta_i}\tr\left(\alpha\ \pp_G(\theta_i^*)\right)\ . \end{eqnarray}
\end{proposition}
\begin{proof}
	The proof relies on the fact that for $\alpha$ in $ \bXi(E)$,  and $\zeta_i$ a ghost  edge we have
	$$
	\int_{\bzeta_i}\alpha=\int_{\zeta_i}\alpha\ .
	$$
	Then the formula follows from proposition \ref{pro:gintegr-alt}.
\end{proof}

\rmka{\sc Ghost integration and Rhombus integration.}
The process described for the ghost integration is a generalization of the Rhombus integration described in \cite{Labourie:2018fj}.

\subsection{A dual cohomology class}\label{sec:dual-form}
Let $\rho$ be a $\Theta$-uniformly hyperbolic bundle. Now let $G$ be a $\Theta$-ghost polygon with configuration $\lceil g_1,\ldots,g_p\rceil$ and $\Theta$-decoration $\AA$. Let $\vartheta=(\theta_1,\ldots,\theta_{2p})$ be the associated ghost polygon and denote by $\zeta_i=\theta_{2i+1}$ the ghost edges. Let  $\bzeta_i$ be the associated polygonal arc associated to the ghost edge $\zeta_i$ as in paragraph \ref{sec:polyg-arc}.

\begin{definition}
	The {\em ghost dual form   to $\rho(G)$}  is a one form with values in the  endormorphism bundle defined by 
\begin{eqnarray*}
	\Omega_{\rho(G)}\defeq\sum_{i=1}^{p}\left(\omega_{g_i}\pp_G(g^*_i) -  \omega_{\bzeta_i}\pp_G(\zeta^*_i)\right)\ . 
\end{eqnarray*}
\end{definition} 
Observe that $\rho(G)$ incorporates a $\Theta$-decoration and so $\Omega_{\rho(G)}$ depends on the $\Theta$-decoration.   Actually ${\Omega_\rho(G)}$ depends on a choice of the mapping $g\mapsto\omega_g$.  
\begin{proposition}[\sc ghost dual form ]\label{pro:dual-ghost-form} We have the following properties
\begin{enumerate}
\item 	The ghost dual form  belongs to $ \bXi(E)$.
	\item Assume that  $\alpha$ belongs to $ \bXi(E)$.	Then
	\begin{equation}
\oint_{\rho(G)}\alpha=\int_\hh\tr\left(\alpha\wedge \Omega_{\rho(G)}\right)\ .\label{eq:Omega-int}
	\end{equation}
	\item  {\sc(exponential decay inequality)} Finally, 
		there exist positive  constants $K$ and $a$ only depending on $\rho$  and $R_0$ so that if the core diameter of $G$ is less than $R_0$, then
	\begin{eqnarray} 
		\Vert\Omega_{\rho(G)}(y)\Vert_y\leq K e^{-a\  d(y,\bary(G))}\ ,\label{eq:bd-Omega}
	\end{eqnarray}
	and, moreover, $\Omega_{\rho(G)}(y)$ vanishes when $d(y,\bary(G))\geq R_0+2$ and $d(y,g)>2$ for all visible edges $g$ of $G$.
\end{enumerate}
 	\end{proposition}
 	  Observe that even though $\Omega_{\rho(G)}$ depends on some choice, the left-hand side of the formula \eqref{eq:Omega-int} does not depend on any choice. Later we will need the following corollary which we prove right after we give the proof of the proposition. 
 	\begin{corollary}\label{coro:bd-Om} 
 	We have the following bounds:
 The map
	\begin{eqnarray*}
		\phi_G\ :\ y&\mapsto&\Vert\Omega_{\rho(G)}(y)\Vert_y\ \ , 
		\end{eqnarray*}
	 belongs   to $L^1(\hh)$, and $\Vert\phi_G\Vert_{L^1(\hh)}$ is bounded by a continuous function of the  core diameter of $G$.	The map \begin{eqnarray*}
		\psi_{G,y}\ : \ \gamma&\mapsto& \Vert\Omega_{\rho(G)}(\gamma y)\Vert_y\  ,
	\end{eqnarray*} belongs  to $\ell^1(\Gamma)$, and $\Vert\psi_{G,y}\Vert_{\ell^1(\Gamma)}$ is bounded by a continuous function of the  core diameter of $G$.	 Finally the map 
	\begin{eqnarray*}
		 \phi\ : H&\mapsto&
		 \Vert\Omega_{\rho(H)}\Vert_\infty=
		 \sup_{y\in\hh}\Vert\Omega_{\rho(H)}(y)\Vert_y\  , 
			\end{eqnarray*} 
			is bounded on every compact set of ${\GT}^p_\star$
 	\end{corollary}

\begin{proof}[Proof of proposition \ref{pro:dual-ghost-form}]  We first prove the exponential decay inequality \eqref{eq:bd-Omega} which implies in particular that $\Omega_\rho(G)$ belongs to $\bLa^\infty(E)$. 

Let $r(G)$ be the core diameter of $G$. As usual, let $g_i$ be a visible edges, $x$ be the barycenter of all $g_i$ and $x_i$ be the projection of $x$ on $g_i$. By the construction of the polygonal arc $\bzeta_i$, it follows that outside of the ball of radius $r(G)+2$ centered at $x$, that
$$
\Omega_{\rho(G)}= \sum_{i}\omega^-_{g_i} [\PP_i,\pp_i]\pp_i  + \sum_{i}\omega^+_{g_i} \pp_i [\PP_i,\pp_i]\ ,
$$

where $\omega^\pm_{g_i}=f^\pm_i\omega_{g_i}$ where $f^\pm_i$ is a function with values in $[0,1]$ with support in the 2-neighbourhood of the arc $[x_i,g_i^\pm]$. Then the decay given in equation \eqref{eq:bd-Omega} is an immediate consequence of the exponential decay given in inequality  \eqref{ineq:expdecay2}.

Observe now that $\Omega_{\rho(G)}$ is closed. 
Let $A$ be a parallel section of $\End(E)$,  then it is easily seen that $\tr(\Omega_{\rho(G)}A)$ is geodesically bounded. It follows that $\Omega_{\rho(G)}$ is in   $\bXi(E)$.

Then the result follows from the alternative formula for ghost integration in  proposition \ref{pro:gintegr-alt-II}. 
\end{proof}
\begin{proof}[Proof of corollary \ref{coro:bd-Om}]  Given a ghost polygon $H$ whose set of visible edges is $g_H$, and core diameter less than $R_0$. Let
$$
V_H=\{y\in\hh\mid d(y,\bary(H))\leq R_0+2\hbox{ or $d(y,g)\leq 2$ for some $g \in g_H$}\}
$$
Observe that the volume of $V_H(R)\defeq V_H\cap B(\bary(H),R)$ has some linear growth as a function of $R$, and moreover this growth is  controlled as a function of $R_0$. This, and  the exponential decay  inequality \eqref{eq:bd-Omega}, implies that $\phi_G$, whose support is in $V_H$, 
is in $L^1(\hh)$ and that its norm is bounded by a constant that only depends on $R_0$. Similarly consider
\begin{eqnarray*}
	F_{H,y}\defeq \{\gamma\in\Gamma\mid d(\gamma(y),\bary(H))\leq R_0+2 \hbox{ or } d(\gamma(y),g)\leq 2  \hbox{ for  $g$ in $g_H$} \}\ .
\end{eqnarray*}
and 
\begin{eqnarray*}
F_{H,y}(R)\defeq \{\gamma\in F_{H,y} \mid d(\gamma (y),\bary(H))\leq R\} \ .
\end{eqnarray*}
Then the cardinality of the subset $F_{H,y}(R)$ has linear growth depending only on $R_0$. Hence, for every $y$, 
$$
 \gamma\mapsto \sum_{\gamma\in F_{H,y}}K_0e^{-a(d(\gamma y,\bary(H))}\ ,
$$
seen as a function of $H$ is in $\ell^1(\Gamma)$ and its $\ell^1$ norm is bounded as a function of $R_0$.

Hence -- as a consequence of the exponential decay inequality \eqref{eq:bd-Omega} -- for every $y$, the map
$$
\gamma\mapsto\Vert\Omega_{\rho(G)}(\gamma y)\Vert_y\ , 
$$
is in $\ell^1(\Gamma)$ and its $\ell^1$ norm is bounded by  a function of $R_0$.

Finally  from inequality \eqref{eq:bd-Omega}, we have obtained that there is a constant $R_1$ only depending on $R_0$ such that
$$
\sup_{y\in\hh}\Vert\Omega_{\rho(H)}(y)\Vert_y\leq \sup_{y\in B(\bary(H), R_1)}\Vert\Omega_{\rho(H)}(y)\Vert_y +1 \  . \  
$$
The bounded cocycle hypothesis, equation \eqref{eq:exp-bounded}, implies that $\sup_{y\in B(\bary(H), R_1)}\Vert\Omega_{\rho(H)}(y)\Vert_y$ is bounded by a function only depending on $R_1$, 
and thus $\sup_{y\in\hh}\Vert\Omega_{\rho(H)}(y)\Vert_y$ is bounded by a function of $R_0$. This completes the proof of the corollary.
\end{proof}
\subsection{Derivative of correlation functions}\label{sec:der-cor}
In this paragraph, as a conclusion of this section, we relate the process of ghost integration with the derivative of correlation functions.

\begin{proposition}\label{pro:dT-integ} Let $\fam{\nabla_t,h}$ be a  bounded variation of a uniformly hyperbolic bundle $\rho=(\nabla,h)$. Assume that $G$ is a $\Theta$-ghost polygon, then \begin{eqnarray}
	\d\T_G(\dot\nabla)=\oint_{\rho(G)}\dot\nabla\ . \  \label{eq:dT-integ}
\end{eqnarray}
\end{proposition}

This proposition is an immediate consequence of the following lemma, which is itself an immediate consequence of the definition of the line integration in paragraph \ref{sec:line integration} and lemma \ref{lem:var-proj0}:
\begin{lemma}\label{pro:var-projJ} 
	Let $(\nabla_t,h_t)$ be a family of uniformly hyperbolic bundles with bounded variation -- see definition \ref{def:bd-var} -- associated  to a  family of fundamental projectors $\fam{{\pp}}$. 
	Then for a decorated geodesic $g$,
	\begin{eqnarray}
			\tr\left(\dot\pp_0(g)\cdot Q\right)=\J_{\rho(g),Q}\left(\dot\nabla\right)\ .\label{eq:varprojJ}
	\end{eqnarray}
\end{lemma}

\subsection{Integration along geodesics}
For completeness, let us introduce {\em ghost integration for geodesics}: we define for any geodesically bounded  1-form $\alpha$ in $\Xi(E)$ and a $\Theta$-geodesic $g$, 
\begin{eqnarray}
	\oint_{\rho(g)}\alpha\defeq \int_g\tr\left(\alpha\pp_g\right)\ .\label{def:ointg}
\end{eqnarray}
\rmka It is important to observe that, contrary to a general ghost polygon, we only integrate geodesically bounded forms, not bounded ones. In particular, we cannot integrate variations of uniformly hyperbolic bundles.

\section{Ghost intersection and the ghost algebra}\label{sec:ghost-inter}

In this section we will effectively  define  and compute the {\em ghost intersections} of ghost polygons or geodesics.  This is the objective of propositions \ref{pro:gh-intersect1} and \ref{pro:gh-intersect2}.

The {\em ghost intersection} is a generalization of the intersection of geodesics. However this procedure is non abelian in nature: it requires the choice of a uniformly hyperbolic bundle $\rho$.

The relation between ghost integration and ghost intersection is akin to the relation between integrating along a geodesic and intersections of geodesic. Thanks to the ghost dual form  $\Omega_g$ of a ghost polygon, we will define the ghost intersection with respect to a uniformly hyperbolic bundle $\rho$ as 
$$
\bI_\rho(G,H)=\oint_{\rho(G)}\Omega_H=-\oint_{\rho(H)}\Omega_G=\int_\hh \tr\left(\Omega_H\wedge\Omega_G\right)\ .
$$
Observe that the above formulae mimic the formulae of the first section about intersection of geodesics in  the hyperbolic plane. We then compute effectively the value of the ghost intersection, and as a consequence show that it only depends on $G$ and $H$.

We then move to a more formal point of view  in order to perform some computations more efficiently in the future.

We define the associated {\em ghost algebra} in paragraph \ref{par:theta-ghost-bra}, the ghost algebra is a vector space generated by ghost polygons equipped with a bracket,
$$
(G,H)\mapsto [G,H]\ ,
$$ 
we  naturally extend the correlation function to the ghost algebra and reinterpret the intersection in the crucial proposition \ref{pro:I-T} as
$$
I_\rho(G,H)=\T_{[G,H]}(\rho)\ .
$$
In paragraph \ref{sec:bracket-ghost}, we relate the corresponding ghost bracket for the projective case to the swapping bracket defined in \cite{Labourie:2012vka} by the second author. 

In the  somewhat independent paragraph \ref{sec:natural}, we define and study {\em natural maps} from the ghost algebra to itself.   The point of this construction is that the swapping bracket is a Lie bracket, while this is not true of the ghost bracket, and we will use the Jacobi identity later on in our applications. 

We will  freely use the definitions given in section \ref{sec0:ghost-polygon} for ghost polygons.

\subsection{Ghost intersection: definitions and computations}   The intersection will be complex (or real) valued and always with respect to a uniformly hyperbolic bundle.  
We proceed step by step with the definitions:   intersections of two geodesics, of one geodesic with a ghost polygons and finally of two polygons. In all cases, the intersection is with respect to a representation.  
\subsubsection{Intersecting two geodesics} Let $g$ and $h$ two $\Theta$-geodesics (in other words, geodesics labelled with an element of $\Theta$). Let us define 
\begin{equation}
\bI_\rho(g,h)\defeq \oint_{\rho(g)} \beta^0_h\ ,
\end{equation}
where $\beta^0_h\defeq\beta_h-\frac{\Theta_h}{\dim(E)}\Id$ is the trace free part of $\beta_h$ and $\Theta_h$ is defined in equation \eqref{def:projet-theta-geod}.
A straightforward computation using equation \eqref{eq:duality1} and \eqref{def:ointg}
 then gives
	\begin{equation}
\bI_\rho(g,h)=\epsilon(h,g)\left(\T_{\lceil g,h\rceil} (\rho) - \frac{1}{\dim(E)}\Theta_g\Theta_h\right)  \ .	\end{equation}

By convention,  the quantity  $\epsilon(g,h)$ for two $\Theta$-decorated geodesics $g$ and $h$ is the same as the intersection of the underlying geodesics.
\subsubsection{Intersecting a ghost polygon with a geodesic}

Let $\rho$ be a $\Theta$-uniformly hyperbolic bundle. Let $G$ be a $\Theta$-ghost polygon and 
$h$ a $\Theta$-geodesic. The {\em ghost intersection of $G$ and $g$} 
is 
\begin{equation}
	\bI_\rho(G,h)\defeq -\oint_{\rho(G)}\beta_{\rho(h)}=-\int_\hh\tr\left(\beta_{\rho(h)}\wedge \Omega_{\rho(G)}\right)=-\int_h\tr\left(\Omega_{\rho(G)}\ \pp(h)\right)\eqdef-\bI_\rho(h,G)\ .\label{defeq:oint}
\end{equation}

By convention we set $\bI_\rho(h,G)\defeq -\bI_\rho(G,h)$.
We will prove that we can effectively compute the ghost intersection,   and in particular show that the ghost intersection only depends on $G$, $h$ and $\rho$.   Then we have

\begin{proposition}[\sc Computation of ghost intersection I]\label{pro:gh-intersect1}
  Let $G$ be a configuration of geodesics, associated to a ghost polygon $\vartheta=(\theta_1,\ldots,\theta_{2p})$. 
The ghost intersection of $h$ and $G$  satisfies
\begin{equation}
			\bI_\rho(G,h)=\sum_{i=1}^{2p}(-1)^{i+1}\epsilon(h,\theta_j)\ \T_{\lceil h, \theta^*_i\rceil}(\rho)\ \label{eqdef:ghost-inters},
\end{equation}
where $\theta_i^*$ is the opposite configuration as in paragraph \ref{sec:ghost-polygon}. In the projective case, that is $\Theta=\{1\}$ we have
\begin{equation}
		\bI_\rho(G,h)=\T_G(\rho)\left(\sum_{i=1}^{2p}(-1)^{i+1}\epsilon(h,\theta_j)\ \T_{\lceil h, \theta_i\rceil}(\rho)\right)\ . \label{eqdef:ghost-inters-proj}
\end{equation}
\end{proposition}

\subsubsection{Intersecting two ghost polygons}

We define the {\em ghost intersection of two ghost polygons}  or equivalently of two configuration of geodesics $G$ and $H$  to be 
\begin{equation}
	\bI_\rho(G,H)\defeq \oint_{\rho(G)}\Omega_{\rho(H)}=\int_\hh\tr\left(\Omega_{\rho(H)}\wedge \Omega_{\rho(G)}\right)\ .\label{eq:def-I}
\end{equation}

We can again compute this intersection effectively   and in particular show that the ghost intersection only depends on $G$, $H$ and $\rho$:

\begin{proposition}[\sc Computation of ghost intersection II]\label{pro:gh-intersect2}   The ghost intersection  of the two configuration $G$ and $H$, associated respectively to the ghost polygons $\vartheta=(\theta_i)_{i\in I}$, with $I=[1,2p]$, and $\varsigma=(\sigma_j)_{j\in J}$, with $J=[1,2m]$,  respectively,  is given by 
\begin{eqnarray*}
			\bI_\rho(G,H)&=&\sum_{i\in I,j\in J}(-1)^{i+j}
		\epsilon(\sigma_j,\theta_i)\ \T_{\lceil \sigma^*_j,\theta^*_i\rceil}(\rho)\ .
\end{eqnarray*}
In the projective case, this simplifies as
\begin{eqnarray*}
			\bI_\rho(G,H)&=&\T_G(\rho)\T_H(\rho)\left(\sum_{i\in I,j\in J}(-1)^{i+j}
		\epsilon(\sigma_j,\theta_i)\ \T_{\lceil \sigma_j,\theta_i\rceil}(\rho)\right)\ .
\end{eqnarray*}
\end{proposition}

\subsection{$\Theta$-Ghost bracket and the ghost space}\label{par:theta-ghost-bra}
We develop a more formal point of view. Our goal is proposition \ref{pro:I-T} that identifies the intersection as a correlation function. Let $\mathcal A$ be the vector space generated by $\Theta$-ghost polygons  (or equivalently configurations of $\Theta$-geodesics) and $\Theta$-geodesics. We add as a generator the element $\casper$, and call it the {\em Casimir element}. By definition, we say $\casper$ has rank $0$. We will see that the Casimir element will generate the center.

Recall also that we can reverse the orientation on geodesics. The corresponding {\em reverse orientation} on configuration is given by 
$\overline{\lceil g_1,\ldots ,g_p\rceil}\defeq \lceil \bar g_p,\ldots ,\bar g_1\rceil$.

\begin{definition}[\sc $\Theta$-ghost bracket on $\mathcal A$] We define the bracket on the basis of $\mathcal A$ and extend it by linearity.

\begin{enumerate}
	\item The bracket of $\casper$ with all elements is 0.
	\item Let $G$ and $H$ be two configurations of $\Theta$-geodesics,  associated respectively to the ghost polygons $\vartheta=(\theta_i)_{i\in I}$, with $I=[1, 2p]$ and $\varsigma=(\sigma_j)_{j\in J}$, with $J=[1,2m]$  respectively.  Their {\em $\Theta$-ghost bracket} is given  by 
 \begin{eqnarray*}
 	[G,H]&\defeq& \sum_{i\in I,j\in J}
		\epsilon(\sigma_j,\theta_i)(-1)^{i+j}{\lceil \theta^*_i,\sigma^*_j\rceil}\ ,\  
\end{eqnarray*}
where we recall that $\theta^*_j$ is the opposite ghost configuration defined in paragraph \ref{sec:ghost-polygon}.

\item Let $g$ and $h$ be two $\Theta$-geodesics and $G$ a ghost polygon as above.  Then we define \begin{eqnarray*}		
		 [g,h]&\defeq&  \epsilon(h,g)\left(\lceil h,g\rceil\ -\Theta_h\Theta_g\cdot \casper\right)\ ,\\
		 \lbrack G,h]&\defeq&   \sum_{j\in J}
		(-1)^{j+1}\epsilon(h,\theta_j)\ \lceil h,\theta^*_j\rceil\eqdef -[h,G]\  ,
		 \end{eqnarray*}
	\end{enumerate}
	
Finally $\mathcal A$ equipped with the ghost bracket is called the {\em ghost algebra}.	
\end{definition} 
 We observe that the ghost bracket is antisymmetric.  However, the $\Theta$-ghost bracket does not always satisfy the Jacobi identity: there are some singular cases. We actually prove in the Appendix \ref{app:Jac}, as Theorem \ref{theo:jac} the following result 
\begin{theorem}\label{theo:jac1}
	Assume $A$, $B$, and $C$ are ghost polygons and that $$
	V_A\cap V_B\cap V_C=\emptyset\ ,
	$$ where $V_A$, $V_B$ and $V_C$ are the set of vertices of $A$, $B$ and $C$ respectively, then
	$$
	[A,[B,C]]+[B,[C,A]]+[C,[A,B]]=0\ .
	$$
\end{theorem}

Finally we now extend the map $\T$ by linearity to $\mathcal A$ so as to define $\T_G(\rho)$ for $G$ an element of $\mathcal A$, while defining
$$\T_{\casper}(\rho) \defeq \frac{1}{\dim(E)}\ .$$

The purpose of this formal point of view is to  obtain  the simple formula:
\begin{proposition}\label{pro:I-T} We have for $G, H$ ghost polygons then
	\begin{equation}
\bI_\rho(G,H)=\T_{[G,H]}(\rho)\ .	
\end{equation}
\end{proposition}
This formula will allow us to compute recursively Poisson brackets of correlation functions.
\begin{proof}
	This is a simple rewriting of Propositions \ref{pro:gh-intersect1} and \ref{pro:gh-intersect2}.
\end{proof}

\subsection{The projective case:   swapping  and ghost algebras}\label{sec:bracket-ghost}
Throughout this section, we will restrict ourselves to the projective case, that is $\Theta=\{1\}$.

\subsubsection{Ghost polygons and multifractions}
 In \cite{Labourie:2012vka}, the second author introduced the {\em swapping algebra} $\mathcal L$ consisting of polynomials in variables $(X,x)$, where $(X,x)$ are points in $S^1$, together with the relation $(x,x)=0$. We introduced the {\em swapping bracket} defined on the generators  by
$$
[(X,x),(Y,y)]=\epsilon\left((Y,y),(X,x)\right)\ \big((X,y)\cdot (Y,x)\big )\ .
$$ 
We proved that the swapping bracket gives to the swapping algebra the structure of a Poisson algebra. 

We also introduced the {\em multifraction algebra} $\mathcal B$ which is the vector space in the fraction algebra of $\mathcal L$ generated by the {\em multifractions} which are elements defined, when $X$ and $x$ are a $n$ tuples of points in the circle and $\sigma$ an element of the symmetric group $\mathfrak S(n)$  by
$$
[X,x;\sigma]\defeq \frac{\prod_{i=1}^n (X_i,x_{\sigma(i)})}{\prod_{i=1}^n(X_i,x_{i})}\ .
$$
We proved that the multifraction algebra is stable by the Poisson bracket, while it is obviously stable by multiplication.

\begin{definition}[\sc Extended swapping algebra]Let us consider the algebra ${\mathcal B}_0$ which is generated as a vector space  by the multifraction algebra to which we add extra generators denoted  $\ell_g$ for any geodesic $g$ \footnote{The generator $\ell_g$ is  formally a logarithm  $\log(g)$ of the geodesic $g$.} as well as a central element $\casper$; we  finally extend 
 the  swapping bracket to  ${\mathcal B}_0$ by adding 
 \begin{eqnarray}
		\ [\ell_g,\ell_h]\defeq \frac{1}{g\ h}[g,h]+\epsilon(g,h)\casper \ \ ,\ \ \ 
	[G,\ell_h] \defeq \frac{1}{h}[G,h] \eqdef-[\ell_h,G] \label{eq:ext-swap} \  \ . 
	 \end{eqnarray}
	 We call ${\mathcal B}_0$ with the extended swapping bracket, the {\em extended swapping algebra}.
	 
	 The reversing orientation is defined on generators by $\overline{\ell_g}=\ell_{\overline{g}}$, 
\end{definition}

We then have
\begin{proposition}\label{pro:ext-swapp}
	The extended swapping algebra is a Poisson algebra. The reversing isomorphisms antipreserves the Poisson structure:
	$
	\left[\overline{G},\overline{H}\right]=-\overline{\left[ G, H\right]}
	$.
\end{proposition}
\begin{proof}  This is just a standard check that adding ``logarithmic derivatives'' to a Poisson algebra still gives a Poisson algebra.
	We first see that that
	$$
	\partial_g: z\mapsto [\ell_g,z]=\frac{1}{g}[g,z]\ ,
	$$
	is a derivation on the fraction algebra of the swapping bracket.
	Indeed, 
		\begin{eqnarray*}
	\partial_g([z,w])&=&\frac{1}{g}[g,[z,w]] 
		=\frac{1}{g}\left([z,[g,w]]-[w,[g,z]]\right)\\
		&=&\left(\left[z,\frac{[g,w]}{g}\right]-\left[w,\frac{[z,w]}{g}\right]\right)=[z,\partial_g(w)]+[\partial_g(z),w]\ .
	\end{eqnarray*}
Moreover,  the bracket of derivation gives
	$
	[\partial_g,\partial_h](z)=[[\ell_g,\ell_h],z]$ 	Let us check this last point:
	$$
\partial_g \left(\partial_h (z)\right)= \frac{1}{g}\left[g,\frac{1}{h}[h,z]\right]=-\frac{1}{gh^2}[g,h][h,z] + \frac{1}{g h}[g,[h,z]]\  .
	$$
	Thus, we complete the proof of the proposition
	\begin{eqnarray*}
		[\partial_g,\partial_h](z)=[g,h]\left(-\frac{[h,z]}{gh^2}- \frac{[g,z]}{hg^2}\right)  + \frac{1}{gh}[[g,h],z]]=\left[\frac{[g,h]}{gh},z\right]\ .
	\end{eqnarray*}
\end{proof}
\subsubsection{Ghost algebra and the extended swapping algebra}

In the projective case, it is convenient to consider the free polynomial algebra $\mathcal A_P$ generated by the ghost polygons, and extend the {\em ghost bracket}  by the Leibnitz rule to $\mathcal A_P$.

In this paragraph, we will relate the algebras $\mathcal A_P$ and $\mathcal B_0$, more precisely we will show:
 
 \begin{theorem}
 	\label{theo:ext-ghost}
 	There exists a homomorphisms of commutative algebra map $$\pi:\mathcal A_P\to\mathcal B_0\ ,$$
which is surjective, preserves the bracket  and the reversing the orientation isomorphism:
 	$$
 	[\pi(A),\pi(B)]=\pi[A,B]\ \ \ ,\ \ \ \overline{\pi(A)}=\pi\left(\overline{A}\right)\ .
 	$$
 Finally if $A$ belongs to the kernel of $\pi$, then for any uniformly hyperbolic bundle $\rho$, $\T_A(\rho)=0$.
\end{theorem} 
 
 Thus, $\mathcal A_P/\ker(\pi)$ is identified as an algebra with  bracket with $\mathcal B_0$; in particular $\mathcal A_P/\ker(\pi)$ is a Poisson algebra.
 
 This will allow in the applications to reduce our computations to calculations in the extended swapping algebra, making use of the fact that  the extended swapping algebra is a Poisson algebra by proposition \ref{pro:ext-swapp}.
 
 Unfortunately,  we do not have the analogue of the swapping bracket in the general $\Theta$-case, although the construction and result above suggests to  find a combinatorially defined ideal $\mathcal I$ in the kernel of $\T(\rho)$ for any $\rho$, so that $\mathcal A/\mathcal I$ satisfies the Jacobi identity.
 
\subsubsection{From the ghost algebra to  the extended swapping algebra}  
In this paragraph, we define the map $\pi$ of Theorem \ref{theo:ext-ghost}. The map $\pi$  is defined  on the generators by 
\begin{eqnarray*}
	 g&\longmapsto& \pi(g)\defeq \ell_g\ , \\
	 G=\lceil g_1,\ldots ,g_p\rceil &\longmapsto&\pi(G)\defeq [X,x;\sigma]=\frac{\prod_{i=1}^n (g^+_i,g^-_{i+1})}{\prod_{i=1}^n(g^+_i,g^-_{i})}\ .
	 \end{eqnarray*}
where $X=(g^+_i)$, $x=(g^-_i)$, $\sigma(i)=i+1$.
Cyclicity is reflected by 
\begin{eqnarray}
\pi(\lceil g_1,\ldots ,g_p\rceil) &=&\pi(\lceil g_2,\ldots ,g_p, g_1 \rceil)\ \label{eq:cyclic}.
\end{eqnarray}
Conversely, we then have the following easy construction.

Let $X=(X_1,\ldots,X_k)$, $x=(x_1,\ldots, x_k)$, and $g_i$ the geodesic $(X_i,x_i)$. Let $\sigma$ be  a permutation of $\{1,\ldots,k\}$ and let us write $\sigma=\sigma_1,\ldots,\sigma_q$ be the decomposition of $\sigma$ into commuting cycles $\sigma_i$ or order $k_i$ with support $I_i$. For every $i$, let $m_i$ be in $I_i$ and let us define 
$$h^i_j= g_{\sigma_i^{j-1}(m_i)}\ \ , \ \ 
	 G_i=\lceil h^i_1, \ldots h^i_{k_i}\rceil\ . $$ 
We then have with the above notation
	\begin{equation}
		[X,x;\sigma]=\pi(G_1\ldots G_q)\ .\label{eq:cyc-dexomp-multi}	
	\end{equation}

\begin{corollary}\label{cor:pi-surj}
		The map $\pi$ is surjective.
\end{corollary}
In the sequel, the decomposition \eqref{eq:cyc-dexomp-multi} will be referred as the {\em polygonal decomposition} of the multifraction 	$[X,x;\sigma]$. We also obviously have
\begin{lemma}
Any tuples of ghost polygons is the polygonal decomposition of a multifraction.
\end{lemma}
\subsubsection{The map $\pi$ and the evaluation $\T$}

For any multifraction $B=[X,x;\sigma]$ and uniformly hyperbolic bundle $\rho$ associated to limit curves $\xi$ and dual limit curves $\xi^*$, we define
$$
\T^P_B(\rho)\defeq \frac{\prod_i \braket{V_i,v_{\sigma(i)}}}{\prod_i \braket{V_i,v_i}}\ ,
$$
where $V_i$ is a non-zero vector in $\xi^*(X_i)$ while $v_i$ is a non-zero vector in $\xi(x^i)$.

Given $\rho$, we now extend $G\mapsto\T_G(\rho)$ and $\T^P_B(\rho)$ to homomorphisms of commutative free algebras to $\mathcal A_p$ and $\mathcal B_0$.
We then have the following result which follows at once since we are only considering rank 1 projectors.
\begin{lemma}\label{lem:kerp}
	We have, for all uniformly hyperbolic bundles $\rho$ \begin{eqnarray*}
	\T^P_{\pi(G)}(\rho)=\T_G(\rho)\ \ , \ \ \T_G(\rho)=\T_{\bar G}(\rho^*)\ ,
\end{eqnarray*}
\end{lemma}
This proposition implies that for every $G$ in the kernel of $\pi$, for every $\rho$, $\T_G(\rho)=0$.

\subsubsection{Swapping bracket}

We now compute the brackets of multifractions. We shall use the notation of paragraph \ref{sec:ghost-polygon} where the opposite configuration $g^*$ of a ghost or visible  edge $g$ is defined. Observe that $g^*$ is an ordered configuration. Then we have 
\begin{proposition}[\sc Computation of swapping brackets]\label{pro:gh-swap} Let $G$ and $H$ be two multifractions that are images of ghost polygons: 
 $G = \pi(\theta_1,\ldots,\theta_{2p})$ and $H = \pi(\zeta_1,\ldots,\zeta_{2q})$. Then their swapping bracket is given by 
	\begin{eqnarray}
	[G  ,H  ] &=&
		 \left(G\  H \ \left( \sum_{i,j}	\epsilon(\zeta_j,\theta_i)(-1)^{i+j}\pi(\lceil \theta_i,\zeta_j \rceil)\right) \right)  \label{def:ghost-bracketGH}\ . \end{eqnarray}
		Moreover, for $g=(X,x)$ and $h=(Y,y)$ geodesics, we have in the fraction algebra of the swapping algebra.
		\begin{eqnarray}
		\ [\ell_h,\ell_g] &=& \left( \epsilon(g,h) \  \pi(\lceil g,h\rceil)\right)   \ .\label{def:swap-bracketgh}\\ \ 
	[G  , \ell_h]&=& \left(G\  \ \left( \sum_{i}
		\epsilon(h,\theta_i)(-1)^{i+1}\pi(\lceil \theta_i,h\rceil)\right)\right)  \  \label{def:swap-bracketgH}\label{def:ghost-bracketgH}\ . 
	 \end{eqnarray}
	 Using the notation $\theta^*_i$ for the opposite edge, we have, for every $i$ and $j$
	 \begin{eqnarray}
	 	\pi(\lceil \theta_i^*,\zeta_j^*\rceil )=G\ H \ \pi(\lceil \theta_i,\zeta_j\rceil)\label{eq:brack-opp}\ .
	 \end{eqnarray}
\end{proposition}
\begin{proof} In this proof, we will omit to write $\pi$ and confuse a ghost polygon and its image under $\pi$. Equation \eqref{def:swap-bracketgh} follows at once from the definition. Now let $G=\lceil g_1,\ldots, g_p\rceil$, let $\eta_i$ be the ghost edges joining $g_{i+1}^-$ to $g_i^+$. Then we may write in the fraction algebra of the swapping algebra
$$
\lceil g_1,\ldots, g_p\rceil  =\frac{\prod_{i=1}^p\eta_i}{\prod_{i=1}^p g_i}\ .
$$
Using logarithmic derivatives we then have
\begin{eqnarray*}
	\frac{1}{G  }[G  ,\ell_h] =\sum_{i=1}^p \left(\frac{1}{h\  \eta_i}[\eta_i,h] -\frac{1}{h \ g_i}[g_i,h] \right)=\sum_{i=1}^p \left(\epsilon(h,\eta_i)\lceil \eta_i,h\rceil  -\epsilon(h,g_i)\lceil g_i,h\rceil   \right)\ ,
\end{eqnarray*}
which gives equation \eqref{def:swap-bracketgH}. 
Writing now 
$$
G  =\lceil g_1,\ldots, g_p\rceil  =\frac{\prod_{i=1}^p\eta_i}{\prod_{i=1}^q g_i}\ \ , \ \ 
H  =\lceil h_1,\ldots, h_q\rceil  =\frac{\prod_{i=1}^q\nu_i}{\prod_{i=1}^q h_i}\ ,
$$
where $\eta_i$ and $\nu_i$ are ghost edges of $G$ and $H$ respectively, we get
\begin{eqnarray*}
\frac{[G  ,H  ] }{G  \  H  }	&=&\sum_{(i,j)}\left(\frac{1}{g_i \ h_j}[g_i,h_j]  -\frac{1}{g_i\  \nu_j}[g_i,\nu_j] +\frac{1}{\eta_i \ \nu_j}[\eta_i,\nu_j]  - \frac{1}{\eta_i \ h_j}[\eta_i,h_j] \right) \crcr
	&=&\sum_{(i,j)}\left(\epsilon(h_j,g_i)\lceil g_i,h_j\rceil    - \epsilon(\nu_j,g_i)\lceil g_i,\nu_j\rceil   +\epsilon(\nu_j,\eta_i)\lceil \eta_i,\nu_j\rceil    - \epsilon(h_j,\eta_i)\lceil \eta_i,h_j\rceil   \right)\ ,\end{eqnarray*}
which is what we wanted to prove. The equation \eqref{eq:brack-opp} follows from the definition of the map $\pi$. 
\end{proof} 
As a corollary we obtain 
	 \begin{corollary}\label{cor:pres-brak}
	  The map $\pi$ preserves the bracket.
	 \end{corollary}

\begin{proof} The proof follows at once from proposition \ref{pro:gh-intersect1} and \ref{pro:gh-intersect2} which computes the ghost intersection and recognizing each term as the correlation functions of a term obtained in the corresponding ghost bracket in proposition \ref{pro:gh-swap}.
\end{proof}
\subsubsection{Proof of Theorem \ref{theo:ext-ghost}} We have proved all that we needed to prove: 
the theorem follows from corollary \ref{cor:pres-brak} and \ref{cor:pi-surj}, as well as lemma \ref{lem:kerp}.

\subsection{Natural maps into the ghost algebra}\label{sec:natural}

\begin{definition}[\sc Natural maps]
	Let $w$ be a $p$-multilinear map from the ghost algebra to itself. We say $w$ is {\em natural}, if for tuples of integers $(n_1,\ldots,n_p)$ there exists an integer $q$, a real number $A$ such that given a tuple of ghost polygons $G=(G_1,\ldots, G_p)$ with $G_i$ in ${\GT}^{n_i}$, then 
$$
w(G_1,\ldots,G_p)=\sum_{i=1}^q \lambda_i H_i\ , 
$$
where $H_i$ are ghost polygons, $\lambda_i$ are real numbers less than $A$ and, moreover, every visible edge of $H_i$ is a visible edge of one of the  $G_i$.\footnote{The existence of $q$ is actually a consequence of the definition: there only finitely many polygons with a given set of visible edges}\end{definition}
We will extend the definition of the core diameter to any element of the ghost algebra
by writing, whenever $H_i$ are  distinct ghost polygons ghost polygons
$$
r\left(\sum_{i=1}^q \lambda_i H_i\right)\defeq \sup_{i=1,\ldots,q} \left(\vert\lambda_i\vert r(H_i)\right)\ , 
$$
We also recall that the core diameter of a ghost polygon, only depends on the set of its visible edges. We then define the core diameter of a tuple of polygons $G=(G_1,\ldots,G_n)$, as the core diameter of the union of the set of edges of the $G_i$'s.

We then have the following inequality of core diameters for a natural map $w$, $G=(G_1,\ldots,G_p)$ and $q$ and $A$ as in the definition
\begin{eqnarray}
	 r(w(G))\leq A \ r(G)\  .\label{ineq:natural-core}
\end{eqnarray}
We now give an example of a natural map
\begin{proposition}\label{pro:bracket-nat}
	The map $
	(G_1,\ldots,G_n)\mapsto[G_1,[G_2,[\ldots [G_{n-1},G_n]\ldots]]]
	$
	is a natural map.
\end{proposition}
\begin{proof}
	This follows at once from the definition of the ghost bracket and a simple induction argument.
\end{proof}

\part{On closed surfaces}

We now move to closed surfaces, or in other words we only consider periodic uniformly hyperbolic bundles which are exactly Anosov representations. In an analogue of the way that one can define length for currents, we define an {\em average correlation functions}  using {\em cyclic currents} and a careful study of integrals allows us to prove the main results announced in  the introduction.
\begin{enumerate}
	\item In section \ref{sec:current}, we introduce {\em cyclic currents} which are certain types of measures on the space of ghost polygons. This allows us to define what are averaged correlations function $\T_\mu$ when $\mu$ is a cyclic current. 
	\item In section \ref{sec:exch-int}, we deal with the following crucial technical and dry issue: how to exchange ghost integration and integration with respect to a cyclic current? The results are necessary to proceed.
	\item In section \ref{sec:average}, using the combinatorial work of the first part and the previous sections,  we are finally able to prove the main result of the article a {\em combinatorial formula to compute the Poisson bracket of correlation functions}. This is obtained by a see-saw argument: we first compute the derivatives of length functions and correlations functions, then this allows us to recognize hamiltonian functions using the ghost dual forms we defined for ghost polygons, and finally we can compute the Poisson bracket of correlations and length functions.
\end{enumerate}

\section{Geodesic and cyclic currents}
\label{sec:current}
In this section, building on the classical notion of geodesic currents introduced by Bonahon in \cite{Bonahon:1988}, we define the notion of higher order geodesic currents, called {\em cyclic  currents}. Among them we identify {\em integrable currents}, show how they can average correlation functions and produce examples of them.

Recall that $\GG$ is the set of oriented geodesics in $\hh$. The set of $\Theta$-geodesics is then denoted ${\GT}\defeq\GG\times\Theta$.

\subsection{Cyclic current} First recall that a {\em signed measure} is a linear combination of finitely many {\em positive measures}. Any signed measure is the difference of two positive measures. A {\em cyclic current} is a $\Gamma$-invariant signed measure on $\GG^n$ invariant under cyclic permutation.
As a first  example let us consider for $\mu$ and $\nu$  geodesic current, the signed measure  $\mu\wedge\nu$ given by 
	\begin{equation}
	\mu\wedge\nu\defeq\frac{1}{2}\epsilon\ (\mu\otimes \nu -\nu\otimes\mu)\ ,\label{eq:def-bra-geod}
	\end{equation}
	where we recall that $\epsilon(g,h)$ is the intersection number of the two geodesics $g$ and $h$.
\begin{lemma}
	The signed measure $\mu\wedge\nu$ is a cyclic current supported on intersecting geodesics. Moreover $\mu\wedge\nu=-\nu\wedge\mu$.
\end{lemma}	
	\begin{proof}
		We have
		\begin{eqnarray*}
				2\int_{{\GT}^2/\Gamma} f(g,h)\ \d \mu\wedge\nu(g,h)=
		\int_{{\GT}^2/\Gamma} f(g,h)\ \epsilon(g,h)\left(\d \mu(g) \d\nu(h)-\d \nu(g) \d\mu(h) \right)\crcr
		=
		\int_{{\GT}^2/\Gamma} f(h,g)\ \epsilon(h,g)\left(\d \mu(h) \d\nu(g)-\d \nu(h) \d\mu(g) \right)\crcr
			=
		\int_{{\GT}^2/\Gamma} f(h,g)\ \epsilon(g,h)\left(\d \nu(h) \d\mu(g)-\d \mu(h) \d\nu(g) \right)\crcr
		=2\int_{{\GT}^2/\Gamma} f(h,g)\ \d \mu\wedge\nu(g,h)\  . \ \ \ \ \ \ \ \ \ \ \ \ \ \ \ \  \ \ \ \ \ \ \ \ \ \ \ \ \ \ \ \ \ \ 
		\end{eqnarray*}
	Hence $\mu\wedge\nu$ is cyclic. The last assertions are obvious. 
\end{proof}

Our main definition is the following, let $\rho$ be a $\Theta$-Anosov representation of $\Gamma$, the fundamental group of a closed surface.

\begin{definition}[\sc Integrable currents]\label{def:integrable-current} We give several definitions, let $w$ be a natural map from $\GG^{p_1}\times\cdots\times\GG^{p_q}$ to $\GG^m$
\begin{enumerate}
\item a {\em $w$-cyclic current} is a $\Gamma$-invariant measure  $\mu=\mu_1\otimes\cdots\otimes\mu_q$ on $\mathcal C^k$ where $\mu_i$  are $\Gamma$-invariant cyclic currents on $\GG^{p_i}$,
\item the $w$-cyclic current  $\mu$  is a {\em $(\rho,w)$-integrable} current if there exists a neighborhood $U$ of $\rho$ in the moduli space of (complexified)  $\Theta$-Anosov representations of $\Gamma$, and a positive function $F$ in $L^1({\GT}_\star^k/\Gamma,\mu)$ so that for all $\sigma$ in $U$,  and $G$ in ${\GT}_\star^k$; 
$$\vert\T_{w(G)}(\sigma)\vert \leq F_0(G)\ ,$$
where $F_0$ is the lift of  $F$ to $\GT_\star^k$.
	\item When $w$ is the identity map $\Id$, we just say a current is {\em $\rho$-integrable}, instead of $(\rho,\Id)$-integrable. 
	\item  A current of order $k$, is  {\em $w$-integrable} or {\em integrable} if it is $(\rho,w)$-integrable or $\rho$-integrable for all  representations $\rho$.
\end{enumerate}

\end{definition}

\subsubsection{$\Gamma$-compact currents}

\begin{definition}
A  $\Gamma$-invariant $w$-cyclic current  $\mu$ is {\em $\Gamma$-compact} if it is supported on a $\Gamma$-compact set of ${\GT}_\star^p$. Obviously a $\Gamma$-compact cyclic current is integrable for any natural map $w$. 
\end{definition}

Here is an important example of a $\Gamma$-compact cyclic current: 
Let $\mathcal L$ be a geodesic lamination on $S$  with a component of its complement $C$ being a geodesic triangle. Let $\pi:\hh\mapsto S$ be the universal covering of $S$ and $x$ a point in $C$

Then $$
\pi^{-1}C=\bigsqcup_{i\in\pi^{-1}(x)} C_i	\ .
$$
The closure of each $C_i$ is an ideal triangle with cyclically ordered edges $(g_i^1,g_i^2,g_i^3)$. We consider the opposite cyclic ordering  $(g_i^3,g_i^2,g_i^1)$. The notation $\delta_x$ denotes the Dirac measure on $X$ supported on a point $x$ of $X$.  Then we obviously have
\begin{proposition}
	The measure defined on ${\GT}^3$ by
	$$
	\mu^*_C=\frac{1}{3}\sum_{i\in\pi^{-1}(x)}\left(\delta_{(g_i^1,g_i^3,g_i^2)}+\delta_{(g_i^2,g_i^1,g_i^3)}+\delta_{(g_i^3,g_i^2,g_i^1)}\right)\ ,	$$
	is a $\Gamma$-compact cyclic current.
\end{proposition}

\subsubsection{Intersecting geodesics}
Let us give an example of an integrable current. 

\begin{proposition}\label{pro:mu-intersect}
	Let $\mu$ be $\Gamma$-invariant cyclic current supported on pairs of intersecting geodesics. Assume furthermore that $\mu({\GT}^2/\Gamma)$ is finite. Then $\mu$ is integrable.
\end{proposition}
This follows at once from the following 
 lemma.
\begin{lemma}\label{lem:Krho}
	Let $\rho_0$ be a $\Theta$-Anosov representation. Then there exists a constant $K_\rho$ in an neighborhood $U$ of $\rho_0$ in the moduli space of Anosov representations, such that for any $\rho$ in $U$ and any pair of intersecting geodesics
	$$
	\vert\T_{\lceil g, h\rceil}(\rho)\vert\leq K_\rho \ .
	$$
\end{lemma}
\begin{proof} 
 Given any pair of  geodesics $(g_1,g_0)$ intersecting at a point $x$, then we can find an element $\gamma$ in $\Gamma$, so that $\gamma x$ belongs to a fundamental domain $V$ of $\Gamma$.
  In particular, there exists a pair of geodesics $h_0$ and $h_1$ passing though $V$ so that 
$$
\T_{\lceil g_0,g_1\rceil}(\rho)=\T_{\lceil h_0,h_1\rceil}(\rho)=\tr\left(\pp_\rho(h_0)\ \pp_\rho(h_1)\right)\ ,
$$
where $\pp_\rho$ is the fundamental projector for $\rho$. Since the set of geodesics passing through  $V$ is relatively compact, the result follows by the continuity of the fundamental projector $\pp_\rho(h)$ on $h$ and $\rho$.  
\end{proof}

\begin{corollary}\label{cor_int_wedge}
	Given $\mu$ and $\nu$, then $\mu\wedge\nu$ is integrable.
	\end{corollary}
\begin{proof}
Let \begin{eqnarray*}
	A\defeq\left\{(g,h)\mid \epsilon(g,h) = \pm 1\right\}\ , \ 	B\defeq\left\{(g,h)\mid \epsilon(g,h) = \pm\frac{1}{2}\right\}\ .
	\end{eqnarray*}
Observe first that denoting  the Bonahon intersection by $i$, we have
$$
\left\vert\mu\wedge\nu\ \left(A/\Gamma\right)\right\vert\leq i(\bar\mu,\bar\nu)<\infty\ , $$ where the last inequality is due to Bonahon \cite{Bonahon:1988}, and
 $\bar\lambda$ is the symmetrised current of $\lambda$. 
 
As $\Gamma$ acts with compact quotient on the set of triples of points on $\partial\hh$, it follows that  $\Gamma$ acts on $B$ with compact quotient and therefore $\mu\wedge\nu(B)$ is finite. Therefore taking the sum we have that $\mu\wedge\nu(\GG^2/\Gamma)$ is finite.
\end{proof}

\subsubsection{A side remark} Here is an example of $(\rho,w)$-integrable current. First we have the following inequality: given a representation $\rho_0$, there is a constant $K_0$, a neighborhood $U$ of $\rho_0$, such that for every $k$-configuration $G$ of geodesics and $\rho$ in $U$ then
$$
\vert \T_G(\rho)\vert\leq e^{kK_0\ r(G)}\ .
$$
Since this is just a pedagogical remark that we shall not use, we do not fill the details of the proof.
From that inequality we see that if $G\mapsto e^{kK_0\ r(G)}$ is in $L^1({\GT}_\star^k/\Gamma,\mu)$ then $\mu$ is $(\rho,w)$-integrable. 
\section{Exchanging integrals} \label{sec:exch-int}
To  use ghost integration to compute the Hamiltonian of  the average of correlation functions with respect to an integrable  current, we will need to exchange integrals.

This section is concerned with proving the two Fubini-type exchange theorems we will need. Recall that the form $\beta_{\rho(g)}$ is defined in equation \eqref{def:beta} and    $\bLa^\infty(E)$ in definition \ref{def:bounded-1-form} .
\begin{theorem}[\sc Exchanging integrals I]\label{theo:oint-inter1}
	Let $\mu$ be a $\Gamma$-invariant geodesic current. Let $G$ be a $\Theta$-ghost polygon. Then \begin{enumerate} 
			\item    The map  $y\mapsto\int_{{\GT}} \beta_{\rho(g)}(y)\d \mu(g)$ --- where $y$ is in $\uh$--- is an element of $ \bLa^\infty(E)$, 
			\item the map $g\mapsto \oint_{\rho(G)}\beta_{\rho(g)}$ is in $L^1({\GT},\mu)$,
\item  finally, we have the exchange formula
	\begin{eqnarray}
	\oint_{\rho(G)}\left(\int_{\GT} \beta_{\rho(g)}\ \d\mu(g)\right)=\int_{\GT} \left(\oint_{\rho(G)}\beta_{\rho(g)}\right)\d\mu(g)\ .\label{eq:oint-inter}
	\end{eqnarray}
	\end{enumerate}
\end{theorem}

Similarly, we have  a result concerning ghost intersection forms.  We have to state it independently in order to clarify the statement. Let us first extend the assigement $G\mapsto \Omega_G$ by linearity to the whole ghost algebra, and observe that if  we have distinct  ghost polygons $G_i$ and 
$$
H=\sum_{i=1}^q \lambda_iG_i\ ,\hbox{ with } \sup_{i\in\{1,\ldots,q\}}\vert\lambda_i\vert=A\ ,
$$
Then 
$$
\Vert \Omega_H(y)\Vert\leq qA\ \sup_{i\in\{1,\ldots,q\}}\Vert\Omega_{G_i}(y)\Vert\ .
$$

\begin{theorem}[\sc Exchanging integrals II]
	\label{theo:oint-exchang2}
	Let $\mu$ be a $w$-cyclic and $\Gamma$-compact current of rank $p$. Let $G$ be a ghost polygon. Let $w$ be a  natural map. Then \begin{enumerate} 
			\item the map $y\mapsto \int_{{\GT}^p} \Omega_{\rho(w(H))}(y)\d \mu(H)$ --- where $y$ is an element of $\uh$ --- is an element of $ \bLa^\infty(E)$,\item the map $H\mapsto \oint_{\rho(G)}\Omega_{\rho(w(H))}$ is in $L^1({\GT}^p,\mu)$,
\item  finally, we have the exchange formula
	\begin{eqnarray}
	\oint_{\rho(G)}\left(\int_{{\GT}^p} \Omega_{\rho(w(H))}\ \d\mu(H)\right)=\int_{{\GT}^p} \left(\oint_{\rho(G)}\Omega_{\rho(w(H))}\right)\d\mu(H)\ .\label{eq:oint-inter2}
	\end{eqnarray}
	\end{enumerate}
\end{theorem}

We first concentrate on Theorem \ref{theo:oint-inter1}, then prove Theorem \ref{theo:oint-exchang2} in paragraph \ref{sec:proo-ointexchang2}.

\subsection{Exchanging line integrals}

Theorem \ref{theo:oint-inter1} is  an immediate consequence of a similar result involving line integrals. 
\begin{proposition}\label{pro:exchan-int1}
	Let $\mu$ be a $\Gamma$-invariant geodesic current  on ${\GT}$,  then \begin{enumerate} 
			\item The map $y\mapsto \int_{{\GT}} \beta_{{\rho(g)}}\d \mu(g)$ --- where $y$ is an element of $\uh$ --- is an element of $ \bLa^\infty(E)$,
			\item Let $g_0$ be a geodesic, $x$ a point on $g_0$ and $Q$ a parallel section of $\End(E)$ along $g_0$, then the map $$
		 g\mapsto \bS_{x,g_0,Q}\left(\beta_{{\rho(g)}}\right)\ ,$$ is in $L^1({\GT},\mu)$.
\item  We have the exchange formula
	\begin{eqnarray}
		\bS_{x,g_0,Q}\left(\int_{{\GT}} \beta_{\rho(g)}\ \d\mu(g)\right)=\int_{{\GT}} \bS_{x,g_0,Q}\left(\beta_{\rho(g)}\right)\ \d\mu(g)\ .\label{eq:exchan-int1}
	\end{eqnarray}
	\end{enumerate}
\end{proposition}
We prove the first item in proposition \ref{sec:proof-1item-exchg}, the second item in \ref{sec:proof-2item-exchange} and the third in \ref{sec:proof-thirditem-exchange}.
 \subsection{Average of geodesic forms and the first item}\label{sec:proof-1item-exchg} 
Let $\mu$ be a $\Gamma$-invariant measure on ${\GT}$.  Let $y$ be a point in $\hh$, and $$
G(y,R)\defeq\{g\in{\GT}\mid d(g,y)\leq R\}\ .
$$ 
As an immediate consequence of the $\Gamma$-invariance we have
\begin{proposition}\label{pro:bd-muB} For every positive $R$, there is a constant $K(R)$ so that for every $y$ in $\hh$
	\begin{equation}
		\mu(G(y,R))\leq K(R)\ . \label{ineq:bd-muB}
	\end{equation}
\end{proposition}
Observe now that if $g$ is not in $G(y)\defeq G(y,2)$, then $y$ is not in the support of $\omega_g$ and thus $\beta_{\rho(g)}(y)=0$. We then define

\begin{definition}
	The {\em $\mu$-integral of geodesic forms} is the form $\alpha$ so that at a point $y$ in $\hh$
	$$
	\alpha_y\defeq\int_{G(y)} \beta_{\rho(g)}(y)\ \d\mu(g)=\int_{\GT} \beta_{\rho(g)}(y)\ \d\mu(g)\ .
	$$
We use some abuse of language and write  
$$
\alpha\eqdef\int_{\GT}\beta_{\rho(g)}\d\mu(g)\ .
$$
\end{definition}
The form $\alpha_y$ is well defined since $G(y)$ is compact. Moreover, the next lemma gives the proof of the first item of proposition \ref{pro:exchan-int1}
\begin{lemma}\label{lem:mub-bLa}
	The $\mu$-integral of geodesic forms belongs to $ \bLa^\infty(E)$ and we have a constant $K_5$ only depending on $\rho$ and $\mu$ so that 
	\begin{equation}
		\left\Vert\int_{\GT}\beta_{\rho(g)}\d\mu(g)\right\Vert_\infty\leq K_5 
		\ .
	\end{equation}
\end{lemma}
\begin{proof}  We have 
\begin{equation}
		\left\vert\int_{\GT}\beta_{\rho(g)}(y)\ \d\mu(g)\right\vert=\left\vert\int_{G(y)}\beta_{\rho(g)}(y)\ \d\mu(g)\right\vert \leq \mu(G(y))\ \ \sup_{g\in G(y)}\left\Vert \beta_{\rho(g)}\right\Vert_\infty \ .
	\end{equation}
Then by proposition \ref{pro:bd-muB}, there is a constant $k_1$ so that $\mu(G(y))\leq k_1$. Recall that $\beta_{\rho(g)}=\omega_g\pp(g)$. Then by the $\psld$ equivariance, $\omega_g$ is bounded independently of $g$, while by lemma \ref{pro:pbd-unique}, $\pp$ is a bounded section of $\End(E)$. The result follows.
\end{proof}

\subsection{Decay of line integrals}
We now recall the following definition.
\begin{eqnarray*}
	\bS_{x,g,\QQ}(\omega)	& = & 
		 \int_{g^+}\tr\left(\omega\ \pp\  [\pp,\QQ] \right) + \int_{g^-}\tr\left(\omega\ [\pp,\QQ]\ \pp\right)\ .
\end{eqnarray*}
We prove in this paragraph the following two lemmas.
\begin{lemma}[\sc Vanishing]\label{lem:vanishing} Let $g_0$ be a geodesic and $x$ a point in $g_0$, 
	Let $g$ be a geodesic such that $d(g,g_0)>1$, then for any function $\psi$ on $\hh$ with values in $[0,1]$:
	\begin{equation}
		\bS_{x,g_0,Q}(\psi\beta_{\rho(g)})=0\ .
	\end{equation}
\end{lemma}
\begin{proof} This follows at once from the fact that under the stated hypothesis, the support of $\omega_g$ does not intersect $g_0$. \end{proof}

\begin{lemma}[\sc Decay]
	\label{lem:expdecay-mu}
	For any endomorphism $\QQ$ and representation $\rho$, there exist positive constants $K$ and $k$, so that for all $g$ so that $d(g,x)>R$, for any function $\psi$ on $\hh$ with values in $[0,1]$:
	\begin{eqnarray*}
		\left\vert \bS_{x,g_0,Q}(\psi\beta_{\rho(g)})\right\vert\leq K e^{-kR} .
	\end{eqnarray*}
\end{lemma}

\begin{proof}We assume $x$ and $g$ are so that $d(g,x)> R$. 
It is enough (using a symmetric argument for $g_0^-$) to show that 	
\begin{eqnarray*}
		\left\vert\int_{g_0^+}\tr(\psi\beta_{\rho(g)}\cdot  \pp\cdot[\pp,\QQ])\right\vert\leq K e^{-kR}\ ,
	\end{eqnarray*}
	where $g_0^+$ is the arc on $g_0$ from $x$ to $+\infty$.
Let us denote by $g_0^{+}(R)$ the set of points of $g_0^+$ at distance at least $R$ from $x$:
$$
g_0^{+}(R)\defeq\{y\in g_0^+\mid d(y,x)\geq R\}\ .
$$ 
Then if $y$  belongs to $g_0^+$ and does not belong to  $g_0^{+}(R-1)$, then $d(y,x)<R-1$. Thus $d(y,g)>1$. Thus, by lemma \ref{lem:vanishing}, $\beta_{\rho(g)}(y)$ vanishes for $y$ in $g_0^+$ and not in $g_0^{+}(R-1)$. Thus 
\begin{eqnarray*}
		\left\vert\int_{g_0^+}\tr(\psi\beta_{\rho(g)}\cdot \pp\cdot [\pp,\QQ])\right\vert\leq \int_{g_0^+(R-1)}\left\vert\tr(\beta_{\rho(g)}(\pt)\cdot \pp\cdot [\pp,\QQ])\right\vert\ \d t\ .
	\end{eqnarray*}
Then the result follows from the exponential decay lemma \ref{lem:exp-decay}.
\end{proof}

\subsection{Cutting in pieces and dominating: the second item}\label{sec:proof-2item-exchange}
We need to decompose ${\GT}$ into pieces. Let $g_0$ be an element of 
${\GT}$ and $x$ a point on $g_0$. Let $x^+(n)$ -- respectively $x^-(n)$ -- the point in $g_0^+$ -- respectively $g^-_0$ -- at distance $n$ from $x$.
Let us consider \begin{eqnarray*}
	U_0&\defeq& \{g\in{\GT}\mid d(g,g_0)>1\}\ ,\crcr
	V^+_n&\defeq&\{g\in{\GT}\mid d(g,x^+(n))< 2 \hbox{ and for all $0\leq p<n$ ,}\ \  d(g,x^+(n))\geq  2  \} \ , \cr 
	V^-_n&\defeq&\{g\in{\GT}\mid d(g,x^-(n))< 2 \hbox{ and for all $0\leq p<n$, \ \ } d(g,x^-(n))\geq 2 \} \ .
 \end{eqnarray*}
	 This gives a covering of ${\GT}$:		 
\begin{lemma}\label{lem:GG-decomp}
We have the decomposition
	$$
{\GT}=U_0\cup \bigcup_{n\in\mathbb N} V^\pm_n\ , 
$$
\end{lemma}
\begin{proof} When $g$ does not belong to $U_0$, there is some $y$ in $g_0$ so that $d(g,y)\leq 1$, hence some $n$ so that either $d(y,g^+(n))\leq 2$, while for all $0\leq p<n$ we have  $d(y,g^+(p))> 2$,  or $d(y,g^-(n))\leq 2$, while for all $0\leq p<n$ we have  $d(y,g^-(p))> 2$.
\end{proof}	
Let now 
$$
n(g)=\sup\{m\in\mathbb N\mid g\in V^+_m\cup V^-_m\}\ .
$$
By convention, we write  $n(g)=+\infty$, whenever $g$ does not belong to $\bigcup_{n\in\mathbb N} V^\pm_n$.

The non-negative  {\em control function} $F_0$ on ${\GT}$ is defined by  $F_0(g)=e^{-n(g)}$.
We now prove
\begin{lemma}[\sc Dominating lemma]\label{lem:domination}
	For any positive $k$, the function $(F_0)^k$ is in $L^1({\GT},\mu)$. 
	
	Moreover, there exist  positive constants $K_9$ and $k_9$ so that for all functions $\psi$ on $\hh$ with values in $[0,1]$ we have 
	\begin{equation}
		\left\vert \bS_{x,g_0,Q}(\psi\beta_g)\right\vert\leq K_9(F_0(g))^{k_9}\ .
	\end{equation}
\end{lemma} 
We now observe that the second item of proposition \ref{pro:exchan-int1} is an immediate consequence of this lemma.

\begin{proof} We first prove that $F_0$ and all its powers are   in $L^1({\GT},\mu)$. 
Observe that $V^\pm_n\subset G(x^\pm(n),2)$. It follows from proposition \ref{pro:bd-muB} that $\mu(V^\pm_n)\leq K(2)$. Moreover, for any $g$ in $V_n^{\pm}$, $F_0(g)^k\leq e^{-kn}$. The decomposition of lemma \ref{lem:GG-decomp} implies that $F_0^k$ is in $L^1({\GT},\mu)$.

Let $g$ be an element of ${\GT}$. 
\begin{enumerate}
	\item When  $g$ belongs to $U_0$, then by lemma \ref{lem:vanishing}, $\bS_{x,g_0,Q}(\beta_{\rho(g)})=0$. Hence $\left\vert \bS_{x,g_0,Q}(\beta_{\rho(g)})\right\vert \leq A (F_0(g))^a$, for any positive  $A$ and $a$.
	\item  When  $g$ does not belong to $U_0$, then  $g$ belongs to $V^\pm_{n(g)}$ with $n(g)<\infty$. By  lemma \ref{lem:hyp}, we have $d(x,g)\geq n(g)$. It follows from lemma \ref{lem:expdecay-mu} that for any positive function $\psi$, we have 
	$$
	\left\vert \bS_{x,g_0,Q}(\psi\beta_{\rho(g)})\right\vert\leq Ke^{-kn(g)}=KF_0(g)^k\ .
	$$
	\end{enumerate}
The last inequality concludes the proof.
\end{proof}

\subsection{Proof of the exchange formula of proposition \ref{pro:exchan-int1}}\label{sec:proof-thirditem-exchange}

Let us choose,  for any positive real $R$, a {\em cut-off function} $\psi_R$, namely a function on $\hh$ with values in $[0,1]$, with support in the ball with center $x$ and radius $R+1$, and equal to 1 on the ball of radius $x$ and radius $R$.
We write 
\begin{eqnarray}
	\left\vert \bS_{x,g_0,Q}\left(\int_{\GT} \beta_{\rho(g)}\ \d\mu(g)\right)-\int_{\GT} \bS_{x,g_0,Q} \left(\beta_{\rho(g)}\right)\d\mu(g)\right\vert\leq A(R)+B(R)+C(R)\ , \label{eq:exchgABC}
\end{eqnarray}
where
\begin{eqnarray*}
	A(R)&=& \left\vert \bS_{x,g_0,Q}\left(\int_{\GT} \beta_{\rho(g)}\ \d\mu(g)\right)-\bS_{x,g_0,Q}\left(\psi_R\int_{\GT}   \beta_{\rho(g)}\d\mu(g)\right)\right\vert\ ,\crcr
		B(R)&=& \left\vert \bS_{x,g_0,Q}\left(\psi_R\int_{\GT} \beta_{\rho(g)}\ \d\mu(g)\right)-\int_{\GT} \bS_{x,g_0,Q}\left(\psi_R \beta_{\rho(g)}\right)\ \d\mu(g)\right\vert\ ,\crcr
		C(R)&=&\left\vert\int_{\GT} \bS_{x,g_0,Q}\left( \psi_R \beta_{\rho(g)}\right)\ \d\mu(g)-\int_{\GT} \bS_{x,g_0,Q}\left( \beta_{\rho(g)}\right)\ \d\mu(g) \right\vert\ .
\end{eqnarray*}
We will prove the exchange formula (the third item of proposition \ref{pro:exchan-int1}) as an immediate consequence of the following three steps
\vskip 0.2 truecm
\noindent{\sc Step 1:} By lemma \ref{lem:mub-bLa},  $\alpha=\int_{\GT}\beta_{\rho(g)}\d\mu_g$  is in $ \bLa^\infty(E)$. By definition of a cutoff function,  the support of $(1-\psi(R))\ \alpha$  vanishes at any point $y$ so that $d(x,y)<R$. Thus the  exponential decay lemma \ref{lem:exp-decay} guarantees that 
$$
A(R)=\vert \bS_{x,g_0,Q}\left((1-\psi(R))\ \alpha\right)\vert \leq K_4e^{-k_4R}\Vert \alpha\Vert_\infty\ .
$$ 
Hence $\lim_{R\to\infty}A(R)=0$.
\vskip 0.2 truecm
\noindent{\sc Step 2:} Observe that 
$$
\psi_R\int_{\GT} \beta_{\rho(g)}\ \d\mu(g)=\int_{\GT} \psi_R\beta_{\rho(g)}\ \d\mu(g)\ .
$$
Moreover the function $g\mapsto\psi_R\beta_{\rho(g)}$ is continuous from ${\GT}$ to $ \bLa^\infty(E)$. Thus  follows from the continuity of $\bS_{x,g_0,Q}$ proved in proposition 
\ref{pro:J-cont} implies that $B(R)=0$.

\vskip 0.2 truecm
\noindent{\sc Final Step:} As a consequence of Lebesgue's dominated convergence theorem and the domination proved in lemma \ref{lem:domination}, we have that
$\lim_{R\to\infty} C(R)=0$.
\vskip 0.1 truecm
Combining all steps
	$$
	\lim_{R\to\infty}(A(R)+B(R)+C(R))=0\ .
	$$ 
	Hence thanks to equation \eqref{eq:exchgABC}, we have
	$$\bS_{x,g_0,Q}\left(\int_{\GT} \beta_{\rho(g)}\ \d\mu(g)\right)=\int_{\GT} \bS_{x,g_0,Q} \left(\beta_{\rho(g)}\right)\ \d\mu(g)\ .$$
\subsection{Proof of Theorem \ref{theo:oint-exchang2}}\label{sec:proo-ointexchang2}
 We assume now that $\mu$ is a $\Gamma$-compact current of order $k>1$. We may also assume -- by decomposing into the positive and negative part, that $\mu$ is a positive current.
\begin{proof}[Proof of the first item] We want to show that
$\int_{{\GT}^p} \Omega_{\rho(w(H))}\d \mu(H)$ --- defined pointwise --- is an element of $ \bLa^\infty(E)$. 

Since $\mu$ is $\Gamma$-compact, it follows that the core diameter of any $H$ in the support of $\mu$ is bounded by some constant $R_0$ by proposition \ref{pro:core-cont}.

It will be enough to prove that 
\begin{eqnarray*}
\int_{{\GT}^p}\Vert\Omega_{\rho(w(H))}(y)\Vert_y\ \d \mu(H)&\leq & K_0\ , \label{eq:GG-Ome-K0}
	\end{eqnarray*}	 for some constant  $K_0$ that depends on $\mu$. 
Let  $\mathcal K$  be a fundamental domain for the action of $\Gamma$ on ${\GT}^p$.	Observe now that 
\begin{eqnarray*}	
\int_{{\GT}^p}\Vert\Omega_{\rho(w(H))}(y)\Vert_y\ \d \mu(H)
	&=&  \sum_{\gamma\in\Gamma}\int_{\gamma \mathcal K}\Vert\Omega_{\rho(w(H))}(y)\Vert_y\ \d\mu(H)\\ =\int_{\mathcal K}\left(\sum_{\gamma\in\Gamma}\Vert\Omega_{\rho(w(H))}(\gamma(y))\Vert_y\right)\ \d\mu(H)&=&\int_\mathcal K \Vert\psi_{w(H),y}\Vert_{\ell^1(\Gamma)}\ \d\mu(H)\ ,
	\end{eqnarray*}
	where $$
\psi_{w(H),y}\ :\ \gamma\ \mapsto \Vert\Omega_{\rho(w(H))}(\gamma(y))\Vert_y\ .
$$
By the second assertion of corollary \ref{coro:bd-Om}, the  map $\psi_{H,y}$ is in $\ell^1(\Gamma)$ and its norm is bounded by a continuous function of the core diameter $r(w(H))$ of $w(H)$, hence by a continuous function of $r(H)$ by inequality \eqref{ineq:natural-core}, hence by a constant on the support of $\mu$, since $r$ is $\Gamma$-invariant and continuous by proposition \ref{pro:core-cont} and $\mu$ is $\Gamma$-compact.

Since  $r(H)$ is bounded on the support of $\mu$,  
the  first item of the theorem follows.
\end{proof}

\begin{proof}[Proof of the second item] 
		
Let us consider the map 
 $$
 \Psi: H\mapsto\oint_{\rho(G)}\Omega_{\rho(w(H))}=\int_\hh \tr\left(\Omega_{\rho(w(H))}\wedge \Omega_{\rho(G)}\right)\  ,
 $$  where we used formula \eqref{eq:Omega-int} in the last equality. Our goal is to prove $\Psi$ is in $L^1({\GT}^p,\mu)$. We have that 
$$
\Vert \Omega_{\rho(w(H))}\wedge \Omega_{\rho(G)}(y)\Vert\leq \Vert \Omega_{\rho(w(H))}(y)\Vert\ \Vert \Omega_{\rho(G)}(y)\Vert \ . 
$$
It follows that 
\begin{eqnarray*}
\int_{{\GT}^p} \vert \Psi(H)\vert \ \d\mu(H) &\leq& \int_{{\GT}^p} \int_\hh \Vert \Omega_{\rho(G)}(y)\Vert\   \Vert \Omega_{\rho(w(H))}(y)\Vert\ \d y \ \d\mu(H)\ 	\cr &\leq &  \int_\hh \Vert\Omega_{\rho(G)}(y)\Vert\  \left( \int_{{\GT}^p}   \Vert \Omega_{\rho(w(H))}(y)\Vert\ \ \d\mu(H)\right)\ \d y \cr 
&\leq&  K_0\int_\hh \Vert\Omega_{\rho(G)}(y)\Vert\ \d y =  K_0\ \  \left\Vert\Omega_{\rho(G)}\right\Vert_{L^1(\hh)}\ ,
\end{eqnarray*}
where we used the inequality in the first item in the third inequality. We can now conclude by using the first assertion, corollary \ref{coro:bd-Om}.\end{proof}

\begin{proof}[Proof of the exchange formula] We use again a family of cutoff functions $\{\psi_n\}_{n\in\mathbb N}$ defined on ${\GT}^p$ with values in $[0,1]$ so that each $\psi_n$ has a compact support, and  $\psi_n$ converges to $1$ uniformly on every compact set. 

It follows from the Lebesgue's dominated convergence theorem and the second item  that 
\begin{eqnarray}
		\lim_{n\to\infty} \int_{{\GT}^p}\left(
		\oint_{\rho(G)}\Omega_{\rho(w(H))}
		\right)
		\psi_n\ \d\mu(H) =\int_{{\GT}^p}\left(\oint_{\rho(G)}\Omega_{\rho(w(H))}\right) \ \d\mu(H)\ .\label{ineq:ointoint-ex1}
\end{eqnarray}
Recall now that by the last assertion of corollary \ref{coro:bd-Om},  $\Vert\Omega_{\rho(H)}\Vert_\infty$ is bounded on every compact set and $\Gamma$-invariant,  hence bounded on the support of $\mu$.  Thus we have   the following convergence in $\bLa^\infty(E)$
\begin{eqnarray}
		\lim_{n\to\infty} \int_{{\GT}^p}\Omega_{\rho(w(H))}
	\ \psi_n\ \d\mu(H) =\int_{{\GT}^p}\Omega_{\rho(w(H))} \ \d\mu(H)\ ,\label{eq:conv-ointoint}
	\end{eqnarray}
From the continuity obtained in proposition \ref{pro:oint-cont}, we then have  that 
\begin{eqnarray}
		\lim_{n\to\infty}\oint_{\rho(G)} \int_{{\GT}^p}\Omega_{\rho(w(H))} \ 
	\psi_n\ \d\mu(H) =\oint_{\rho(G)}\int_{{\GT}^p}\Omega_{\rho(w(H))} \ \d\mu(H)\ .\label{ineq:ointoint-ex2}
\end{eqnarray}
Finally, for every $n$, since $\psi_n$ has compact support
the  following  formula holds
$$
\oint_{\rho(G)}\left( \int_{{\GT}^p}\Omega_{\rho(w(H))} \ 
	\psi_n\ \d\mu(w(H))\right)= \int_{{\GT}^p}\left(\oint_{\rho(G)}\Omega_{\rho(w(H))}\right) \ 
	\psi_n\ \d\mu(H)\ .
	$$
The exchange formula now follows from both assertions \eqref{ineq:ointoint-ex1} and 
\eqref{ineq:ointoint-ex2}.
\end{proof}

\section{Hamiltonian and brackets: average of correlation  and length functions}\label{sec:average}

Recall that in this part, we have left the realm of uniformly hyperbolic bundles in general and focus only on periodic ones. This corresponds to the study of Anosov representations of the fundamental group of a closed surface.

The fact that $S$ is closed allows us to introduce a new structure: the smooth part of the representation variety of projective representations carries the Goldman symplectic form, defined in paragraph \ref{sec:gold-sympl}, see also \cite{Labourie:2013ka}. Hence we have a Poisson bracket on functions on the character variety.

In this section, we recall the definition of  {\em length functions} in paragraph \ref{sec:defi-length} and introduce {\em  averaged correlation functions}  in paragraph \ref{sec:defi-corell}. The main results are then described in paragraph \ref{sec:main-result}.

The proof of these results occupy the rest of the section.
\begin{enumerate}
	\item In paragraph \ref{sec:regu}, we prove the regularity of the functions that we consider.
	\item In paragrah \ref{sec:hamil-length},  we use the result about derivatives of length functions to identify the Hamiltonian vector field of these length functions, and compute their Poisson brackets.
	\item In paragraph \ref{sec:bra-length-discrete}, we compute the bracket of a length function with a (non averaged) correlation function, using the formula obtained in proposition \ref{pro:dT-integ} relating the derivative of a correlation function with the ghost integration.
	\item In paragraph \ref{sec:bra-length-corel}, we move to compute the Poisson bracket of a length function with an averaged  correlation function.
	\item In paragraph \ref{sec:hamil-corell} we use the previous computation to identify the Hamiltonian of a correlation function.
	\item Finally in paragraph \ref{sec:bra-corell}, we compute the Poisson bracket of correlation functions.
\end{enumerate}
All these computations rely on the technical results  dealt with in the previous section: exchanging integrals.

\subsection{Averaged length function: definition}\label{sec:defi-length}
 As a first step in the construction, let us consider a $\Theta$-decorated current $\mu^\aa$ supported on $\GG\times\{\aa\}$ where $\aa$ is in $\Theta$. The associated {\em length function} on the character variety  of Anosov  representations is the function $\Ell^\aa_{\mu^\aa}$  defined by 
\begin{equation}
\Ell^\aa_{\mu^\aa}(\rho)\defeq\int_{\USi/\Gamma} R_\aa^\sigma\ \d\mu^\aa\ ,	\label{eq-def:length-aa}
\end{equation}
where $R_\aa^\sigma$ is the (complex valued in the case of complex bundles) 1-form associated to a section $\sigma$ of $\det(F_\aa)$ by $\nabla_u\sigma=R^\sigma(u)\cdot\sigma$. Although $R^\sigma$ depends on the choice of the section $\sigma$, the integrand over $\USi$ does not in the real case.  In the complex case, we see the length functions as taking values in $\mathbb C/2\pi i \mathbb Z$, since $2\pi i\mathbb Z$ is the the group of periods of the form $\frac{\d z}{z}$.

 Recall that in our convention $\det(F_\aa)$ is a contracting bundle and thus the real part of $\Ell_\mu$ is positive. Moreover for a closed geodesic $\gamma$, the associated geodesic current, supported on $\GG\times\{\aa\}$ is also denoted by $\gamma^\aa$.
\begin{eqnarray}
\exp\left(- \Ell^\aa_\gamma(\rho)\right)=\det\left(\left.\operatorname{Hol}(\gamma)\right\vert_{F_\aa} \right)\ ,\label{eq:length-holo}	
\end{eqnarray}
 where $\operatorname{Hol}(\gamma)$ is the holonomy of $\gamma$.
 For a geodesic current $\delta$ supported on a closed geodesic, the length function $\Ell_\delta$ is analytic. This extends to all geodesic currents by density and Morera's Theorem (See \cite{Bridgeman:2015ba} for a related discussion in the real case). The notion extends naturally -- by additivity -- to a general $\Theta$-geodesic current.
 
 We can now extend the length function to any $\Theta$-geodesic current. Let $\mu$ be a $\Theta$-geodesic current on $\GG\times\Theta$, we can then write uniquely
 $$
 \mu=\sum_{\aa\in\Theta}\mu^\aa\ ,
 $$
 where $\mu^\aa$ is supported on $\GG\times\{\aa\}$, then by definition the $\mu$-averaged length function\footnote{In the complex case, since the logarithm, hence the length, is defined up to an additive constant, the Hamiltonian is well defined and the bracket of a length function and any other function makes sense.}
 is
\begin{equation}
 \Ell_\mu(\rho)\defeq \sum_{\aa\in\Theta}\Ell^\aa_{\mu^\aa}(\rho)\ .\label{eq-def:length}
	\end{equation}

\subsection{Averaged correlation function: definition}\label{sec:defi-corell}

When $w$ is a natural map, $\mu$ a $(\rho,w)$-integrable cyclic current, the associated  {\em averaged correlation function of order $n$} $\T_{w(\mu)}$ on the moduli space of $\Theta$-Anosov representations is  defined by 
\begin{equation}
	\T_{w(\mu)}(\rho)\defeq\int_{{\GT}^n/\Gamma}\T_{w(G)}(\rho) \ \d\mu(G)\  , 
\end{equation}   
where  ${G}=(G_1,\ldots,G_p)$ with  and $\T_G$ is the correlation function associated to a $\Theta$-configuration of geodesics defined in paragraph \ref{sec:cor-func}. As we shall see in proposition \ref{pro:reg-corr}, the function $\T_{w(\mu)}$ is analytic . 

\subsection{The main result}\label{sec:main-result} Our main result is a formula for the Poisson bracket of those functions. We use a slightly different convention, writing $\T^k$ for a correlation function of order $k$ and $\T^1_\mu=\Ell_\mu$.

\begin{theorem}[\sc Ghost representation]\label{theo:poiss-bracket} Let $\mu$ be  either a $w$-integrable $\Theta$-cyclic currents at $\rho_0$ or a $\Theta$-geodesic current. Similarly, let $\nu$ be  either a $v$-integrable $\Theta$-cyclic currents at $\rho_0$ or a $\Theta$-geodesic current.  

Then the measure $\mu\otimes\nu$ is $z$-integrable at $\rho_0$, where $z(G,H)=[w(G),v(H)]$ and moreover
 \begin{eqnarray*}
	\{\T^p_{w(\mu)},\T^n_{v(\nu)}\}(\rho)&=&
	\int_{{\GT}^{p+n}/\Gamma} \bI_\rho\left(w(G),v(H)\right)\ \d\mu(G)\d\nu(H)\crcr&=&\int_{{\GT}^{p+n}/\Gamma} \T_{[w(G),v(H)]}(\rho)\ \d\mu(G)\d\nu(H)\ .
\end{eqnarray*}
\end{theorem}
As a corollary, generalizing Theorem \ref{theo:A} given in the introduction, using a simple induction and proposition \ref{pro:bracket-nat} we get
\begin{corollary}[\sc Poisson stability]\label{coro:Poisson-stab}
	The vector space generated by length functions, averaged correlations functions and constants is stable under Poisson bracket. More precisely, let $\mu_1$, \ldots $\mu_p$ cyclic currents of order $n_i$, and $N=n_1+\ldots n_p$ then 
 \begin{eqnarray*}
	\{\T^{n_1}_{\mu_1},\{\T^{n_2}_{\mu_2},\ldots \{\T^{n_{p-1}}_{\mu_{p-1}},\T^{n_p}_{\mu_p}\}\ldots\}\}(\rho)&=&
	\int_{{\GT}^{N}/\Gamma} \T^N_{[G_1,[G_2,[\ldots,[G_{p-1},G_p]\ldots]]]}(\rho)\ \d\mu_1(G_1)\ldots\d\mu_1(G_p)\ .
\end{eqnarray*}
\end{corollary}

In the course of the proof, we will also compute the Hamiltonians of the corresponding functions.

\begin{theorem}[\sc Hamiltonian] Let $\mu$ be a $\Theta$-geodesic current.
	The Hamiltonian of the length function $\Ell_\mu$ is $H^0_\mu$ the trace free part of $H_\mu$, where 
\begin{eqnarray*}
		H_\mu\defeq-\int_{{\GT}}\beta_{\rho(g)}\ \d\mu(g)\ , 
\end{eqnarray*}

Let $w$ be a natural function. Let $\nu$ be a $(\rho,w)$ integrable cyclic current. The Hamiltonian of the correlation function $\T_w(\nu)$ of order $n$, with $n>1$ is 
\begin{eqnarray*}
		\Omega_{w(\nu)}\defeq\int_{{\GT}^n}\Omega_{\rho(w(G))}\ \d\nu(G)\ , 
\end{eqnarray*}
Both $H_\mu$ and $\Omega_{w(\nu)}$ are in $ \bLa^\infty(E)$.
\end{theorem}

\subsection{Preliminary and convention in  symplectic geometry} 
Our convention is that if $f$ is a smooth function and $\bOm$ a symplectic form,  the {\em Hamiltonian vector field} $X_f$ of $f$  and the {\em Poisson bracket} $\{f,g\}$ of $f$ and $g$ are  defined by
\begin{eqnarray} 
		\d f(Y)&=&\bOm(Y,X_f)\ , \ \label{eqdef:ham-vf}\\
		\{f,g\}&=&\d f(X_g)=\bOm(X_g,X_f)=- \d g(X_f)\ .
\label{eqdef:Poisson-fg}\end{eqnarray}

\vskip 0.5 truecm

Observe that if $\Omega$ is a complex valued symplectic form -- which naturally take entries in the complexified vector bundle -- and $f$ a complex valued function then the Hamiltonian vector field is a complexified vector field. The bracket of two complex valued functions is then a complex valued function. 

In the sequel, we will not write different results in the complex case (complex valued symplectic form and functions) and the (usual) real case.

\subsection{Regularity of averaged correlations functions}\label{sec:regu} We prove here
\begin{proposition}\label{pro:reg-corr}
	Let $w$ be a natural function. Let $\mu$ be a $(\rho,w)$-integrable current, then 
	\begin{enumerate}
		\item $\T_{w(\mu)}$ is an analytic function in a neighborhood of $\rho$,
		\item For any tangent vector $v$ at $\rho$, then $\d\T_{w(G)}(v)$ is in $L^1(\mu)$ and 
		$$ \d\T_{w(\mu)}(v)=\int_{{\GT}_\star^n/\Gamma}\d\T_{w(G)}(v)\d\mu(G)\ .$$
	\end{enumerate}
\end{proposition}
As in proposition \ref{pro:ana-TG}, we work in the context of complex uniform hyperbolic bundles, possibly after complexification of the whole situtation.

\begin{proof}
Let us first treat the case when $\mu$ is $\Gamma$-compact. In that case, the functions $\T_G:\rho\mapsto T_{w(G)}$ are all complex analytic by proposition \ref{pro:ana-TG},  uniformly bounded with uniformly bounded derivatives in the support of $\mu$. Thus the result follows from classical results. 

We now treat the non $\Gamma$-compact case. Let now consider an exhaustion of  ${\GT}_\star^n/\Gamma$ by compacts $K_n$ and write $\mu_n=1_{K_n}\mu$. Let then 
$$
\T_n=\int_{K_n}\T_{w(\mu_n)}\d\mu\ .
$$
Then by our integrability hypothesis and Lebesgue dominated convergence Theorem,  $\T_n$ converges uniformly to $\T_{w(\mu)}$. Since all $\T_n$ are complex analytic, by Morera Theorem, $\T_{w(\mu)}$ is complex analytic and $\T_n$ converges $C^\infty$ to $\T_{w(\mu)}$.
It thus follows that 
$$
\d\T_{w(\mu)}(v)=\lim_{n\to\infty}\d\T_n(v)=\lim_{n\to\infty}\int_{K_n}\d\T_{w(G)}(v)\d\mu(G)\ .
$$
We now conclude the proof using Lemma \ref{lem:L1}.
\end{proof}

\subsection{Length functions: their Hamiltonians and brackets}\label{sec:hamil-length}
We first start by computing the bracket and Hamiltonian of length functions. The first step in our proof is to understand the variation of length. 
\begin{proposition}\label{pro:lw}
	The derivatives of a length function with respect to a variation $\dot\nabla$ is given by 
	\begin{eqnarray}
	\d\Ell_\mu(\dot\nabla)=\int_{\Theta\times\USi/\Gamma}\tr\left(\pp\dot\nabla\right)\ \d\mu(x)\ .\label{eq:der-length}
\end{eqnarray}
\end{proposition} 
\begin{proof} By the linearity of the definition, see equation \eqref{eq-def:length},  it is enough to consider a $\Theta$-geodesic current $\mu^\aa$ supported on $\GG\times\{\aa\}$.

Let $E^\aa\defeq\bigwedge^{\dim(F_\aa)} E$, and $\Lambda^\aa$ the natural exterior representation from $\mk{sl}(E)$ to $\mk{sl}(E^\aa)$.
Then by \cite[Lemma 4.1.1]{Labourie:2018fj} and formula \eqref{eq:length-holo} we have 
	\begin{eqnarray}
	\d\Ell_\mu(\dot\nabla)=\int_{\{\aa\}\times\USi/\Gamma}\tr\left(\pp^1_\aa \Lambda^\aa(\dot\nabla)\right)\ \d\mu^\aa(x)\ ,\end{eqnarray}
where $\pp^1_\aa$ is the section of $\End(E_\aa)$ given by the projection on the line  $\det(F_a)$ induced by the projection on $F_\aa$ parallel to $F_\aa^\circ$ -- see section \ref{sec:Theta} for notation.

We now conclude by observing--using just a litle bit of linear algebra--
 that for any element in $\mk{sl}(E)$
$$
\tr\left(\pp^1_\aa \Lambda^\aa(A)\right)=\tr\left(\pp_\aa A\right)\ .$$
Indeed let us choose a basis $(e_1,\ldots, e_p)$ of $F_\aa$ completed by a basis $(f_1,\ldots, f_m)$ of $F_\aa^\circ$ and choose a metric so that this basis  is orthonormal. Then 
\begin{eqnarray*}
& &\Lambda^\aa(A)(e_1\wedge\ldots\wedge e_p)=\sum_{i=1^p}e_1\wedge\ldots e_{i-1}\wedge A(e_i)\wedge e_{i+1}\wedge\ldots e_p\ ,\\
& &\tr\left(\pp^1_\aa \Lambda^\aa(A)\right)=\braket{e_1\wedge\ldots\wedge e_p ,\Lambda^\aa(A)(e_1\wedge\ldots\wedge e_p)}
=\sum_{i=1}^p\braket{e_i\ ,\ A(e_i)}=\tr(\pp_\aa A)\ .\qedhere
\end{eqnarray*}
\end{proof}
Let then
\begin{eqnarray*}
	H_\mu= -\int_{{\GT}}\beta_{\rho(g)}\ \ \d \mu(g) \ . 
\end{eqnarray*}
We proved that $H_\mu$ lies in $ \bLa^\infty(E)$ in lemma \ref{lem:mub-bLa}.
We now prove the following proposition.
\begin{proposition}[\sc Length functions]\label{pro:ham-leng}
	The Hamiltonian vector field of $\Ell_\mu$ is given by $H^0_\mu$, which is the trace free part of $H^\mu$. Then 
	\begin{equation}
	\{\Ell_\nu,\Ell_\mu\}=\bOm(H^0_\mu,H^0_\nu)=\int_{\GT_{\star}^2/\Gamma}\bI_\rho(g,h)\ \d  \nu(g)\otimes \d \mu(h) \ ,\label{eq:Gold-munu}
	\end{equation}
\end{proposition}

Observe that if $\mu$ and $\nu$ are both supported on finitely many geodesics, then the support of $\mu\otimes\nu$ is finite in ${\GT}^2$ and its cardinality is the geometric intersection number of the support of $\mu$, with the support of $\nu$.  This is a generalization of Wolpert cosine formula, see \cite{Wolpert:1981vt}.

 Remark that $\epsilon\mu\otimes \nu$ is supported in ${\GT}^2$ on a set on which $\Gamma$ acts properly.	
\begin{proof}
	Let us first consider the computation of $\bOm(H_\mu,H_\nu)$. Let $\Delta_0$ be a fundamental domain for the action of $\Gamma$ on $\hh$  and  $\Delta_1$ be a fundamental domain for the action of $\Gamma$ on $\GT^2$. Then denoting $\pp^0_g$ the traceless part of $\pp_g$
	\begin{eqnarray*}
\bOm(H^0_\mu,H^0_\nu)&=&\int_{\Delta_0}\tr\left(\left(\int_{{\GT}}\beta^0_h \d\mu(h)\right)\wedge \left(\int_{{\GT}}\beta^0_g \d\nu(g)\right)\right)\crcr
&=&\int_{\Delta_0}\int_{{\GT}\times{\GT}}
\omega_h\wedge\omega_g \tr\left(\pp^0 (g)\pp^0 (h)\right)\ \d\mu(h)\d\nu(g)\crcr
&=&	\int_{\Delta_1}\int_{\hh}\omega_h\wedge\omega_g \tr\left(\pp^0 (g)\pp^0 (h)\right)\ \d\mu(h)\d\nu(g)\crcr
&=&\int_{\Delta_1} \epsilon(h,g)\tr(\pp^0 (g)\pp^0 (h))\ \d\mu(h) \d\nu(g)\ .
\end{eqnarray*}
Let us comment on this series of equalities: the first one is the definition of the symplectic form and that of $H_\mu$ and $H_\nu$, for the second one, we use the pointwise definition of $H_\mu$ and $H_\nu$, for the third one we use proposition \ref{prop:inter}. Observe that the final equality gives formula \eqref{eq:Gold-munu}.

From the third equality we also have
\begin{eqnarray*}
\bOm(H^0_\mu,H^0_\nu)
&=&\int_{\Delta_1}\left(\int_{g}\omega_h \tr\left(\pp^0 (g)\pp^0 (h)\right)\right)\ \d\mu(g)\d\nu(h)\ .
\end{eqnarray*}
Let now consider the fibration $z:\USi\times {\GT}\to{\GT}^2$ and observe that $z^{-1}(\Delta_1)$ is a fundamental domain for the action of $\Gamma$ in $\USi\times {\GT}$.
  Let $\Delta_2$ be a fundamental domain for the action of $\Gamma$ on $\USi$ and observe that $\Delta_2\times{\GT}$ is a fundamental domain for the action on $\Gamma$ on $\USi\times {\GT}$. Then the above equation leads to 
\begin{eqnarray*}
\bOm(H^0_\mu,H^0_\nu)
&=&	\int_{z^{-1}(\Delta_1)}\omega_h \tr\left(\pp^0 (g)\pp^0 (h)\right)\ \d\mu(g)\d\nu(h)\crcr
=\int_{\Delta_2\times{\GT}} \omega_h \tr\left(\pp^0 (g)\pp^0 (h)\right)\ \d\mu(g)\d\nu(h)&=&\int_{\Delta_2}\tr\left(\pp^0 (g)\int_{\GT} \beta^0_{\rho(h)}\d\nu(h) \right)\d\mu(g)\ \crcr
=-\int_{\Delta_2}\tr\left(\pp^0 (g)H^0_\nu\right)\d\mu(g) \ &=&-\d \Ell_\mu(H^0_\nu)=\d \Ell_\nu(H^0_\mu)\ .
\end{eqnarray*}
 As a conclusion, it remains to prove that $H^0_\mu$ is the Hamiltonian of $\Ell_\mu$. For simplicity, let us prove it first when $\mu$ is the current $\delta_g$ supported on the closed geodesic $g$ of $\hh/\Gamma$.  Let $X$ be a closed $\Gamma$ equivariant 1-form with values in $\End_0(E)$, where $\End_0(E)$ is the vector space of traceless endormorphisms, on $\hh$.Then, by proposition~\ref{pro:lw}
$$
\d L_{\delta(g)}(X)=\int_g\tr(\pp^0(g) X^*)\ .
$$
Let $\bar g$ be a lift of $g$ in $\hh$, $\Gamma_0$ the subgroup of $\Gamma$ fixing $\bar g$. Then we can rewrite the previous equation as
\begin{eqnarray*}
	 \d L_{\delta(g)}(X)
	 &=&
\sum_{\gamma\in \Gamma/\Gamma_0}\int_{\gamma\bar g\cap \Delta_0}	\tr \left(
\pp^0(\gamma \bar g) X\right)\ .
\end{eqnarray*}

Let us now use  $\chi_0$ the characteristic function of $\Delta_0$. We now have the string of equalities that we explain after
\begin{eqnarray*}
	 \d L_{\delta(g)}(X)&=& \sum_{\gamma\in \Gamma/\Gamma_0}\int_{\gamma\bar g}	\tr \left(
\pp^0(\gamma \bar g) \chi_0 X \right) \crcr
	&=& \sum_{\gamma\in \Gamma/\Gamma_0}\int_{\hh}\omega_{\gamma \bar g}\wedge	\tr \left(
\pp^0(\gamma \bar g) \chi_0 X\right)\cr  
	&=& \int_{\hh} \tr\left(\left(\sum_{\gamma\in \Gamma/\Gamma_0}\omega_{\gamma \bar g}\pp^0(\gamma \bar g) \right)\wedge	
\chi_0 X\right) \cr 
&=&  \int_{\hh} \tr\left(\left(\int_\GG \omega_h\pp^0(h)\ \d \delta_g(h)\ \right)\wedge	
 \chi_0 X\right)\cr 
&=&\int_{\hh}\tr(H^0_{\delta_g}\wedge \chi_0X)=\int_{\Delta_0}\tr(H^0_{\delta_g}\wedge X)=\bOm(H^0_{\delta_g},X)\ . 
\end{eqnarray*}

The second equality comes from the fact that $\chi_0 X$ is 
with compact support, hence geodesically bounded and we can use proposition \ref{prop:inter}. The third equality uses the fact that we have a sum over only finitely many $\gamma$ in $\Gamma/\Gamma_0$. The fourth uses the definition of $\delta_g$, and the last ones comes from the definitions. 

We have proved that $H^0_\mu$ is the Hamiltonian of $\Ell_\mu$ for all $\mu$, geodesic currents for a closed geodesic. The general result follows by density of this type of currents.

This completes the proof.
\end{proof}
As noted, the above gives a generalization of Wolpert's cosine formula. Explicitly we have for two $\Theta$-geodesic currents $\mu,\nu$ then
\begin{equation}
\{\Ell_\nu,\Ell_\mu\} = \int_{(\GT^2)_\star/\Gamma}\epsilon(g,h)\left(\tr\left(\pp(g)\pp(h)\right) - \frac{\Theta(g)\Theta(h)}{\dim(E)}\right)\ \d \mu(g)\d  \nu(h)\ .\label{eq:cosine}
\end{equation}
\subsection{Bracket of length function and discrete correlation function}\label{sec:bra-length-discrete}
We have

\begin{proposition}
	Let $G$ be a $\Theta$-configuration and $\mu$ a $\Theta$-geodesic current, then 
	\begin{eqnarray*}
		\{\T_G,\Ell_\mu\}=-\int_{{\GT}}\left(\oint_{\rho(G)}\beta_{\rho(g)}\right)\ \d \mu(g)=\int_{{{\GT}}}\bI_\rho(G,g) \d \mu(g)\ .
	\end{eqnarray*}

\end{proposition}
\begin{proof} By equation \ref{eq:dT-integ}, we have
$$
\d \T_G(H_\mu)=-\oint_{\rho(G)}\left(\int_{{\GT}}\beta_{\rho(g)}\d\mu\right)\ .
$$	
Thus by the exchange formula \eqref{eq:exchan-int1}, we have 
$$
\d \T_G (H_\mu)=-\int_{\GT}\left(\oint_{\rho(G)}\beta_{\rho(g)}\right)\ \d \mu(g) \ .
$$	
Thus we conclude by using equation \eqref{defeq:oint}
$$
\{\T_G,\Ell_\mu\}= \d \T_G (H_\mu)=-\int_{\GT}\left(\oint_{\rho(G)}\beta_{\rho(g)}\right)\ \d \mu(g)=\int_{\GT} \bI_\rho(G,g)\ \d\mu(g) \ .
$$	
\end{proof}

\subsection{Bracket of length functions and correlation functions}\label{sec:bra-length-corel}

Our first objective is, given a family of flat connections $\fam{\nabla}$ whose variation at zero is $\dot\nabla$, to compute $
\d \T_\mu(\dot\nabla)
$.

\begin{proposition}[\sc Bracket of length and correlation functions]\label{pro:dTlmu} Assume that the $\Theta$-cyclic current $\mu$ is $(\rho,w)$-integrable. Then
	\begin{eqnarray*}
		\{\T_{w(\mu)},\Ell_\nu\}(\rho)&=&\int_{{{\GT}}^{n+1}/\Gamma}\bI_\rho(w(G),g)\ \d \nu(g)\ \d\mu(G)\ \ .
	\end{eqnarray*}
\end{proposition}

\begin{proof} 
	By Theorem \ref{pro:ham-leng}, the Hamiltonian vector field of $\Ell_\nu$ is given by
	$$
	H^0_\nu=-\int_{{\GT}}\beta^0_{\rho(g)}\ \d\nu(g)\ .
	$$
Let $\Delta$ be a fundamental domain for the action of $\Gamma$ on $\GG^n$, and observe that $\Delta\times{\GT}$ is a fundamental domain for the action of $\Gamma$ on ${\GT}^{n+1}$. It follows since $H_\nu$ is $\rho$-equivariant and proposition \ref{pro:dT-integ}
 that
\begin{eqnarray*}
	\{\T_{w(\mu)},\Ell_\nu\}=\d \T_{w(\mu)}(H^0_\nu)
		&=&\int_{\Delta}\d\T_{w(G)}(H^0_\nu)\ \d\mu(G)\crcr
	=\int_{\Delta}\left(\oint_{\rho(w(G))}H^0_\nu\right)\ \d\mu(G)	&=&-\int_{\Delta}\int_{\GT} \left(\oint_{\rho(w(G))}\beta_{\rho(g)}\right)\ \ \d\nu(g)\d\mu(G)\crcr
	= \int_{\Delta}\int_{\GT} \left(\bI_\rho(w(G),g)\right)\ \ \d\nu(g)\d\mu(G)
	&=&\int_{{\GT}^{n+1}/\Gamma} \bI_\rho(w(G),g) \ \d\mu(G)\d\nu(g)\ .
\end{eqnarray*}
For the second equality we used proposition \ref{pro:reg-corr} and  that integrating a 1-form with values in the center gives a trivial result by proposition \ref{pro:notrace}.
\end{proof}

\subsection{Hamiltonian of correlation functions}\label{sec:hamil-corell}

We  now prove the following result.

\begin{proposition}[\sc Hamiltonian of correlation functions]\label{theo:ham-cor} Let $w$ be a natural map  defined on $\mathcal G^p$.
	Let $\mu$ be a $(\rho,w)$-integrable $\Theta$-current.    Then, \begin{enumerate}
\item 	for every $y$ in $\uh$  the   map on $\mathcal G^p$ with values in $(\T_y\uh)^*\otimes\End(E_y)$, defined by $$G\mapsto\Omega_{\rho(w(G))}(y)\ ,$$ is integrable with respect to $\mu$
 \item Moreover
	\begin{equation}
\Omega_{w(\mu)}(\rho)\defeq \int_{{\GT}^p}\Omega_{\rho(w(G))}\d\mu(G)\ \end{equation}
seen as vector field on the character variety, is the Hamiltonian of the correlation function $\T_{w(\mu)}$. 
 \end{enumerate}
   
\end{proposition}
We first prove proposition \ref{theo:ham-cor} under the additional hypothesis that $\mu$ is a $\Gamma$-compact current, then move to the general case by approximation.

\begin{proof}[Proof for a  $\Gamma$-compact current]

Assume $\mu$ is a $\Gamma$-compact current. By the density of derivatives of length functions, it is enough to prove that for any geodesic current $\nu$ associated to a length function $\Ell_\nu$ whose Hamiltonian is $H_\nu$ we have 
$$
\{\Ell_\nu,\T_{w(\mu)}\}=\bOm(\Omega_{w(\mu)},H_\nu)=\d\Ell_\nu(\Omega_{w(\mu)})\ .
$$	
Then using a fundamental domain $\Delta_0$ for the action of $\Gamma$ on $\USi$, and $\Delta_1$ a fundamental domain for the action of $\Gamma$ on ${\GT}^n$, and finally denoting $\nu_0$ the flow invariant measure in $\USi$ associated to the current $\nu$
\begin{eqnarray*}
\d\Ell_\nu(\Omega_{w(\mu)})
	=\int_{\Delta_0} \tr(\pp\Omega_{w(\mu)})\d\nu_0(g)	&=&\int_{\Delta_0}\left(\int_{{\GT}^n} \tr(\pp (g)\Omega_{\rho(w(G))})\ \d\mu(G)\right)\d\nu_0(g)\crcr
	=\int_{{\GT}^n}\!\left(\int_{\Delta_0} \!\!\tr(\pp (g)\Omega_{\rho(w(G))})\ \d\nu_0(g)\right)\d\mu(G)	&=&\int_{\Delta_1}\left(\int_{\USi} \tr(\pp (g)\Omega_{\rho(w(G))})\ \d\nu_0(g)\right)\d\mu(G)\crcr
	=\int_{\Delta_1}\!\int_{\GT}\!\int_{g} \tr(\pp (g)\Omega_{\rho(w(G))}))\ \d\nu(g)\d\mu(G)	&=&\int_{{\GT}^n/\Gamma}\!\!\left(\int_{\GT}\!\int_{\hh} \!\!\tr(\omega_g\pp (g)\wedge\Omega_{\rho(w(G))})\ \right)\d\nu(g) \d\mu(G)\crcr
	=-\int_{\left(\GT^n/\Gamma\right)\times \GT}\bI_\rho(w(G),g)\ \d\mu(G)\d\nu(g)&=&\{\Ell_\nu,\T_\mu\}\  .
	\end{eqnarray*}
The first equality uses equation \eqref{eq:der-length}, the second uses the definition of $\Omega_\mu$, the third one comes from Fubini's theorem, the fourth one from lemma
\ref{lem:interDELTA}, the fifth one from the fibration from $\USi$ to $\GG$, the sixth one from formula \eqref{eq:duality1}, the seventh one definition \eqref{defeq:oint}.
\end{proof}

\begin{proof}[Proof for an integrable $\mu$] Let us now prove the general case when $\mu$ is a $\rho$-integrable current. Let us consider an exhaustion $\seq{K}$ of ${\GT}^p/\Gamma$ by compact sets. Assume that the interior of $K_{m+1}$ contains $K_m$. Let $\mathcal K$ be a fundamental domain of the action $\Gamma$ on  ${\GT}_\star^p$.
Let 
$$
\T_m(\rho)\defeq\int_{K_m}\T_{w(G)}(\rho)\ \d\mu(G)\ .
$$
The functions $\T_m$ are analytic and converges $C^0$ on every compact set to $\T_\mu$ by the integrability of $\mu$. Thus, by Morera's Theorem,  $\T_\mu$ is analytic and $\seq{\T}$ converges $C^\infty$ on every compact . Let us call $X$ the Hamiltonian vector field of $\T_\mu$ and $X_m$ the Hamiltonian vector field of $\T_m$. It follows that $\seq{X}$ converges to $X$. 

We let $C_m$ be the preimage of $K_m$ in $\GT^p$. We have just proven in the previous paragraph that the Hamiltionian of $\T_m$ is
$$
X_m=\int_{C_m}\Omega_{\rho(H)}\ \d \mu\ .$$
From corollary \ref{coro:bd-Om}, for every $y$ and $H$, the function
$\gamma\mapsto \Vert\Omega_{\rho(\gamma w(H))}(y)\Vert$, is in $\ell^1(\Gamma)$. 
It follows that 
$$
X_m(y)=\int_{C_m\cap \mathcal K}\left(\sum_{\gamma\in\Gamma}\Omega_{\rho(\gamma H)}(y)\ \d\mu(H)\right)\ .
$$
Since $\{X_m(y)\}_{m\in\mathbb N}$ converges for any exhaustion of $\mathcal K$ to $X(y)$. It follows by lemma \ref{lem:L1} that 
$$
H\mapsto \sum_{\gamma\in\Gamma}\Omega_{\rho(\gamma w(H))}(y)\ \d\mu(H)\ ,
$$ is in $L^1(\mathcal K,\mu)$ and that 
$$X(y)= \int_{\mathcal K}\sum_{\gamma\in\Gamma}\Omega_{\rho(\gamma w(H))}(y)\ \d\mu(H)=\int_{{\GT}^p}\Omega_{\rho(w(H))}(y)\ \d\mu(H)\ ,
$$
where we applied  Fubini again in the last equality. This is what we wanted to prove.
\end{proof}

\subsection{Bracket of correlation functions}\label{sec:bra-corell} We have
\begin{proposition}[\sc Bracket of correlation functions]\label{pro:brack-cor}
Let $\mu$  and $\nu$ be respectively $(\rho,v)$ and $(\rho,w)$ integrable $\Theta$-currents with $\mu$ of rank $m$ and $\nu$ of rank $n$. Let $p=m+n$, then
	\begin{equation}
		\{\T_{w(\nu)},\T_{v(\mu)}\}=\int_{{{\GT}}^{p}/\Gamma}\bI_\rho(w(H),v(G))\ \d\nu\otimes\d\mu(H,G)\ .
	\end{equation}	
\end{proposition}
\begin{proof} We have 
\begin{eqnarray*}
  	\{\T_{w(\nu)},\T_{v(\mu)}\}=\d\T_{w(\nu)}(\Omega_{v(\mu)})
  	&=&\int_{{\GT}^n/\Gamma}\d\T_{w(H)}(\Omega_{v(\mu)})\ \d\nu(H)\crcr
  	=\int_{{\GT}^n/\Gamma}\left(\oint_{\rho(w(H))}\Omega_{v(\mu)}\right)\ \d\nu(H)  	&=&\int_{{\GT}^n/\Gamma}\left(\oint_{\rho(w(H))}\int_{{\GT}^m}\Omega_{\rho(v(G))}\d\mu(G)\right)\ \d\nu(H)\crcr
  	=\int_{{\GT}^n/\Gamma}\left(\int_{{\GT}^m}\oint_{\rho(w(H))}\Omega_{\rho(v(G))}\d\mu(G)\right)\ \d\nu(H)
  	&=&\int_{{\GT}^p/\Gamma}\bI_\rho(w(H),v(G))\ \d\nu(H) \d\mu(G) \ .
\end{eqnarray*}
The crucial point in this series of equalities is the exchange formula for the fifth equality which comes from Theorem \ref{theo:oint-exchang2}. \end{proof}

With the above, we have completed the proof of the ghost representation Theorem \ref{theo:poiss-bracket}.

\part{Applications} 

We prove in this last part, two applications of the main result. First the convexity of length functions (in the projective case), then we explain that geodesic laminations define commuting subalgebras  on the character varieties, a feature which is well known for Hitchin components by Bonahon--Dreyer \cite{Bonahon:2014woa}, Zhe Sun  and Tengren Zhang \cite{Sun:2017}, Sun--Wienhard--Zhang \cite{Sun:2020vm}, and and goes back in Teichmüller theory to Bonahon \cite{Bonahon:1996}, but that we extend to any  deformation space of Anosov representations.

\section{Convexity of length functions}\label{sec:convex}

Our goal is a generalization of Kerckhoff's theorem \cite{Kerckhoff:1983th} of the convexity of length functions, as well as  a generalization of Wolpert's Sine Formula for the second derivatives along twist orbits \cite{Wolpert:1983td}. Both results will follow from computations in the ghost algebra combined with the Ghost Representation Theorem \ref{theo:poiss-bracket}. 

Our first theorem is a generalization of Wolpert's Sine Formula.

\begin{theorem}\label{theo:sine-formula}
Let $\mu$ be an oriented  geodesic current supported on non-intersecting geodesics. Then for any geodesic current $\nu$ and any projective Anosov representation $\rho$, we have 
\begin{eqnarray*}
\{\Ell_\mu,\{\Ell_\mu,\Ell_\nu\}\}(\rho)=2\int_{\GG^{3,+}/\Gamma}\!\!\! \epsilon(g_0,h)\epsilon(g_1,h)\left(\T_{\lceil g_1,h,g_0\rceil -\lceil g_1,h\rceil\lceil g_0,h\rceil }\right)(\rho)\ \  \d \mu^2 (g_1,g_0) \d \nu(h)  \ .
\end{eqnarray*}
where $\GG^{3,+}$  is the set of $(g_1,h,g_0)$ so that $h$ intersects both $g_1$ and $g_0$, as well as $h$ intersecting $g_1$ before $g_0$. 
\end{theorem}
Observe that in this formula, the representation is just assumed to be Anosov.

Let us first say, following Martone--Zhang \cite{Martone:2019uf} that a representation has {\em a positive cross ratio} if for all intersecting geodesics $g$ and $h$
$$
0<\T_{\lceil g,h \rceil}(\rho)<1\ .
$$

We now restate the Convexity Theorem \ref{theo:B}.

\begin{theorem}[\sc Convexity)]\label{theo:fin-convex}
	Let $\mu$ be an oriented  geodesic current supported on non-intersecting geodesics. Then for any geodesic current $\nu$ and any projective Anosov representation $\rho$ with a positive cross ratio, we have 
\begin{equation}
	\{\Ell_\mu,\{\Ell_\mu,\Ell_\nu\}\}(\rho)\geq 0\ .
\end{equation}
Furthermore the inequality is strict if and only if $i(\mu,\nu) \neq 0$.
\end{theorem}

We start by computing double brackets in the Ghost Algebra.

\subsection{Double derivatives of length functions in the swapping algebra}
In order to prove our convexity result, we will need to calculate  double brackets. By Theorem  \ref{theo:ext-ghost}, as the map $A \rightarrow \T_A$ on the ghost algebra factors through the extended swapping bracket $\mathcal B_0$, it suffices to do our calculations in $\mathcal B_0$. For simplicity, we will further denote the elements $\ell_g$ in $\mathcal B_0$ by $g$.
 
\begin{lemma}\label{lem:tripBrak}
Let $h$ be an oriented  geodesic and  $g_0, g_1$ be two geodesics so that $\epsilon(g_0,g_1)=0$. Let $\epsilon_i=\epsilon(g_i,h)$. \begin{enumerate}
	\item Assume first that $\epsilon_0\epsilon_1=0$, then 
$
[g_1,[g_0,h]]=0$.
\item Assume otherwise that $h$ intersect $g_1$ before $g_0$ or that $g_1=g_0$. Then \begin{eqnarray*}
[g_1,[g_0,h]]&=& \epsilon_1\epsilon_0  \left(\lceil g_1,h,g_0 \rceil-  \lceil g_1,h\rceil \ \lceil g_0,h \rceil \right)\\
 &=&\epsilon_1\epsilon_0\   \lceil g_1,h\rceil \ \lceil g_0,h \rceil\ \left(\lceil \gamma_0 ,\gamma_1 \rceil-1 \right)\ ,
	\end{eqnarray*}
	where  $\gamma_0 \defeq (g_0^+,h^-) $ and $\gamma_1 \defeq (h^+, g_1^-)$. \footnote{Observe that $\gamma_0$ and $\gamma_1$ are not phantom geodesics by hypothesis.}
\end{enumerate}
\end{lemma}

\begin{proof}
From equation \eqref{eq:ext-swap} of paragraph \ref{sec:bracket-ghost},
\begin{equation*}
		[g_0,h]=\epsilon(h,g_0)\lceil g_0,h \rceil\  + \epsilon(g_0,h)\casper\ . 
	\end{equation*}
	It follows that if $\epsilon(g_0,h)=0$, then  \begin{eqnarray}
[g_1,[g_0,h]]=0\  .\label{eq:llsym}
\end{eqnarray}
	The same holds whenever $\epsilon(g_1,h)=0$ by the symmetry given  the Jacobi identity for the extended swapping bracket, which gives, since $[g_0,g_1]=0$,  \begin{eqnarray*}
[g_1,[g_0,h]]=[g_0,[g_1,h]]\ .
\end{eqnarray*}
	
	Assume now that $\epsilon_0\epsilon_1\not=0$.
		Then let $(g_0,\zeta_0,h,\eta_0)$ be the associated ghost polygon to $\lceil g_0,h\rceil $ with ghost edges $\zeta_0 = (g_0^+,h^-)$ and $\eta_0 = (h^+,g_0^-)$.
	Thus from the hypothesis $\epsilon(g_0,g_1)=0$, and using the notation $\epsilon_i=\epsilon(g_i,h)$ we get from equation \eqref{def:ghost-bracketgH}
	\begin{eqnarray*}
	[g_1,[g_0,h]]
			=-\epsilon_0\lceil g_0,h \rceil \left(\epsilon_1 \lceil g_1,h\rceil- \epsilon(g_1,\zeta_0)\lceil g_1,\zeta_0\rceil -\epsilon(g_1,\eta_0)\lceil g_1, \eta_0 \rceil \right)\ . 
	\end{eqnarray*}
	Since $h$ intersects $g_1$ before $g_0$, we have $\epsilon(g_1,\eta_0)=0$ and  $\epsilon(g_1,\zeta_0)=\epsilon(g_1,h)$. 	Thus 
	\begin{eqnarray}
		[g_1,[g_0,h]]
			&=&\epsilon_1\epsilon_0 \left( \lceil g_1, \zeta_0 \rceil \lceil g_0,h \rceil- \lceil g_1,h\rceil \lceil g_0,h \rceil \right)\ \label{eq:gho-brakll}. 
	\end{eqnarray}
As $\zeta_0 = (g_0^+,h^-)$,  formulating the computations the swapping algebra, we get
\begin{eqnarray}
  \lceil g_1, \zeta_0 \rceil \lceil g_0,h \rceil &=& \frac{(g_1^+,h^-)(g_0^+,g_1^-)(g_0^+,h^-)(h^+,g_0^-)}{(g_1^+,g_1^-)(g_0^+,h^-)(g_0^+,g_0^-)(h^+,h^-)} \cr
  &=& \frac{(g_1^+,h^-)(h^+,g_0^-)(g_0^+,g_1^-)}{(g_1^+,g_1^-)(h^+,h^-)(g_0^+,g_0^-)} = \lceil g_1, h, g_0 \rceil\ .
  \label{hrel}\end{eqnarray}
Similarly
\begin{eqnarray}\frac{\lceil g_1, h, g_0 \rceil}{ \lceil g_1, h \rceil \lceil g_0, h\rceil} &=&   \frac{(g_0^+,g_1^-)(h^+,h^-)}{(g_0^+,h^-)(h^+,g_1^-)} = \lceil (g_0^+,h^-),(h^+,g_1^-)\rceil\ .\label{hrel2}\end{eqnarray} The result then follows from equations \eqref{hrel} and the fact that $\gamma_0=(g_0^+,h^-)$ and $\gamma_1=(h^+,g_1^-)$. \end{proof}
\subsection{Proof of the Sine Formula Theorem \ref{theo:sine-formula}}

We are now in position to prove the sine formula.

\begin{proof}
		By the Representation Theorem and its corollary \ref{coro:Poisson-stab} 
	\begin{equation*}
		\{\Ell_\mu,\{\Ell_\mu,\Ell_\nu\}\}(\rho)=\int_{\GG^3/\Gamma}\T_{[g_1,[g_0, h]]}(\rho)\ \d\mu(g_0)\d\mu(g_1)\d\nu(h)\ .	\end{equation*}
	Since the support of $\mu$ consists of non intersecting geodesics, we have for $g_0$ and $g_1$ in the support of $\mu$, $[g_0,g_1]=0$. Hence the
 Jacobi identity for the swapping bracket gives 
$$ [g_0,[g_1, h]]=[g_1,[g_0, h]]\ ,$$
for $g_0$ and $g_1$ in the support of $\mu$. It follows that 
$$\int_{\GG^3/\Gamma}\T_{[g_1,[g_0, h]]}(\rho)\ \d\mu(g_0)\d\mu(g_1)\d\nu(h)\ =2\int_{\GG^{3,+}/\Gamma}\T_{[g_1,[g_0, h]]}(\rho)\ \d\mu(g_0)\d\mu(g_1)\d\nu(h)\ .$$
	Then we use lemma \ref{lem:tripBrak} to conclude.
		\end{proof}
 \subsection{Positivity}

Recall that a projective representation $\rho$ has a positive cross ratio (according to Martone--Zhang\cite{Martone:2019uf})  if for all $g,h$ intersecting geodesics $0 < \T_{\lceil g,h\rceil}(\rho) < 1$. 
Our goal is the following.

\begin{proposition}[\sc Sign proposition] \label{cor:signlemma}
	Assume $\rho$ is a projective Anosov representation with a positive cross ratio.  Let $g_1$, $g_0$ be such that $\epsilon(g_0,g_1) = 0$. Then we have the inequality
		$$
	\T_{[g_1,[g_0, h]]}(\rho)\geq0\ .
	$$
	Furthermore the inequality is strict if and only if $h$ intersects both $g_0, g_1$ in their interiors.
\end{proposition}
We first  give an equivalent definition of positivity.
\begin{lemma} \label{pro:inter-pos}
	A projective Anosov representation  $\rho$ has a positive cross ratio if and only if for all $(X,Y,y,x)$ cyclically oriented 
	$$
		\T_{\lceil (X,x),(Y,y)\rceil}(\rho)>1\ .
	$$
\end{lemma}
\begin{proof} Let $X,x,Y,y$ be 4 points. We observe that $(X,Y,y,x)$ is cyclically  oriented if and only if  geodesics $(X,y),(Y,x)$ intersect. The result then follows from
$$
\lceil (X,x),(Y,y)\rceil =\frac{(X,y)\ (Y,x)}{(X,x)\ (Y,y)}=\left(\frac{(X,x)\ (Y,y)}{(X,y)\ (Y,x)}\right)^{-1}=\lceil (X,y),(Y,x)\rceil^{-1}\ .\qedhere
$$ 
\end{proof}

We now prove  proposition \ref{cor:signlemma}

\begin{proof}[Proof of proposition \ref{cor:signlemma}] The Jacobi identity  for the swapping bracket \ref{pro:ext-swapp}
 gives that $[g_0,[g_1,h]]=[g_1,[g_0,h]]$ since $[g_0,g_1]=0$. Thus in the statement of the proposition, we can always assume that not only  $h$ intersects both $g_1$ and $g_0$, but furthermore that  $h$ intersects $g_1$ before $g_0$.  By lemma \ref{lem:tripBrak}, to prove the proposition it is enough to prove that 
	\begin{equation}
\epsilon_1\epsilon_0	\T_{\lceil g_1 ,h , g_0\rceil- \lceil g_1,h\rceil \lceil g_0,h \rceil}(\rho)\geq 0 \ ,
	\end{equation}
	and furthermore the inequality is strict if and only if $h$ intersects both $g_0, g_1$ in their interiors  (i.e. if and only if $|\epsilon_0\epsilon_1| = 1$).
By Lemma \ref{lem:tripBrak} we have, since  $g_1$ meets $h$ before $g_0$. 
	\begin{eqnarray*}\epsilon_0\epsilon_1\left(\lceil g_1,h,g_0\rceil - \lceil g_1,h\rceil\lceil g_0,h\rceil\right) &=&\epsilon_1\epsilon_0\   \lceil g_1,h\rceil \ \lceil g_0,h \rceil\ \left(\lceil \gamma_0 ,\gamma_1 \rceil-1 \right)\ .
	\end{eqnarray*}
	where  $\gamma_0 \defeq (g_0^+,h^-) $ and $\gamma_1 \defeq (h^+, g_1^-)$. We will also freely use that if $x^+=y^-$ or $x^-=y^+$, then $\T_{\lceil x,y\rceil}=0$, while if $x^+=y^+$ or $x^-=y^-$ then $\T_{\lceil x,y\rceil}=1$.
\vskip 0.2 truecm 
\noindent{\sc First case: $\epsilon_0\epsilon_1 =0$.}
In that case, we have equality.

\vskip 0.2 truecm 
\noindent{\sc Second case: $0<\vert\epsilon_0\epsilon_1\vert <1$.} In that situation one of the end point of $h$ is an end point of $g_0$ or $g_1$.  
\begin{enumerate}
	\item Firstly, the cases  $g_0^\pm  = h^-$ or $g_1^\pm = h^+$ are impossible since $h$ meets $g_1$ before $g_0$. 
	\item Secondly if  $g_1^+= h^-$ or $g_0^- = h^+$, then 
	 $
			\T_{\lceil g_1,h\rceil}\T_{\lceil g_0,h\rceil}(\rho) = 0$.

	\item Finally, if  $g_1^-= h^-$ or $g_0^+ = h^+$, then either $\gamma_0^+=\gamma_1^+$ or $\gamma_0^-=\gamma_1^-$. In both cases, $\T_{\lceil\gamma_0,\gamma_1\rceil}(\rho)=1$ and hence 
 it follows that $\T_{ \lceil g_1 ,h , g_0\rceil- \lceil g_1,h\rceil \lceil g_0,h \rceil}(\rho)=0$.
\end{enumerate}

\vskip 0.2 truecm 
\noindent{\sc Final case: $\vert\epsilon_0\epsilon_1\vert=1$.}

As both $g_0$ and $g_1$ intersect $h$ and $\rho$ has a positive cross ratio, then by lemma \ref{pro:inter-pos},
\begin{eqnarray}
	\T_{\lceil g_1,h\rceil \ \lceil g_0,h \rceil}(\rho)=\T_{\lceil g_1,h\rceil}(\rho) \T_{ \lceil g_0,h \rceil}(\rho)>0\ .\label{eq:TTpos}
\end{eqnarray}
We can then split into two  cases as in figure \eqref{fig:gammacurves}:
\begin{enumerate}
	\item \label{conf1} If $\epsilon_0\epsilon_1>0$, then $\gamma_0$ and $\gamma_1$ do not intersect, and $(h^-,g_0^+,h^+,g_1^-)$ is a cyclically oriented quadruple. Hence, by definition
	$\T_{\lceil\gamma_0,\gamma_1\rceil}(\rho)>1
	$. See figure \eqref{fig:gammacurves1})
	\item \label{conf2}If now $\epsilon_0\epsilon_1<0$, then $\gamma_0$ and $\gamma_1$  intersect, and by lemma  \ref{pro:inter-pos}
	$\T_{\lceil\gamma_0,\gamma_1\rceil}(\rho)<1 
	$.(see figure \eqref{fig:gammacurves0})
\end{enumerate}
\begin{figure}[h]
 \begin{subfigure}[h]{0.3\textwidth} \begin{center}
  \includegraphics[width=0.8\textwidth]{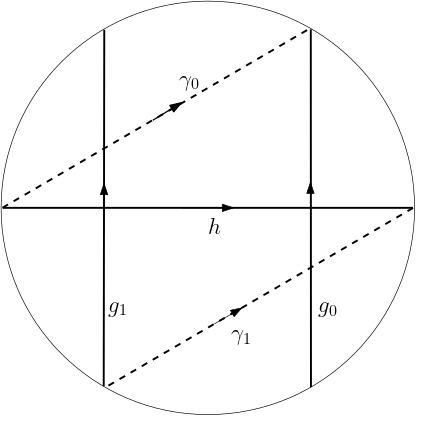}
  \end{center}
    \caption{$\epsilon_0\epsilon_1>0$}\label{fig:gammacurves1}
 \end{subfigure} 
 \quad
\begin{subfigure}[h]{0.3\textwidth}   \begin{center}
    \includegraphics[width=0.8\textwidth]{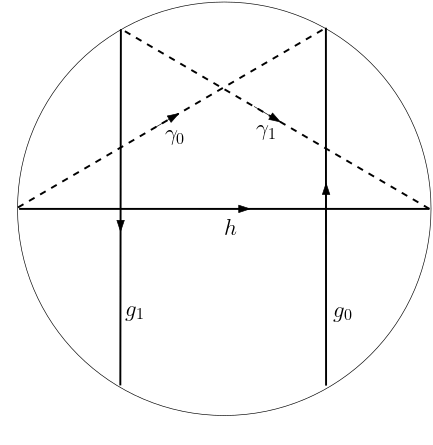}
    \caption{$\epsilon_0\epsilon_1<0$}\label{fig:gammacurves0}
  \end{center}
  \end{subfigure}
   \caption{Curves $\gamma_0$ and $\gamma_1$}
   \label{fig:gammacurves}
\end{figure}
Combining both cases, we get that 
\begin{eqnarray}
\epsilon_0\epsilon_1	\left(\T_{\lceil\gamma_0,\gamma_1\rceil}(\rho)-1\right)> 0\ . \label{eq:TTTpos}
\end{eqnarray}
The result follows from equations \eqref{eq:TTTpos} and \eqref{eq:TTpos}.\end{proof}

\subsection{Proof of the Convexity Theorem \ref{theo:fin-convex}}
\begin{proof}
	By the representation theorem and its corollary \ref{coro:Poisson-stab} 
	\begin{equation*}
		\{\Ell_\mu,\{\Ell_\mu,\Ell_\nu\}\}(\rho)=\int_{\GG^3/\Gamma}\T_{[g_1,[g_0, h]]}(\rho)\ \d\mu(g_0)\d\mu(g_1)\d\nu(h)\ .	\end{equation*}
	Since by the Sign Proposition \ref{cor:signlemma}, the integrand is non-negative, the integral is non-negative. 
	
	Let us  finally treat the equality case.	
	If $i(\mu,\nu)= 0$ then  for all $g$ in the support of $\mu$ and $h$ in the support of $\nu$, $|\epsilon(g,h)| < 1$. Thus by the equality case of proposition \ref{cor:signlemma} for $g_0,g_1$ in the support of $\mu$ and $h$ in the support of $\nu$   then
	$$\T_{[g_1,[g_0, h]]}(\rho) = 0\ .$$
	Thus the integral is zero for $i(\mu,\nu)= 0$.
	
	If $i(\mu,\nu) \neq 0$ then there exists $g_0, h$ in the supports of $\mu,\nu$ respectively such that $|\epsilon(g_0,h)| = 1$.  If $h$ is descends to a closed geodesic then it is invariant under an element $\gamma$ of $\Gamma$ then we let $g_1 = \gamma g_0$. Then the triple $(g_1,g_0,h)$ is in the support of $\mu\otimes\mu\otimes \nu$. Thus $\T_{[g_1,[g_0, h]]}(\rho) > 0$ and the integral is positive. If $h$ does not descend to a closed geodesic, then as any geodesic current is a limit of a discrete geodesic currents, it follows that $h$ intersects   $g_1 = \gamma g_0$ for some $\gamma$ in $\Gamma$. Again the triple $(g_1,g_0,h)$ are in the support of $\mu\otimes\mu\otimes \nu$ with  $\T_{[g_1,[g_0, h]]}(\rho) > 0$. Thus the integral is positive.
	This completes the proof of Theorem \ref{theo:fin-convex}.  \end{proof}

\section{Commuting subalgebras}

Our second application allows us to construct commuting subalgebras in the Poisson algebra of correlation functions for projective Anosov representations. Let $\mathcal L$ be a geodesic lamination.  Associated to this lamination we get several functions that we called {\em associated to the lamination}
\begin{enumerate}
	\item The length functions associated to geodesic currents supported on the lamination, 
	\item functions associated to any complementary region of the lamination. 
\end{enumerate}
Let  $F_{\mathcal L}$ be the vector space generated by these functions.
Our result is then,

\begin{theorem}\label{theo:commuting}
	Let $\mathcal L$ be a geodesic lamination, then  the vector space $F_{\mathcal L}$ consists of pairwise Poisson commuting functions.
\end{theorem}

An interesting example is the case of the maximal geodesic lamination coming from a decomposition into pair of pants. An easy check gives that  there are $4g-4$ triangle functions and, as there there are 3g-3 geodesics, there are $6g-6$ oriented length functions. Thus we have $10g-10$ commuting functions. However in the $\pslt$ case the dimension of the space of is $16g-16$ and it follows that there are relations between these functions. It is interesting to notice that these relations may not be algebraic ones:  in that specific case, we wonder whether some of these relations may be  given by the higher identities \cite{McShane-Lab} generalizing Mirzakhani--McShane identities.

  As we said before the fact that these subalgebras are commuting is related, in the special  case of Hitchin representations, to the coordinate systems discussed in \cite{Bonahon:2014woa,Sun:2017,Sun:2020vm}.

 \subsection{Triangle functions and double brackets}

Let $\delta_0=(a_1,a_2,a_3)$ be an oriented  ideal triangle, we associate to such a triangle the configuration
\begin{equation}
	t_0\defeq \lceil a_1,a_3,a_2\rceil\ .
\end{equation}
The reader should notice the change of order. 

One can make the following observation. First $t\ \bar t =1$. Thus for a self-dual representation $\rho$, we have  $\T_t(\rho)^2=1$ and in particular $\T_t$ is constant along self dual representations.

\begin{lemma}[\sc Brackets of triangle functions]
Let $t_0 = \lceil a_1,a_3,a_2\rceil $ be a triangle, then 
\begin{eqnarray*}
	[t_0, g]
		&=&	\sum_{j\in\{1,2,3\}}
		\epsilon(a_j,g)\  t_0\ \left(\lceil g,a_j\rceil+\lceil g,\bar a_j\rceil\right)\ . 
\end{eqnarray*}
Let $t_1 = \lceil b_1,b_3,b_2\rceil $. Then
\begin{eqnarray*}
[t_1,t_0] = t_1\cdot t_0\sum_{i,j\in\{1,2,3\}} \epsilon(a_i,b_j)(\lceil a_i,b_j\rceil + \lceil a_i, \overline b_j\rceil+\lceil \overline a_i,b_j\rceil + \lceil \overline a_i, \overline b_j\rceil ) = t_0\sum_{i\in\{1,2,3\}} [t_1,\overline a_i-a_i]\ .
\end{eqnarray*}
\end{lemma}
\begin{proof} 
The ghost polygon associated to $t_0$ is $(a_1,\overline a_2,a_3,\overline a_1,a_2,\overline a_3)$. 
Thus 
\begin{eqnarray*}
	[t_0, g]
		=	t_0 \sum_{j\in\{1,2,3\}}
		\epsilon(a_j,g)\lceil g,a_j\rceil-\epsilon(\overline a_j,g)\lceil g,\bar a_j\rceil
		= t_0 \sum_{j\in\{1,2,3\}}
		\epsilon(a_j,g)\left(\lceil g,a_j\rceil+\lceil g,\bar a_j\rceil\right)\  . 
\end{eqnarray*}

For $t_0,t_1$ we have
\begin{eqnarray*}
[t_1,t_0] &=& t_1\cdot t_0\sum_{i,j\in\{1,2,3\}} \epsilon(a_i,b_j)\lceil a_i,b_j\rceil - \epsilon(a_i,\overline b_j)\lceil a_i, \overline b_j\rceil-\epsilon(\overline a_i,b_j)\lceil \overline a_i,b_j\rceil + \epsilon(\overline a_i,\overline b_j)\lceil \overline a_i, \overline b_j\rceil\nonumber\\
&=& t_0\cdot t_1\sum_{i,j\in\{1,2,3\}} \epsilon(a_i,b_j)(\lceil a_i,b_j\rceil +\lceil a_i, \overline b_j\rceil+\lceil \overline a_i,b_j\rceil + \lceil \overline a_i, \overline b_j\rceil)\nonumber\\
&=& t_0\sum_{i\in\{1,2,3\}} [t_1, \overline a_i- a_i]
\end{eqnarray*}

\end{proof}

\subsection{Proof of Commuting Subalgebra Theorem \ref{theo:commuting}}
We now consider $\mathcal L$ a maximal lamination and $F_{\mathcal L}$ the vector space of functions associated to $\mathcal L$ generated by triangle functions and length functions supported on $\mathcal L$. The proof of Theorem \ref{theo:C} that  $F_{\mathcal L}$ is a commuting subalgebra now follows from the below proposition.

\begin{proposition}
Let $g$ be disjoint from the interior of ideal triangle $\delta$. Then $g$ and the triangle function $t$ commute.
Similarly let $\delta_0, \delta_1$ be ideal triangles with disjoint interiors. Then the associated triangle functions $t_0,t_1$  commute.
\end{proposition}

\begin{proof}
We first make an observation. If $\epsilon(g, h) = \pm 1/2$ then
$$\lceil g,h \rceil + \lceil g, \overline h\rceil = 1\ .$$
To see this,  assume $g^+=h^-$. Then
$\lceil g,h\rceil = 0$ and $\lceil g, \overline h\rceil$ has ghost polygon $(g,\overline h,\overline h,  g)$ giving
$$ \lceil g, \overline h\rceil = \frac{\overline h \cdot g}{g\cdot \overline h} =1\ .$$
By symmetry, this holds for all $g,h$ with $\epsilon(g,h) = \pm 1/2$.

Let $g$ be disjoint from the interior of ideal triangle $\delta = (a_1,a_2,a_3)$. Then from above
$$[g, t] = t\sum_{i\in\{1,2,3\}} \epsilon(g,a_i)(\lceil g,a_i\rceil + \lceil g, \overline a_i\rceil) = t\sum_{i\in\{1,2,3\}} \epsilon(g,a_i)\ .$$

If $\epsilon(g,a_i) = 0$ for all $i$ then trivially $[ g, t] = 0$. Thus we  can assume $\epsilon(g, a_1) = 0$ and $\epsilon(g,a_2),\epsilon(g,a_3) \neq 0$.  If $g = a_1$ then as $\epsilon(a_1, a_2) = -\epsilon(a_1, a_3)$  then $[g, t] =0$. Similarly for $g = \overline{a_1}$. 

Otherwise $g, a_2, a_3$ share a common endpoint and $a_2,a_3$ have opposite orientation at the common endpoint. Therefore as $g$ is not between $a_2$ and $a_3$ in the cyclic ordering about their common endpoint, then $\epsilon(g,a_2)= -\epsilon(g,a_3)$ giving $[g, t] =0$.

Let $t_0, t_1$ be the triangle function associated to ideal polygons $\delta_0, \delta_1$ with  $t_0 = [a_1, a_3,a_2]$.  Then from above
$$[t_1,t_0] = t_0\sum_i [t_1,\overline a_i-a_i]\ .$$
Thus if  $t_0,t_1$ have  ideal triangles with disjoint interiors then by the above, $[a_i,t_1] = [\overline a_i,t_1]=0$ giving $[t_0,t_1]=0$.
\end{proof}
\part{Addendum}
\begin{appendix}
\section{Fundamental domain and $L^1$-functions}
If $\Gamma$ is a countable group acting on $X$ preserving a measure $\mu$, a {\em $\mu$-fundamental domain} for this action is a measurable set $\Delta$ so that
$
\sum_{\gamma\in\Gamma}{\boldsymbol 1}_{\gamma(\Delta)}=1$, $\mu$-almost everywhere.
A function $F$ on $X$ is $\Gamma$-invariant if for every $\gamma$ in $\Gamma$,
$
F=F{\circ}\gamma$,  $\mu$--almost everywhere. Then 
\begin{lemma}\label{lem:fund-doma}
For any $\Gamma$-invariant positive  function, if  $\Delta_0$ and $\Delta_1$ are $\mu$-fundamental domain  then 
$$
\int_{\Delta_0} F\ \d\mu=\int_{\Delta_1} F\ \d\mu\ .
$$	
\end{lemma}
\begin{proof}
Using the $\gamma$-invariance of $F$
\begin{equation*}
		\int_{\Delta_0}F=\sum_{\gamma\in\Gamma}\int_X F\cdot {\boldsymbol 1}_{\Delta_0\cap\gamma(\Delta_1)}\d\mu
	=\sum_{\eta\in\Gamma}\int_X F\cdot {\boldsymbol 1}_{\eta(\Delta_0)\cap\Delta_1}\d\mu
	=	\int_{\Delta_1}F\ .\qedhere
\end{equation*}
\end{proof}
Let $\Gamma$ be a group acting properly on $X_0$ and $X_1$ preserving $\mu_0$ and $\mu_1$ respectively. Assume that $\Delta_0$ -- respectively $\Delta_1$ -- is a fundamental domain for the action of $\Gamma$ on $X_0$ and $X_1$, then.
\begin{lemma}\label{lem:interDELTA}
Let $F$ be a positive function on $X_0\times X_1$ which is $\Gamma$ invariant, where $\Gamma$ acts diagonally and the action on each factor preserves measures called $\mu_0$ and $\mu_1$ and admits a fundamental domain called $\Delta_0$ and $\Delta_1$, then 
$$
\int\!\!\!\int_{\Delta_0\times X_1}F\ \d\mu_0\otimes\d\mu_1=\int\!\!\!\int_{X_0\times \Delta_1}F\ \d\mu_0\otimes\d\mu_1\ .
$$
\end{lemma}
\begin{proof}
Indeed $\Delta_0\times X_1$ and  $X_0\times \Delta_1$ are both fundamental domains for the diagonal action of $\Gamma$ on $X_0\times X_1$. The lemma then follows from the previous one and Fubini's theorem.
\end{proof}
Let $f$ be a continuous function defined on a topological space $X$. Let $\mu$ be a Radon measure on $X$. Then the following lemma holds as a consequence of Lebesgue dominated convergence.
 
 \begin{lemma}\label{lem:L1}
 	Assume that there exists a real constant $k$ so that for every exhausting sequence $\seq{K}$ of compacts of $X$,
 	$
 	\lim_{m\to\infty}\int_{K_m}f\ \d\mu =k 	$.
 	Then $f$ belongs to $L^1(X,\mu)$ and $\int_Xf\ \d\mu=k$.
 \end{lemma}

\section{A lemma in hyperbolic geometry}

\begin{lemma}\label{lem:hyp}
	For any geodesic $g$ and $g_0$, where $g_0$ is parametrized by  arclength, the following holds.
	If $R> 1$ and $d(g_0(R), g)<2$, while $d(g_0(R-1),g)\geq 2$, then
	$$
	d(g_0(0),g)\geq R\ .
	$$
\end{lemma}

\begin{proof}
We let $h$ be a geodesic with $d(g_0(R),h) = d(g_0(R-1),h) = 2$. Then we observe that $d(g_0(0),g) \geq d(g_0(0), h)$. We drop perpendiculars from $g_0(R-1),g_0(R-\frac{1}{2})$ and $g_0(0)$ to $h$.  The perpendicular from $g_0(R-1)$ to $h$ is length $2$ and let $a$ be the length of the perpendicular from $g_0(R-\frac{1}{2})$. For the Lambert quadrilateral with opposite sides of length $a$ and $2$ we have 
\begin{eqnarray*}
	\sinh(a)\cosh\left(\frac{1}{2}\right) = \sinh(2)\ &,&\
	\sinh(a)\cosh\left( R-\frac{1}{2}\right)=\sinh D\ ,\end{eqnarray*}
where $D = d(g_0(0), h)$. It follows easily that 
$$\frac{e^D}{2} \geq \sinh(D) = \sinh(a) \cosh\left(R-\frac{1}{2}\right) \geq  \frac{\sinh (a)}{2} e^{R-1/2}\ .$$
Thus
$$d(g_0(0),g) \geq D \geq R-\frac{1}{2}+\log(\sinh(a)) \geq R\ .\qedhere
$$\end{proof}

\section{The Jacobi identity for the $\Theta$-ghost bracket}\label{app:Jac}
 
We now explain  the Jacobi identity for polygons with disjoint set of vertices.

\subsection{Linking number on a set} Let us recall some constructions  from \cite{Labourie:2012vka}. Let $\PP$ be  a set, $\mathcal G_1$ be the set of pair of points of $\PP$. We denote temporarily the pair $(X,x)$ with the symbol $Xx$. We also define  a {\em linking number} on $\PP$ to be  a map from 
${\PP}^4$ to a commutative ring $\mathbb A$
$
(X,x,Y,y)\to \epsilon(Xx,Yy)
$,
so that for all points $X,x,Y,y,Z,z$ the following conditions are satisfied  
\begin{eqnarray*}
		\epsilon(Xx,Yy)+\epsilon(Xx,yY)=\epsilon(Xx,Yy)+\epsilon(Yy,Xx)&=&0\ ,\\
\epsilon(zy,XY)+\epsilon(zy,YZ)+\epsilon(zy,ZX)&=&0\ ,\\
\epsilon(Xx,Yy).\epsilon(Xy,Yx)&=&0\ .
\end{eqnarray*}
The second author proved in \cite[Proposition 2.1.3]{Labourie:2012vka} the following.
\begin{proposition}[\sc The hexagonal relation]\label{pro:hex}
Let $(X,x,Z,z,Y,y)$ be 6 points on the set ${\PP}$ equipped with an linking number, then
\begin{eqnarray}
	\epsilon(Xy,Zz)+\epsilon(Yx,Zz)=\epsilon(Xx,Zz)+\epsilon(Yy,Zz).
\end{eqnarray}
Moreover, if $
\{X,x\}\cap\{Y,y\}\cap\{Z,z\}=\emptyset
$, 
then
\begin{eqnarray}
	\epsilon(Xx,Yy)\epsilon(Xy,Zz)+\epsilon(Zz,Xx)\epsilon(Zx,Yy)+\epsilon(Yy,Zz)\epsilon(Yz,Xx)&=&0\ ,\\ 
\epsilon(Xx,Yy)\epsilon(Yx,Zz)+\epsilon(Zz,Xx)\epsilon(Xz,Yy)+\epsilon(Yy,Zz)\epsilon(Zy,Xx)&=&0\ .
\end{eqnarray}
\end{proposition}

\subsection{The ghost algebra of a set with a linking number}

\subsubsection{Ghost polygons and edges}
We say a {\em geodesic} is a pair of points in $\PP$. We write $g=(g_-,g_+)$. A {\em configuration} $G \defeq  \lceil g_1,\ldots g_n\rceil$  is a tuple of geodesics $(g_1,\ldots g_n)$ up to cyclic ordering, with $n\geq 1$. The positive integer $n$ is the rank of the configuration.

To a configuration of rank greater than 1, we associate a {\em ghost polygon}, also denoted $G$ which is  a tuple  $G = (\theta_i,\ldots,\theta_{2n})$ where $g_i = \theta_{2i}$ are the {\em visible edges} and $\phi_i = \theta_{2i+1}\defeq ((g_{i+1})_-,(g_i)_+)$ are the ghost edges.

 The {\em ghost index} $i_e$ of an edge $e$ is an element of $\mathbb Z/2\mathbb Z$ which is zero for a visible edge and one for a ghost edge. In other words 
$
i_{\theta_k}\defeq k\ [2]
$.

We will then denote by $G_{\circ}$ the set of edges (ghost or visible) of the configuration $G$. 

Geodesics, or rank 1 configurations, play a special role. In that case $G=\lceil g\rceil$, by convention $G_{\circ}$ consists of of single element $g$ which is a visible edge.  

\subsubsection{Opposite edges} We now define the {\em opposite} of an edge in a configuration. Recall that a configuration is a tuple up to cyclic permutation. In this section we will denote a tuple by
$
\lfloor g_1,\ldots g_n\rfloor
$. We denote by   $\bullet$ the concatenation of tuples:
$$
\lfloor g_1,\ldots g_n\rfloor\bullet\lfloor h_1,\ldots h_p\rfloor\defeq \lfloor g_1,\ldots g_n, h_1,\ldots h_p\rfloor \ .
$$
	
 We introduce the following notation. If $\theta$ is a visible edge of $G$, we define $\theta_+ = \theta_- = \theta$ and if $\theta$ is a ghost edge of $G$ then we define $\theta_+$ to be the visible edge after $\theta$ and $\theta_-$ the visible edge before. The opposite of an edge is 
 $\theta^* \defeq  \lfloor \theta_{+}\ldots\theta_{-} \rfloor  $
 where the ordering is an increasing ordering of visible edges from $\theta_{+}$ to $\theta_{-}$.
 More specifically 
\begin{enumerate}
	\item For a visible edge $g_i$, the opposite is the tuple
$g_i^* = \lfloor g_i, g_{i+1},\ldots g_{i-1},g_i\rfloor$,
\item while for a ghost edge $\phi_i$ the opposite is
$\phi_i^* = \lfloor g_{i+1},g_{i+2},\ldots, g_{i-1}, g_i\rfloor$. 
\item if $\lceil h\rceil$ is a rank 1 configuration. The opposite of its unique edge $h$ is $h$ itself.

\end{enumerate} 

\subsection{Ghost bracket and our main result} 

We now define the {\em ghost algebra} of $\PP$ to be the polynomial algebra  $\mathcal A_0$ freely generated by ghost polygons and geodesics. The ghost algebra is equipped with the antisymmetric  {\em ghost bracket}, given on the generators $\mathcal A$ by, for two ghosts polygons $B$ and $C$ and geodesics $g$ and $h$,\begin{eqnarray}
	[B,C] &=& \sum_{(b,c)\in B_{\circ}\times C_{\circ}}\epsilon(c,b)(-1)^{i_b+i_c} \lceil c^*, b^*\rceil\  .
	\end{eqnarray}
	It is worth writing down the brackets of two geodesics $g$ and $h$, as well as the bracket of a geodesic $g$ and a configuration $B$,
	 \begin{eqnarray}
		-[g,B]=[B,g] &=& \sum_{b\in B_{\circ}}\epsilon(g,b)(-1)^{i_b+1} \lceil g,b^*\rceil\ ,\\
		-[g,h]=[h,g]&=&\epsilon(g,h) \lceil g, h\rceil\ .
\end{eqnarray}

Our goal in this section is to prove

 \begin{theorem}[\sc Jacobi identity]\label{theo:jac}
 Let $A$, $B$, $C$ be three ghost polygons with no common vertices: 
 \begin{equation}
 V_A\cap V_B\cap V_C=\emptyset,  \label{ass:disjoint-support}
 \end{equation}
 where $V_G$ is the set of vertices  of the ghost polygon $G$. Then the ghost bracket satisfies the Jacobi identity for $A$, $B$, $C$: 
$$			[A,[B,C]]+[B,[C,A]]+[C,[A,B]] =0\ .$$
 \end{theorem} 
 
 As the formula for the bracket differs based on whether ghost polygons are rank 1 or higher, we will need to consider the different cases based on the rank of the three elements. We will denote rank 1 elements by $a,b,c$ and higher rank by $A,B,C$.  For $a$, $b$ and $c$ edges in $A$, $B$, $C$ ghost or otherwise, we label their ghost indexes by $i_a$ ,$i_b$, $i_c$ and their opposites by $a^*, b^*, c^*$.

\subsection{Preliminary: more about opposite edges} 
 
We also use the following notation:  if  $\theta_k$ and $\theta_l$ are two edges, ghost or visible   of a ghost polygon,  then
 $$G(\theta_k,\theta_l) = \lfloor {\theta_{k}}_+\ldots {\theta_{l}}_-\rfloor\ ,$$
 where again this is an increasing ordering of visible edges.  The tuple $G(\theta_k,\theta_l)$ is an ``interval" defined by $\theta_k$ and $\theta_l$.
In order to continue our description of the triple brackets, we need to understand, in the above formula, what are the opposite of $\phi^*$ in  $[b^*,c^*]$. Our preliminary result is the following.

\begin{lemma}[\sc Opposite edges in a bracket]\label{lem:cbk} Let $B$ and $C$ be two ghost polygons, $b$ and $c$ edges in $B$ and $C$ respectively.
Let $\phi$ be an edge in $\lceil b^*,c^*\rceil$, then we have the following eight possibilities
\begin{enumerate}
\item[1:] Either $\phi$ is an edge of $B$, different from $b$ or a ghost edge, then $$
\phi^*= G(\phi,b)\bullet c^*\bullet G(b,\phi)\ ,$$ 
\item[2:] $b$ is a visible edge, $\phi$ is the initial edge $b$ in $b^*$ and then
$$
\phi^*=b^*\bullet c^*\bullet b\ .
$$
\item[3:] $b$ is a visible edge, $\phi$ is the final edge $b$ in $b^*$ and then
$$
\phi^*=b\bullet c^*\bullet b^*\ .
$$
\item[4, 5, 6:] Or $\phi$ is an an edge of $C$, and the three items above apply with some obvious symmetry, giving three more possibilities.
\item[7:] or $\phi$ is the  edge $u_{b,c}\defeq (c_-^-,b_+^+)$ of $\lceil b^*,c^*\rceil$ which is neither an edge of $b$ nor an edge of $c$, a ghost edge, and
	$$
	\phi^*=\lfloor c^*,b^*\rfloor \ .
	$$
\item[8:]  $\phi$ is the  edge $u_{c,b}\defeq (b_-^-,c_+^+)$ of $\lceil b^*,c^*\rceil$ which is neither an edge of $b$ nor an edge of $c$,  a ghost edge, and
	$$
	\phi^*=\lfloor b^*,c^*\rfloor \ .
	$$
\end{enumerate}
\end{lemma}
\begin{proof}
	This follows from a careful book-keeping and the previous definitions.  
\end{proof}
\subsection{Cancellations}

Let us introduce the following quantities for any triple of polygons $A$, $B$, $C$ whatever their rank. They will correspond to the cases obtained corresponding to the cases observed in lemma \ref{lem:cbk}:
\begin{eqnarray*}
\hbox{Case 1: }	P_1(A,B,C)&\defeq& \sum_{\substack{(a,c,b,\phi)\in A_{\circ}\times C_{\circ}\times B_{\circ}^2\\ \phi\not=b}} \epsilon(a,\phi)\epsilon(c,b)(-1)^{i_a+i_\phi+i_b+i_c} \lceil a^* \bullet G(\phi,b)\bullet c^*\bullet G(b,\phi) \rceil\ ,\\
\hbox{Case 2: } P_2(A,B,C)&\defeq& \sum_{\substack{(a,b,c,\phi)\in A_{\circ}\times B_{\circ}\times C_{\circ}^2\\  \phi\not=c}} \epsilon(a,\phi)\epsilon(c,b)(-1)^{i_a+i_\phi+i_b+i_c} \lceil a^*\bullet  G(\phi,c)\bullet b^*\bullet G(c,\phi) \rceil\ ,\\
\hbox{Case 3: }Q_1(A,B,C)&\defeq& \sum_{(a,b,c)\in A_{\circ}\times B_{\circ}\times C_{\circ}} \epsilon(a,b)\epsilon(c,b)(-1)^{i_a+i_c} \lceil a^*\bullet b \bullet c^*\bullet b^* \rceil\ ,\\
		\hbox{Case 4: }Q_2(A,B,C)&\defeq& \sum_{(a,b,c)\in A_{\circ}\times B_{\circ}\times C_\circ} \epsilon(a,b)\epsilon(c,b)(-1)^{i_a+i_c} \lceil a^*\bullet b^* \bullet c^*\bullet b\rceil \ ,\\
		\hbox{Case 5: }R_1(A,B,C)&\defeq& \sum_{(a,b,c)\in A_{\circ}\times B_{\circ}\times C_\circ} \epsilon(a,c)\epsilon(c,b)(-1)^{i_a+i_b} \lceil a^*\bullet c \bullet b^*\bullet c^* \rceil\ ,\\
		\hbox{Case 6: }R_2(A,B,C)&\defeq& \sum_{(a,b,c)\in A_{\circ}\times B_{\circ}\times C_{\circ}} \epsilon(a,c)\epsilon(c,b)(-1)^{i_a+i_b} \lceil a^*\bullet c^* \bullet b^*\bullet c \rceil\ ,\\
		\hbox{Case 7: }S_1(A,B,C)&\defeq& \sum_{(a,b,c)\in A_{\circ}\times B_{\circ}\times C_{\circ}} \epsilon(a,u_{b,c})\epsilon(c,b)(-1)^{i_a+i_c+i_b} \lceil a^*\bullet c^* \bullet b^* \rceil\ ,\\
			\hbox{Case 8: }S_2(A,B,C)&\defeq& \sum_{(a,b,c)\in A_{\circ}\times B_{\circ}\times C_{\circ}} \epsilon(a,u_{c,b})\epsilon(c,b)(-1)^{i_a+i_c+i_b} \lceil a^*\bullet b^* \bullet c^* \rceil\ .				\end{eqnarray*}

			We then have, 
\begin{lemma}[\sc Cancellations]
				We have the following cancellations, where the two last ones use the hypothesis \eqref{ass:disjoint-support}
\begin{eqnarray}
	P_1(A,B,C)+P_2(C,A,B)&=&0\ ,\  \hbox{\em first cancellation}\ ,\crcr
	R_1(A,B,C)+Q_2(B,C,A)&=&0\ ,\  \hbox{\em second cancellation-1}\ ,\crcr
	R_2(A,B,C)+Q_1(B,C,A)&=&0\ ,\  \hbox{\em second cancellation-2}\ ,\crcr
	S_1(A,B,C)+S_1(B,C,A)+S_1(C,A,B)&=&0\  ,\  \hbox{\em hexagonal cancellation-1}\ ,\cr
		S_2(A,B,C)+S_2(B,C,A)+S_2(C,A,B)&=&0\  ,\  \hbox{\em hexagonal cancellation-2}\ .\label{eq:cancel}
\end{eqnarray}
	\end{lemma}
\begin{proof} For the first cancellation, we have 
\begin{eqnarray*}
		& &P_1(A,B,C)+P_2(C,A,B)\\
	&=&\sum_{\substack{(a,c,b,\phi)\in A_{\circ}\times C_{\circ}\times B_{\circ}^2\\ \phi\not=b}}\epsilon(a,\phi)\epsilon(c,b)(-1)^{i_a+i_\phi+i_b+i_c} \lceil a^* \bullet G(\phi,b)\bullet c^*\bullet G(b,\phi) \rceil\\
	&+&\sum_{\substack{(c,a,b,\phi)\in  C_{\circ}\times A_{\circ}\times B_{\circ}^2\\ \phi\not=b}}\epsilon(c,\phi)\epsilon(b,a) (-1)^{i_a+i_\phi+i_b+i_c} \lceil c^* \bullet G(\phi,b)\bullet a^*\bullet G(b,\phi) \rceil\\
	&=&\sum_{\substack{(a,c)\in A_{\circ}\times C_\circ \\ (b_0,b_1)\in B_{\circ}^2\\ b_0\not=b_1}}\left(\epsilon(a,b_1)\epsilon(c,b_0) + \epsilon(c,b_0)\epsilon(b_1,a) \right)(-1)^{i_a+i_{b_0}+i_{b_1}+i_c} \lceil a^* \bullet G(\phi,b)\bullet c^*\bullet G(b,\phi) \rceil=0\ ,\end{eqnarray*}
where we used  the change of variables $(b_0,b_1)=(b,\phi)$ in the second line and $(b_0,b_1)=(\phi,b)$ in the third and used the cyclic invariance.

The second cancellation-1 follows by a similar argument
\begin{eqnarray*}
		R_1(A,B,C)+Q_2(B,C,A)
	&=&\sum_{(a,b,c)\in A_{\circ}\times B_{\circ}\times C_{\circ}} \epsilon(a,c)\epsilon(c,b)(-1)^{i_a+i_b} \lceil b^*\bullet c^* \bullet a^*\bullet c) \rceil\\\
	&+&\sum_{(a,b,c)\in A_{\circ}\times B_{\circ}\times C_{\circ}} \epsilon(b,c)\epsilon(a,c)(-1)^{i_a+i_b} \lceil b^*\bullet c^* \bullet a^*\bullet c) \rceil\ =0\ .
\end{eqnarray*}
Similarly for the second cancellation-2. Finally the hexagonal cancellation-1 follows from the hexagonal relation
\begin{eqnarray*}
\epsilon(a,u_{b,c})\epsilon(c,b)+\epsilon(b,u_{c,a})\epsilon(a,c)+\epsilon(c,u_{a,b})\epsilon(b,a)=0	\ ,
\end{eqnarray*}
which is itself a consequence of lemma \ref{pro:hex} and the assumption \eqref{ass:disjoint-support}. A  similar argument works the second hexagonal relation. \end{proof}

\subsection{The various possibilities for the triple bracket}
We have to consider 3 different possibilities for the triple brackets $[A,[B,C]]$ taking in account whether $B$ and $C$ have rank 1.

The following lemma will be a consequence of  lemma \ref{lem:cbk}. We will also use the following conventions: 
\begin{eqnarray*}
	& &\hbox{ if } Q_1(U,V,W)=Q_2(U,V,W)\hbox{ , then we write } Q(U,V,W)\defeq Q_1(U,V,W)=Q_2(U,V,W)\ ,\\
	& &\hbox{ if } R_1(U,V,W)=R_2(U,V,W)\hbox{ , then we write } R(U,V,W)\defeq R_1(U,V,W)=R_2(U,V,W)\ .
\end{eqnarray*}
\begin{lemma}[\sc Triple bracket]\label{lem:trip}
	We have the following four possibilities (independent of the rank of $U$) for the triple brackets
\begin{enumerate}
	\item The polygons $V$ and $W$ have both rank greater than 1, then \begin{eqnarray}
[U,[V,W]] &=&P_1(U,V,W)+P_2(U,V,W)+Q_1(U,V,W)+Q_2(U,V,W)\nonumber\\
&+&R_1(U,V,W)+R_2(U,V,W)+ S_1(U,V,W)+S_2(U,V,W)\ .\label{eq:22}
\end{eqnarray}
\item Both $v\defeq V$ and $w\defeq W$ have rank 1, then  \begin{eqnarray}
	[U,[v,w]]&=&Q(U,v,w)+R(U,v,w)+S_1(U,v,w)+S_2(U,v,w)\ .\label{eq:21}
	\end{eqnarray}
\item The polygon  $W$ has  rank greater than  1, while $v\defeq V$ has rank 1, then 
\begin{eqnarray*}
	[U,[v,W]]=P_2(U,v,W)+Q(U,v,W)+R_1(U,v,W)+R_2(U,v,W)+S_1(U,v,W)+S_2(U,v,W) \ .
	\end{eqnarray*}
	\item The polygon  $W$ has  rank greater than  1, while $v\defeq V$ has rank 1, then 
\begin{eqnarray*}
	[U,[V,w]]&=&P_2(U,W,v)+R(U,W,v)+Q_1(U,W,v)+R_2(U,W,v)\\ &+& S_1(U,W,v)+S_2(U,W,v) \ .
	\end{eqnarray*}
\end{enumerate}

\end{lemma} 
\begin{proof} This is deduced from lemma \ref{lem:cbk}. Indeed we deduce from that lemma that  we have 
\begin{enumerate}
	\item if $V$ is a geodesic, then case 1 does not happen, and case 3 and case 4 coincide, thus
	$$
	P_1(U,V,W)=0\ , \ Q_1(U,V,W)=Q_2(U,V,W)\eqdef Q(U,V,W)\ .	$$
	\item Symmetrically, if $W$ is a geodesic, then case 2 does not happen, and case 5 and case 6 coincide, thus
	$$
	P_2(U,V,W)=0\ , \ R_1(U,V,W)=R_2(U,V,W)\eqdef R(U,V,W)\ .	\qedhere$$
\end{enumerate}
\end{proof}
\subsection{Proof of the Jacobi identity}
We will use freely in that paragraph lemma \ref{lem:trip}
\begin{proof}[Proof 1:  all three ghost polygons have rank greater than 1]
The previous discussion gives 
\begin{eqnarray*}
[A,[B,C]]&=&	P_1(A,B,C)+P_2(A,B,C)+Q_1(A,B,C)+Q_2(A,B,C)\\
&+&R_1(A,B,C)+R_2(A,B,C)+ S_1(A,B,C)+S_2(A,B,C)\ ,\\
\lbrack B,[C,A]]&=&P_1(B,C,A)+P_2(B,C,A)+Q_1(B,C,A)+Q_2(B,C,A)\\
&+&R_1(B,C,A)+R_2(B,C,A)+ S_1(B,C,A)+S_2(B,C,A)\ ,\\
\lbrack C,[A,B]]&=&P_1(C,A,B)+P_2(C,A,B)+Q_1(C,A,B)+Q_2(C,A,B)\\
&+&R_1(C,A,B)+R_2(C,A,B)+ S_1(C,A,B)+S_2(C,A,B)\ .
\end{eqnarray*}
The proof of the Jacobi identity then follows from  the  cancellations \eqref{eq:cancel}.
\end{proof}

\begin{proof}[Proof  2: all three ghost polygons have rank  1], then   writing $a\defeq A$, $b\defeq B$ and $c\defeq C$, we have 
\begin{eqnarray*}
	[a,[b,c]]&=&Q(a,b,c)+R(a,b,c)+S_1(a,b,c)+S_2(a,b,c)\ ,\\
	\lbrack b,[c,a]]&=&Q(b,c,a)+R(b,c,a)+S_1(b,c,a)+S_2(b,c,a)  \\
		\lbrack c,[b,a]]&=&Q(c,a,b)+R(c,a,b)+S_1(c,a,b)+S_2(c,a,b)\ . 
	\end{eqnarray*}
The Jacobi identity follows from the  cancellations \eqref{eq:cancel}.
\end{proof}

\begin{proof}[Proof 3: exactly one of the three polygons has  rank 1] 
Assume $a\defeq A$ is a geodesic, $B$ and $C$ has rank greater than  1. Then 
\begin{eqnarray*}
[a,[B,C]]&=&P_1(a,B,C)+P_2(a,B,C)+Q_1(a,B,C)+Q_2(a,B,C)\\
&+&R_1(a,B,C)+R_2(a,B,C)+ S_1(a,B,C)+S_2(a,B,C)\ ,\\
\lbrack C,[a,B]]&=&P_2(C,a,B)+Q_1(C,a,B)+Q_2(C,a,B)
+R(C,a,B)+ S_1(C,a,B)+S_2(C,a,B)\ ,\\
\lbrack B,[C,a]] &=&P_1(B,C,a)+R_1(B,C,a)+R_2(B,C,a)
+Q(B,C,a)+ S_1(B,C,a)+S_2(B,C,a)\ .
\end{eqnarray*} 
Then again the cancellations \eqref{eq:cancel}, yields the Jacobi identity in that case.
\end{proof}

\begin{proof}[Proof of the final possibility : exactly two of the three polygons have  rank 1] 

We have here that $A$ has rank greater than 1, while $b\defeq B$ and $c\defeq C$ are geodesics, then
\begin{eqnarray*}
	[A,[b,c]]&=&Q(A,b,c)+R(A,b,c)+S_1(A,b,c)+S_2(A,b,c)\ 
	,\\
	\lbrack b,[c,A]]&=&P_1(b,c,A)+Q_1(b,c,A)+Q_2(b,c,A)
+R(b,c,A)+ S_1(b,c,A)+S_2(b,c,A)\ ,\\
	\lbrack c,[A,b]]&=&P_2(c,A,b)+R_1(c,A,b)+R_2(c,A,b)
+Q(c,A,b)+ S_1(c,A,b)+S_2(c,A,b)\ .
\end{eqnarray*} 
For the last time,  the cancellations \eqref{eq:cancel}, yields the Jacobi identity in that case.	
\end{proof}

\end{appendix}

\bibliographystyle{amsplain}
\providecommand{\bysame}{\leavevmode\hbox to3em{\hrulefill}\thinspace}
\providecommand{\MR}{\relax\ifhmode\unskip\space\fi MR }
\providecommand{\MRhref}[2]{%
  \href{http://www.ams.org/mathscinet-getitem?mr=#1}{#2}
}
\providecommand{\href}[2]{#2}

\end{document}